\documentclass[12pt]{amsart}
\usepackage{amsmath}
\usepackage{amsxtra}
\usepackage{amscd}
\usepackage{amsthm}
\usepackage{amsfonts}
\usepackage{amssymb}
\usepackage{eucal}
\usepackage{verbatim}
\usepackage[all]{xy}
\usepackage{color}
\definecolor{darkspringgreen}{rgb}{0.09, 0.45, 0.27}

\usepackage{stmaryrd}

\usepackage[pdftex,bookmarks=false,colorlinks=true,citecolor=darkspringgreen,debug=true,
  naturalnames=true,pdfnewwindow=true]{hyperref}

\newtheorem{thm}{Theorem}[subsection]

\newtheorem{cor}[thm]{Corollary}

\newtheorem{lem}[thm]{Lemma}
\newtheorem{prop}[thm]{Proposition}
\newtheorem{conj}[thm]{Conjecture}

\theoremstyle{definition}

\theoremstyle{remark}
\newtheorem{rem}[thm]{Remark}
\newtheorem{rems}[thm]{Remarks} 

\newcommand{\lemref}[1]{Lemma~\ref{#1}}

\newcommand{\nc}{\newcommand}
\nc{\renc}{\renewcommand} \nc{\ssec}{\subsection}
\nc{\sssec}{\subsubsection}

\nc{\on}{\operatorname} \nc{\wh}{\widehat}
\nc\ol{\overline} \nc\ul{\underline} \nc\wt{\widetilde}

\newcommand{\red}[1]{{\color{red}#1}}

\nc{\BA}{{\mathbb{A}}} \nc{\BC}{{\mathbb{C}}} \nc{\BF}{{\mathbb{F}}}
\nc{\BD}{{\mathbb{D}}} \nc{\BG}{{\mathbb{G}}} \nc{\BQ}{{\mathbb{Q}}}
\nc{\BM}{{\mathbb{M}}} \nc{\BN}{{\mathbb{N}}} \nc{\BO}{{\mathbb{O}}}
\nc{\BP}{{\mathbb{P}}} \nc{\BR}{{\mathbb{R}}}
\nc{\BZ}{{\mathbb{Z}}} \nc{\BS}{{\mathbb{S}}} \nc{\BW}{{\mathbb{W}}}

\nc{\CA}{{\mathcal{A}}} \nc{\CB}{{\mathcal{B}}} \nc{\CalD}{{\mathcal{D}}}
\nc{\CE}{{\mathcal{E}}} \nc{\CF}{{\mathcal{F}}}
\nc{\CG}{{\mathcal{G}}} \nc{\CH}{{\mathcal{H}}}
\nc{\CI}{{\mathcal{I}}} \nc{\CK}{{\mathcal{K}}} \nc{\CL}{{\mathcal{L}}}
\nc{\CM}{{\mathcal{M}}} \nc{\CN}{{\mathcal{N}}}
\nc{\CO}{{\mathcal{O}}} \nc{\CP}{{\mathcal{P}}}
\nc{\CQ}{{\mathcal{Q}}} \nc{\CR}{{\mathcal{R}}}
\nc{\CS}{{\mathcal{S}}} \nc{\CT}{{\mathcal{T}}}
\nc{\CU}{{\mathcal{U}}} \nc{\CV}{{\mathcal{V}}}  \nc{\CY}{{\mathcal Y}}
\nc{\CW}{{\mathcal{W}}} \nc{\CZ}{{\mathcal{Z}}}

\nc{\cM}{{\check{\mathcal M}}{}} \nc{\csM}{{\check{\mathcal A}}{}}
\nc{\oM}{{\overset{\circ}{\mathcal M}}{}}
\nc{\obM}{{\overset{\circ}{\mathbf M}}{}}
\nc{\oCA}{{\overset{\circ}{\mathcal A}}{}}
\nc{\obA}{{\overset{\circ}{\mathbf A}}{}}
\nc{\ooM}{{\overset{\circ}{M}}{}}
\nc{\osM}{{\overset{\circ}{\mathsf M}}{}}
\nc{\vM}{{\overset{\bullet}{\mathcal M}}{}}
\nc{\nM}{{\underset{\bullet}{\mathcal M}}{}}
\nc{\oD}{{\overset{\circ}{\mathcal D}}{}}
\nc{\obD}{{\overset{\circ}{\mathbf D}}{}}
\nc{\oA}{{\overset{\circ}{\mathbb A}}{}}
\nc{\op}{{\overset{\bullet}{\mathbf p}}{}}
\nc{\cp}{{\overset{\circ}{\mathbf p}}{}}
\nc{\oU}{{\overset{\bullet}{\mathcal U}}{}}
\nc{\ofZ}{{\overset{\circ}{\mathfrak Z}}{}}

\nc{\ff}{{\mathfrak{f}}} \nc{\fv}{{\mathfrak{v}}}
\nc{\fa}{{\mathfrak{a}}} \nc{\fb}{{\mathfrak{b}}}
\nc{\fd}{{\mathfrak{d}}} \nc{\fe}{{\mathfrak{e}}}
\nc{\fg}{{\mathfrak{g}}} \nc{\fgl}{{\mathfrak{gl}}}
\nc{\fh}{{\mathfrak{h}}} \nc{\fri}{{\mathfrak{i}}}
\nc{\fj}{{\mathfrak{j}}} \nc{\fk}{{\mathfrak{k}}} \nc{\fl}{{\mathfrak{l}}}
\nc{\fm}{{\mathfrak{m}}} \nc{\fn}{{\mathfrak{n}}}
\nc{\ft}{{\mathfrak{t}}} \nc{\fu}{{\mathfrak{u}}}
\nc{\fw}{{\mathfrak{w}}} \nc{\fz}{{\mathfrak{z}}}
\nc{\fp}{{\mathfrak{p}}} \nc{\fq}{{\mathfrak{q}}} \nc{\frr}{{\mathfrak{r}}}
\nc{\fs}{{\mathfrak{s}}} \nc{\fsl}{{\mathfrak{sl}}}
\nc{\hsl}{{\widehat{\mathfrak{sl}}}}
\nc{\hgl}{{\widehat{\mathfrak{gl}}}}
\nc{\hg}{{\widehat{\mathfrak{g}}}}
\nc{\chg}{{\widehat{\mathfrak{g}}}{}^\vee}
\nc{\hn}{{\widehat{\mathfrak{n}}}}
\nc{\chn}{{\widehat{\mathfrak{n}}}{}^\vee}

\nc{\fA}{{\mathfrak{A}}} \nc{\fB}{{\mathfrak{B}}} \nc{\fC}{{\mathfrak{C}}}
\nc{\fD}{{\mathfrak{D}}} \nc{\fE}{{\mathfrak{E}}}
\nc{\fF}{{\mathfrak{F}}} \nc{\fG}{{\mathfrak{G}}} \nc{\fH}{{\mathfrak{H}}}
\nc{\fI}{{\mathfrak{I}}} \nc{\fJ}{{\mathfrak{J}}}
\nc{\fK}{{\mathfrak{K}}} \nc{\fL}{{\mathfrak{L}}}
\nc{\fM}{{\mathfrak{M}}} \nc{\fN}{{\mathfrak{N}}}
\nc{\frP}{{\mathfrak{P}}} \nc{\fQ}{{\mathfrak{Q}}}
\nc{\fS}{{\mathfrak{S}}} \nc{\fT}{{\mathfrak{T}}} \nc{\fU}{{\mathfrak{U}}}
\nc{\fV}{{\mathfrak{V}}} \nc{\fW}{{\mathfrak{W}}}
\nc{\fX}{{\mathfrak{X}}} \nc{\fY}{{\mathfrak{Y}}}
\nc{\fZ}{{\mathfrak{Z}}}

\nc{\ba}{{\mathbf{a}}}
\nc{\bb}{{\mathbf{b}}} \nc{\bc}{{\mathbf{c}}}
\nc{\be}{{\mathbf{e}}} \nc{\bj}{{\mathbf{j}}} \nc{\bm}{{\mathbf{m}}}
\nc{\bn}{{\mathbf{n}}} \nc{\bp}{{\mathbf{p}}}
\nc{\bq}{{\mathbf{q}}} \nc{\br}{{\mathbf{r}}} \nc{\bt}{{\mathbf{t}}}
\nc{\bfu}{{\mathbf{u}}} \nc{\bv}{{\mathbf{v}}}
\nc{\bx}{{\mathbf{x}}} \nc{\by}{{\mathbf{y}}} \nc{\bz}{{\mathbf{z}}}
\nc{\bw}{{\mathbf{w}}} \nc{\bA}{{\mathbf{A}}}
\nc{\bB}{{\mathbf{B}}} \nc{\bC}{{\mathbf{C}}}
\nc{\bD}{{\mathbf{D}}} \nc{\bF}{{\mathbf{F}}} \nc{\bG}{{\mathbf{G}}}
\nc{\bH}{{\mathbf{H}}} \nc{\bI}{{\mathbf{I}}} \nc{\bJ}{{\mathbf{J}}}
\nc{\bK}{{\mathbf{K}}} \nc{\bM}{{\mathbf{M}}} \nc{\bN}{{\mathbf{N}}}
\nc{\bO}{{\mathbf{O}}} \nc{\bS}{{\mathbf{S}}} \nc{\bT}{{\mathbf{T}}}
\nc{\bU}{{\mathbf{U}}} \nc{\bV}{{\mathbf{V}}} \nc{\bW}{{\mathbf{W}}}
\nc{\bX}{{\mathbf{X}}}
\nc{\bY}{{\mathbf{Y}}} \nc{\bP}{{\mathbf{P}}}
\nc{\bZ}{{\mathbf{Z}}} \nc{\bh}{{\mathbf{h}}}

\nc{\sA}{{\mathsf{A}}} \nc{\sB}{{\mathsf{B}}}
\nc{\sC}{{\mathsf{C}}} \nc{\sD}{{\mathsf{D}}}
\nc{\sE}{{\mathsf{E}}} \nc{\sF}{{\mathsf{F}}} \nc{\sG}{{\mathsf{G}}}
\nc{\sI}{{\mathsf{I}}} \nc{\sK}{{\mathsf{K}}} \nc{\sL}{{\mathsf{L}}}
\nc{\sfm}{{\mathsf{m}}} \nc{\sM}{{\mathsf{M}}} \nc{\sO}{{\mathsf{O}}}
\nc{\sQ}{{\mathsf{Q}}} \nc{\sP}{{\mathsf{P}}}
\nc{\sT}{{\mathsf{T}}} \nc{\sZ}{{\mathsf{Z}}}
\nc{\sV}{{\mathsf{V}}} \nc{\sW}{{\mathsf{W}}}
\nc{\sfp}{{\mathsf{p}}} \nc{\sq}{{\mathsf{q}}} \nc{\sr}{{\mathsf{r}}}
\nc{\st}{{\mathsf{t}}} \nc{\sfb}{{\mathsf{b}}}
\nc{\sfc}{{\mathsf{c}}} \nc{\sd}{{\mathsf{d}}}
\nc{\sz}{{\mathsf{z}}}

\nc{\tA}{{\widetilde{\mathbf{A}}}}
\nc{\tB}{{\widetilde{\mathcal{B}}}}
\nc{\tg}{{\widetilde{\mathfrak{g}}}} \nc{\tG}{{\widetilde{G}}}
\nc{\TM}{{\widetilde{\mathbb{M}}}{}}
\nc{\tO}{{\widetilde{\mathsf{O}}}{}}
\nc{\tU}{{\widetilde{\mathfrak{U}}}{}} \nc{\TZ}{{\tilde{Z}}}
\nc{\tx}{{\tilde{x}}} \nc{\tbv}{{\tilde{\bv}}}
\nc{\tfP}{{\widetilde{\mathfrak{P}}}{}} \nc{\tz}{{\tilde{\zeta}}}
\nc{\tmu}{{\tilde{\mu}}}

\nc{\urho}{\underline{\rho}} \nc{\uB}{\underline{B}}
\nc{\uC}{{\underline{\mathbb{C}}}} \nc{\ui}{\underline{i}}
\nc{\uj}{\underline{j}} \nc{\ofP}{{\overline{\mathfrak{P}}}}
\nc{\oB}{{\overline{\mathcal{B}}}}
\nc{\og}{{\overline{\mathfrak{g}}}} \nc{\oI}{{\overline{I}}}

\nc{\eps}{\varepsilon} \nc{\hrho}{{\hat{\rho}}}
\nc{\blambda}{{\boldsymbol{\lambda}}} \nc{\bmu}{{\boldsymbol{\mu}}} \nc{\bnu}{{\boldsymbol{\nu}}}

\nc{\one}{{\mathbf{1}}} \nc{\two}{{\mathbf{t}}}

\nc{\Sym}{\mathop{\operatorname{\rm Sym}}}
\nc{\Tot}{{\mathop{\operatorname{\rm Tot}}}}
\nc{\Spec}{\mathop{\operatorname{\rm Spec}}}
\nc{\Ker}{{\mathop{\operatorname{\rm Ker}}}}
\nc{\Isom}{{\mathop{\operatorname{\rm Isom}}}}
\nc{\Hilb}{{\mathop{\operatorname{\rm Hilb}}}}
\nc{\deeq}{{\mathop{\operatorname{\rm deeq}}}}
\nc{\End}{{\mathop{\operatorname{\rm End}}}}
\nc{\Ext}{{\mathop{\operatorname{\rm Ext}}}}
\nc{\Hom}{{\mathop{\operatorname{\rm Hom}}}}
\nc{\CHom}{{\mathop{\operatorname{{\mathcal{H}}\it om}}}}
\nc{\GL}{{\mathop{\operatorname{\rm GL}}}}
\nc{\gr}{{\mathop{\operatorname{\rm gr}}}}
\nc{\Id}{{\mathop{\operatorname{\rm Id}}}}
\nc{\perf}{{\mathop{\operatorname{\rm perf}}}}
\nc{\defi}{{\mathop{\operatorname{\rm def}}}}
\nc{\length}{{\mathop{\operatorname{\rm length}}}}
\nc{\supp}{{\mathop{\operatorname{\rm supp}}}}
\nc{\HC}{{\mathcal H}{\mathcal C}}

\nc{\pr}{{\operatorname{pr}}}
\nc{\Cliff}{{\mathsf{Cliff}}}
\nc{\loc}{{\operatorname{loc}}}
\nc{\Fl}{{\mathbf{Fl}}} \nc{\Ffl}{{\mathcal{F}\ell}}
\nc{\Fib}{{\mathsf{Fib}}}
\nc{\Coh}{{\mathsf{Coh}}} \nc{\FCoh}{{\mathsf{FCoh}}}
\nc{\Perf}{{\mathsf{Perf}}}

\nc{\reg}{{\text{\rm reg}}}
\nc{\gvee}{{\mathfrak g}^{\!\scriptscriptstyle\vee}}
\nc{\tvee}{{\mathfrak t}^{\!\scriptscriptstyle\vee}}
\nc{\nvee}{{\mathfrak n}^{\!\scriptscriptstyle\vee}}
\nc{\bvee}{{\mathfrak b}^{\!\scriptscriptstyle\vee}}
       \nc{\rhovee}{\rho^{\!\scriptscriptstyle\vee}}

\nc{\cplus}{{\mathbf{C}_+}} \nc{\cminus}{{\mathbf{C}_-}}
\nc{\cthree}{{\mathbf{C}_*}} \nc{\Qbar}{{\bar{Q}}}

\newcommand{\oV}{{\mathbf V}}
\newcommand{\ooV}{{\mathbf V}}
\newcommand\iso{\mathbin{\vphantom{j^{X^2}}\smash{\overset{\sim}{\vphantom{\rule{0pt}{0.20em}}\smash{\longrightarrow}}}}}
\nc{\Gtimes}{\vphantom{j^{X^2}}\smash{\overset{G}{\vphantom{\rule{0pt}{0.30em}}\smash{\times}}}}
\nc{\sGtimes}{\vphantom{j^{X^2}}\smash{\overset{\mathsf G}{\vphantom{\rule{0pt}{0.30em}}\smash{\times}}}}

\nc{\bOmega}{{\overline{\Omega}}}

\nc{\seq}[1]{\stackrel{#1}{\sim}}

\nc{\aff}{{\operatorname{aff}}}
\nc{\fin}{{\operatorname{fin}}}
\nc{\mir}{{\operatorname{mir}}}
\nc{\triv}{{\operatorname{triv}}}
\nc{\ext}{{\operatorname{ext}}}
\nc{\righ}{{\operatorname{right}}}
\nc{\lef}{{\operatorname{left}}}
\nc{\forg}{{\operatorname{forg}}}
\nc{\fid}{{\operatorname{fd}}}
\nc{\modu}{{\operatorname{-mod}}}
\nc{\Gr}{{\mathbf{Gr}}}
\nc{\FT}{{\operatorname{FT}}}
\nc{\Mat}{{\operatorname{Mat}}}
\nc{\MSt}{{\operatorname{MSt}}}
\nc{\sph}{{\operatorname{sph}}}
\nc{\GR}{{\mathbf{Gr}}}
\nc{\Perv}{{\operatorname{Perv}}}
\nc{\Rep}{{\operatorname{Rep}}}
\nc{\Ind}{{\operatorname{Ind}}}
\nc{\IC}{{\operatorname{IC}}}
\nc{\Bun}{{\operatorname{Bun}}}
\nc{\Proj}{{\operatorname{Proj}}}
\nc{\Stab}{{\operatorname{Stab}}}
\nc{\pt}{{\operatorname{pt}}}
\nc{\bfmu}{{\boldsymbol{\mu}}}
\nc{\bfomega}{{\boldsymbol{\omega}}}
\nc{\calM}{\mathcal M}
\nc{\calA}{\mathcal A}
\nc{\calO}{\mathcal O}
\nc{\CC}{\mathbb C}
\nc{\calN}{\mathcal N}
\nc{\grg}{\mathfrak g}
\nc{\tslash}{/\!\!/\!\!/}
\nc\grt{\mathfrak t}
\nc\bfM{\mathbf M}
\nc\bfN{\mathbf N}
\nc\Sig{\Sigma}
\nc\ZZ{\mathbb{Z}}
\nc\calC{\mathcal C}
\nc\calF{\mathcal F}
\nc\calX{\mathcal X}
\nc\calY{\mathcal Y}
\nc\QCoh{\operatorname{QCoh}}
\nc\IndCoh{\operatorname{IndCoh}}
\nc\Maps{\operatorname{Maps}}
\nc\Dmod{D-\operatorname{mod}}
\newcommand\Hecke{\operatorname{Hecke}}
\nc{\calD}{\mathcal D}
\nc\bfO{\mathbf O}
\renewcommand{\AA}{\mathbb A}
\nc\GG{\mathbb G}
\nc\calK{\mathcal K}
\nc{\calG}{\mathcal G}
\nc\RHom{\operatorname{RHom}}
\nc\Res{\operatorname{Res}}
\nc\Av{\operatorname{Av}}
\nc\grs{\mathfrak s}
\nc{\tilX}{\widetilde X}
\nc\calB{\mathcal B}
\nc\calS{\mathcal S}
\nc\calT{\mathcal T}
\nc\calZ{\mathcal Z}
\nc\LS{\operatorname{LocSys}}
\nc\bfL{\on{\mathbf L}}

\newcommand*\circled[1]
{*{#1}}
\makeatletter
\newcommand{\raisemath}[1]{\mathpalette{\raisem@th{#1}}}
\newcommand{\raisem@th}[3]{\raisebox{#1}{$#2#3$}}
\nc{\binlim}[2][]{\def\@tempa{#1}\@ifnextchar^{\@binlim{#2}}{\@binlim{#2}^{}}}
\def\@binlim#1^#2{\mathbin{\@ifempty{#2}{\mathop{#1}}{\mathop{#1}\@xp\displaylimits\@tempa^{#2}}}}

\makeatother

\nc\cX{{\mathcal X}}

\newcommand{\dbkts}[1]{[\![#1]\!]}
\newcommand{\dprts}[1]{(\!(#1)\!)}
\nc\Gm{{\mathbb G_m}}
\nc{\isoto}
    {\stackrel{\displaystyle\raisemath{-.2ex}{\sim}}\to}
\nc\cD\CalD
\nc\setm\setminus
\nc\cE\CE
\renc\Hecke{\mathit{\CH\kern-.2ex ecke}}
\nc\bsl\backslash
\nc\Fq{\mathbb F_q}
\nc\x\times
\nc\bGO{{\bG_\bO}}
\nc\bla\blambda
\nc\blt\bullet
\nc\opp{{\on{op}}}
\nc\srel\stackrel
\nc\oast\circledast
\nc\tbx{\binlim{\widetilde\boxtimes{}}}
\nc\into\hookrightarrow
\nc\otto\leftrightarrow
\renc\setminus\smallsetminus
\nc\phitau{\varphi\tau}
\newenvironment{i-ii-iii}{%
\begin{enumerate}
}%
{\end{enumerate}}
\nc\ceil[1]{\lceil#1\rceil}  \nc\floor[1]{\lfloor#1\rfloor}
\nc\Lie{\on{Lie}}
\def\arxiv#1{\href{http://arxiv.org/abs/#1}{\tt arXiv:#1}} \let\arXiv\arxiv
\nc\kap{\kappa}
\nc\gra{\mathfrak a}
\nc\gl{\mathfrak{gl}}
\nc\sTr{\operatorname{sTr}}
\nc\hatG{\widehat{G}}
\nc\calL{\mathcal L}
\nc\Whit{\operatorname{Whit}}
\nc\KL{\operatorname{KL}}
\newcommand\Tr{\operatorname{Tr}}
\newcommand\crit{\operatorname{crit}}

\newcommand\Killing{\operatorname{Killing}}
\newcommand\PP{\mathbb P}
\makeatletter
\renewcommand{\subsection}{\@startsection{subsection}{2}{0pt}{-3ex
plus -1ex minus -0.2ex}{-2mm plus -0pt minus
-2pt}{\normalfont\bfseries}} \makeatother

\numberwithin{equation}{subsection}
\allowdisplaybreaks
\nc\mto{\mapsto }
\nc\en{\enspace }

\begin{document}
\author[A.Braverman]{Alexander Braverman}
\address{Department of Mathematics, University of Toronto and Perimeter Institute
of Theoretical Physics, Waterloo, Ontario, Canada, N2L 2Y5;
\newline Skolkovo Institute of Science and Technology}
\email{braval@math.toronto.edu}
\author[M.Finkelberg]{Michael Finkelberg}
\address{National Research University Higher School of Economics, Russian Federation,
  Department of Mathematics, 6 Usacheva st, 119048 Moscow;
\newline Skolkovo Institute of Science and Technology;
\newline Institute for the Information Transmission Problems}
\email{fnklberg@gmail.com}
\author[V.Ginzburg]{Victor Ginzburg}
\address{Department of Mathematics, University of Chicago, Chicago, USA}
\email{vityaginzburg@gmail.com}
\author[R.Travkin]{Roman Travkin}
\address{Skolkovo Institute of Science and Technology, Moscow, Russia}
\email{roman.travkin2012@gmail.com}

\title
{Mirabolic Satake equivalence and supergroups}
\dedicatory{To our friend Sasha Shen on the occasion of his 60th birthday}



\begin{abstract}
 We construct a mirabolic analogue of the geometric Satake equivalence.
We also prove an equivalence that relates representations of a supergroup
with the category of
$\GL(N-1,{\mathbb C}[\![t]\!])$-equivariant perverse sheaves on the affine Grassmannian of $\GL_N$.
We explain how our equivalences fit into a more general framework of conjectures due to Gaiotto
and to Ben-Zvi, Sakellaridis and Venkatesh.
\end{abstract}
\maketitle

\tableofcontents

\section{Introduction}
\label{intro}

\subsection{Reminder on  Geometric Satake} Let $\bF=\BC\dprts{t}\supset\BC\dbkts{t}=\bO$.
Throughout the paper, we fix an integer $N\geq 1$, set
$\bG_\bF=\GL(N,\bF)$, resp.\ $\bG_\bO=\GL(N,\bO)$, and let $\Gr=\bG_\bF/\bG_\bO$
be  the  affine Grassmannian of $\GL_N$. This is an ind-scheme equipped with
a natural action of the group $\bG_\bO\rtimes\BC^\times$, where $\BC^\times$ acts by loop rotation.
Let $D_{\bG_\bO}(\Gr)$, resp.\ $D_{\bG_\bO\rtimes\BC^\times}(\Gr)$, be the $\bG_\bO$-equivariant,
resp.\ $\bG_\bO\rtimes\BC^\times$-equivariant,
constructible derived category of $\Gr$.\footnote{Throughout the paper, we consider the sheaves with
complex coefficients (with a few technical exceptions).}
This is a monoidal category with respect to convolution (that coincides with fusion), cf.~\cite{MV}.

Let $\fgl_N$ be the complex vector space of $N\times N$-matrices
and let  $\GL_N$ act on $\fgl_N$ by conjugation.
Write $\Sym^\bullet(\fgl_N[-2])$ for the symmetric algebra of $\fgl_N$
viewed  as a  dg-algebra such that the space $\fgl_N$, of generators, is
placed in degree $2$ and the differential is equal to zero.
Let $D_\perf^{\GL_N}(\Sym^\bullet(\fgl_N[-2]))$
be  the triangulated category
of perfect $\GL_N$-equivariant dg-modules over $\Sym^\bullet(\fgl_N[-2])$,
localized with respect to quasiisomorphisms.
Tensor product of dg-modules gives this category a monoidal structure.
One of the versions of  {\em derived  Satake equivalence} proved in ~\cite{bf}
states that there is an equivalence
$D_\perf^{\GL_N}(\Sym^\bullet(\fgl_N[-2]))\iso D_{\bG_\bO}(\Gr)$,
of  triangulated monoidal categories.

It will be convenient for us to reformulate the above result as follows. Let $T^*\GL_N[2]$
be the shifted cotangent bundle on $\GL_N$, viewed  as a dg-scheme
equipped  with zero differential.
The action of  $\GL_N$ on  itself by left and right translations
induces a $\GL_N\times \GL_N$-action  on $T^*\GL_N[2]$.
Let $D_\perf^{\GL_N\times\GL_N}(T^*\GL_N[2])$  be the triangulated category of
$\GL_N\times\GL_N$-equivariant  perfect complexes of $\CO_{T^*\GL_N[2]}$-modules on $T^*\GL_N[2]$.
The fiber of   $T^*\GL_N[2]$ over $1\in \GL_N$ may (and will) be identified
with $(\fgl_N[-2])^*=\fgl_N^*[2]$.
Restriction to this fiber induces a monoidal equivalence
$D_\perf^{\GL_N\times\GL_N}(T^*\GL_N[2])\iso D_\perf^{\GL_N}(\Sym^\bullet(\fgl_N[-2]))$,
where the category on the right is identified with the triangulated category of
$\GL_N$-equivariant  perfect complexes on $(\fgl_N[-2])^*$.
Thus, the derived  Satake equivalence stated above may be interpreted as
 an equivalence
$D_\perf^{\GL_N\times\GL_N}(T^*\GL_N[2])\iso D_{\bG_\bO}(\Gr)$.

There is also a natural `quantum' counterpart of the latter equivalence,
where the category $D_\perf^{\GL_N\times\GL_N}(T^*\GL_N[2])$
is replaced by  an appropriately defined category of
asymptotic shifted weakly equivariant $D$-modules on $\GL_N$, and the
category $D_{\bG_\bO}(\Gr)$ is replaced by  $D_{\bG_\bO\rtimes\BC^\times}(\Gr)$, see~\cite{bf}.

\subsection{Mirabolic Satake category}
In the present paper we will be interested in a
mirabolic analogue of the above setting. To explain this,
fix an $N$-dimensional vector space $V$ and put
$\bV=V\otimes\bF$, resp. $\bV_0= V\otimes\bO$ and
${\vphantom{j^{X^2}}\smash{\overset{\circ}{\vphantom{\rule{0pt}{0.55em}}\smash{\mathbf V}}}}=
\bV\smallsetminus\{0\}$.
We identify $\bG=\GL_N$ with $\GL(V)$, so that $\bG_\bF$ acts on $\bV$ and
$\bG_\bO$ is the stabilizer of $\bV_0$.
Following~\cite{fgt}, the {\em mirabolic affine Grassmannian} is defined as
$\Gr\times\bV$. We let $\bG_\bO$ act on  $\Gr\times\bV$ diagonally. The orbits of $\bG_\bO$ in
$\Gr\times{\vphantom{j^{X^2}}\smash{\overset{\circ}{\vphantom{\rule{0pt}{0.55em}}\smash{\mathbf V}}}}$ are, unlike the case of $\bG_\bO$-orbits in $\Gr$, not finite dimensional.
Heuristically, these orbits  are
of semi-infinite type in the sense that  the `closure' of an orbit
projects onto a (finite dimensional) Schubert variety in $\Gr$ and onto a lattice in $\bV$.

In view of the above, defining the correct mirabolic anague of the equivariant derived Satake
category requires some care.
According to our definition, an object of this category is supported on the product of
a finite-dimensional Schubert variety in
$\Gr$ and a lattice inside the Tate vector space $\bV$; moreover, this object is pulled back from a
finite-dimensional quotient of this lattice. According to three possible choices
($!$-,$*$-, or $!*$-) of pull-back, one gets the three versions
$D_{!\bG_\bO}(\Gr\times\bV), D_{*\bG_\bO}(\Gr\times\bV), D_{!*\bG_\bO}(\Gr\times\bV)$ of
$\bG_\bO$-equivariant constructible derived categories on $\Gr\times\bV$.
These categories are related to each other by certain renormalization equivalences.

We equip the above categories with  monoidal structures given by
various types of convolution operation.
The convolution along the Grassmannian, the first factor in $\Gr\times\bV$,
is defined similarly to the case of  the usual Satake category.
The convolution along the second factor depends on the choice
of category. Specifically, the convolution operation $\srel!\oast$ in
$D_{!\bG_\bO}(\Gr\times\bV)$
involves the~{$!$-tensor} product of constructible sheaves on $\bV$.
The convolution operation $\srel**$ on $D_{*\bG_\bO}(\Gr\times\bV)$ involves
a $*$-push-forward along $+\colon \bV\times\bV\to\bV$, the map given by addition.
These two types of  convolution are related to each other
via Fourier Transform (along $\bV$).
Finally, we define a monoidal structure $\star$ on $D_{!*\bG_\bO}(\Gr\times\bV)$
via the fusion operation on a mirabolic analogue of the
Beilinson-Drinfeld Grassmannian.

The categories above have natural counterparts
 involving the action of $\BC^\times$ on
$\Gr$ by loop rotation. These are  $\BC[\hbar]$-linear categories where
$\BC[\hbar]=H^\bullet_{\BC^\times}(\pt)$.
Thus, there is a category
$D_{!\bG_\bO\rtimes\BC^\times}(\Gr\times\bV)$, resp.\ $D_{*\bG_\bO\rtimes\BC^\times}(\Gr\times\bV)$,
equipped with a similarly defined
monoidal structure  $\srel!\oast$, resp.\ $\srel**$.
The fusion operation $\star$ on $D_{!*\bG_\bO}(\Gr\times\bV)$
has {\em no} $\bG_\bO\rtimes\BC^\times$-equivariant counterpart, however.

\subsection{Mirabolic Satake equivalence}
\label{qua}
The category on the `other side' of mirabolic Satake equivalence
is  an appropriately defined triangulated category of  ``equivariant
asymptotic shifted $D$-modules'' on the vector space $\fgl_N$ of $N\times N$-matrices.
More formally, we introduce a $\BC[\hbar]$-algebra $\fD^\bullet$, a
graded version  of the algebra of differential operators on $\fgl_N$ (where $\deg\hbar=2$).
The action of $\GL_N$ on $\fgl_N$ by left and right multiplication induces
a $\GL_N\times\GL_N$-action on $\fD^\bullet$.
The relevant category is then defined to be the derived category of  weakly
$\GL_N\times\GL_N$-equivariant perfect dg-modules over $\fD^\bullet$,
where $\fD^\bullet$ is viewed as a dg-algebra with zero differential.
Similarly to the constructible story,
we actually define three versions $D_\perf^{\GL_N\times\GL_N}(\fD^\bullet_{1,1})$,
resp.\ $D_\perf^{\GL_N\times\GL_N}(\fD^\bullet_{0,2})$ and
$D_\perf^{\GL_N\times\GL_N}(\fD^\bullet_{2,0})$, of such a category that correspond
to three different choices of grading on our algebra.
The algebra  of asymptotic differential operators specializes at
$\hbar=0$ to the algebra $\BC[T^*\fgl_N]$. We denote the specialization of
$\fD^\bullet_{i,2-i}$ at $\hbar=0$ by $\fS^\bullet_{i,2-i}$.
Accordingly, the above defined  $\BC[\hbar]$-linear  categories specialize at
$\hbar=0$ to various versions of the derived category of $\GL_N\times\GL_N$-equivariant
coherent sheaves on shifted cotangent bundle $T^*\fgl_N$.

Next, we equip  the above defined categories with monoidal structures.
The  monoidal  structure  $\stackrel{A}*$ on  $D_\perf^{\GL_N\times\GL_N}(\fD^\bullet_{2,0})$
is defined as a convolution operation on $D$-modules associated
with the map $\fgl_N\times \fgl_N\to \fgl_N$, given by  matrix multiplication.
The  monoidal  structure  $\stackrel{B}*$ on
$D_\perf^{\GL_N\times\GL_N}(\fD^\bullet_{0,2})$ is defined as
${\mathsf{F}}^{-1}\circ(\stackrel{A}*)\circ(\mathsf{F}\x\mathsf{F})$,
where $\mathsf{F}$ is the functor of Fourier Transform on $D$-modules.
Each of these monoidal structures has a quasiclassical limit at $\hbar=0$, defined
as a convolution of coherent sheaves arising from a certain Lagrangian correspondence.
Finally, tensor product of coherent sheaves,
that is, the functor  $\CM,\CM'\mapsto\CM\otimes_{{\fS_{1,1}^\bullet}}\CM'$,
gives a monoidal structure on $D_\perf^{\GL_N\times\GL_N}(\fS^\bullet_{1,1})$.
This  monoidal structure has no counterpart for $D$-modules, i.e. for $\hbar\ne0$.

One of the main results of the present paper,
see~Theorems~\ref{main spherical mirabolic},\ref{main quantum},
states that one has the following
equivalences of  triangulated monoidal categories, called mirabolic Satake equivalences:
  \begin{align*}  \big(D_\perf^{\GL_N\times\GL_N}(\fD^\bullet_{2,0}),\ \stackrel{A}*\big)\ \  &\xrightarrow[{}^{\cong}]{\Phi^{2,0}_\hbar}
   \   \    \big(D_{!\bG_\bO\rtimes\BC^\times}(\Gr\times\oV),\ \stackrel{!}\oast\big);\\
 \big(D_\perf^{\GL_N\times\GL_N}(\fD^\bullet_{0,2}),\ \stackrel{B}*\big)\  \   &\xrightarrow[{}^{\cong}]{\Phi^{0,2}_\hbar}\  \
   \  \      \big(D_{*\bG_\bO\rtimes\BC^\times}(\Gr\times\oV),\ \stackrel{*}*\big);\\
   \big(D_\perf^{\GL_N\times\GL_N}({\fS^\bullet_{1,1}}),\
 \otimes_{_{{\fS_{1,1}^\bullet}}}\big)\  \  &\xrightarrow[{}^{\cong}]{\Phi^{1,1}}\  \
    \   \
       \big(D_{!*\bG_\bO}(\Gr\times\oV),\ \star\big).
\end{align*}
Furthermore, it turns out that
the  triangulated category $D_{!*\bG_\bO}(\Gr\times\bV)$ is equivalent
to the bounded  derived category of  the abelian category
$\Perv_{\bG_\bO}(\Gr\times\bV)$ of $\bG_\bO$-equivariant perverse sheaves
on $\Gr\times\bV$, the heart of the perverse t-structure on  $D_{!*\bG_\bO}(\Gr\times\bV)$.

\begin{rem}  The counterpart of the last statement in the case of
 the usual Satake category is false: the  triangulated category $D_{!*\bG_\bO}(\Gr)$
is not  equivalent
to the  derived category of  the  category
$\Perv_{\bG_\bO}(\Gr)$, which is well known to be a semisimple  abelian category.
\end{rem}

\begin{rem} One can view  the group  $\GL_N$ as a Zariski open subset
of the vector space $\fgl_N$,  of all $N\times N$-matrices.
Associated with the
 open  imbedding $\GL_N\into\fgl_N$,  one has a restriction functor on $D$-modules.
It turns out that the counterpart of this functor for constructible derived categories is
 a suitably defined version of restriction
with respect to   the `zero section'  $\Gr\times\{0\}\into\Gr\times\bV$. That is,
one has
a natural  functor   $D_{\bG_\bO}(\Gr\times\bV)$ $\to D_{\bG_\bO}(\Gr)$.
We will show that the later functor makes  the  derived
Satake category $D_{\bG_\bO}(\Gr)$  a localization of the mirabolic  derived Satake category ${D_{\bG_\bO}(\Gr\times\bV)}$.
Moreover, the mirabolic Satake equivalence is compatible with the standard Satake
equivalence in the sense that there is a commutative diagram of functors

\[
\xymatrix{
^{\text{asymptotic equivariant}} _{\enspace\text{$D$-modules on $\fgl_N$}}
\ \ar[rrrr]^<>(0.5){\text{mirabolic derived Satake}}_<>(0.5){\cong}
\ar[d]^<>(0.5){\text{restriction}}
&&&&
D_{\bG_\bO\rtimes\BC^\times}(\Gr\times\bV)\ \ar[d]^<>(0.5){\text{localization}}\\
\ ^{\text{asymptotic equivariant}} _{\enspace\text{$D$-modules on $\GL_N$}}
\ \ar[rrrr]^<>(0.5){\text{derived Satake}}_<>(0.5){\cong}&&&&
\ D_{\bG_\bO\rtimes\BC^\times}(\Gr).
}
\]
\end{rem}

\subsection{Conjectural Iwahori-equivariant version}
Let $\bI\subset\bG_\bO$ be an Iwahori subgroup and $\Fl:=\bG_\bF/\bI$ the affine
flag variety. In ~\cite{b}, Bezrukavnikov established an equivalence of $D_\bI(\Fl)$, the
$\bI$-equivariant constructible derived category of $\Fl$, and the
derived category of $\GL_N$-equivariant coherent sheaves on an appropriate dg-version of the
Steinberg variety.
Motivated by this result, we expect that there is a mirabolic counterpart of this equivalence.

To explain this, fix a pair $V_1,V_2$, of $N$-dimensional vector spaces
and let $\Ffl_i,\ i=1,2$, denote  the variety of complete flags in $V_i$.
Further, consider a dg-scheme with zero differential
\[\on{H}_\mir:=\Hom(V_1,V_2)[1]\times\Hom(V_2,V_1)[1]\times\Ffl_1\times\Ffl_2.\footnote{Here we view
  both $\Hom(V_1,V_2)$ and $\Hom(V_2,V_1)$ as {\em odd} vector spaces, so that
  the functions on $\Hom(V_1,V_2)[1]\times\Hom(V_2,V_1)[1]$
  (with grading disregarded) form really a symmetric (infinite-dimensional)
  algebra, not an exterior algebra.}\]
Write $A$, resp.\ $B$,
for an element of $\Hom(V_1,V_2)[1]$, resp.\ $\Hom(V_2,V_1)[1]$, and
$F_i=(F^{(1)}_i\subset F^{(2)}_i\subset\ldots\subset F^{(N)}_i=V_i)$ for an element of $\Ffl_i$.

We define the {\em mirabolic Steinberg scheme}   to be
a dg-subscheme $\on{St}_\mir$ of $\on{H}_\mir$
cut out by the equations saying that the flag $F_2$ is stable under the composition
$AB$ and  the flag $F_1$ is stable under the composition
$BA$. Thus the mirabolic  Steinberg scheme is a shifted variety of quadruples:
\[
\on{St}_\mir=\{(A,B, F _1, F _2)\in \on{H}_\mir\mid
AB(F^{(j)}_2)\subseteq F^{(j)}_2\en \& \en
BA(F^{(j)}_1)\subseteq F^{(j)}_1,\ \forall j\in [1,N]
\}.
\]
Let $D_\bI(\Fl\times\bV)$ be the $\bI$-equivariant constructible derived category of $\Fl\times\bV$.
We propose the following
\begin{conj} There exists an equivalence of triangulated categories
\[
D^{\GL(V_1)\times\GL(V_2)}  \on{Coh}(\on{St}_\mir)
\cong D_\bI(\Fl\times\bV).
\]
\end{conj}

This conjecture would explain, in particular,
the appearance  of the same  polynomials, called the Kostka-Shoji polynomials,
in two different problems.
On the one hand, it was proved in~\cite{fgt} that these  polynomials are equal to the
Poincar\'e polynomials of the stalks of $\bG_\bO$-equivariant IC sheaves on
the mirabolic affine Grassmannian.
On the other hand, it was proved in~\cite{fi} that
the Kostka-Shoji polynomials are equal to the
Poincar\'e polynomials of graded multiplicities of the natural
$\GL_N\times\GL_N$-action on the space of global sections of line
bundles on a convolution diagram of the cyclic $\widetilde{A}_1$-quiver.

\subsection{Satake equivalence for some Lie supergroups}
\label{degen}
We consider the Lie superalgebra $\fgl(M|N)$ of endomorphisms of a super vector space
$\BC^{M|N}$, and the corresponding Lie supergroup $\GL(M|N)=\on{Aut}(\BC^{M|N})$.
We also consider a degenerate version $\ul\fgl(M|N)$ where the
supercommutator of the even elements (with even or odd elements) is the same as in $\fgl(M|N)$,
while the supercommutator of any two odd elements is set to be zero.
In other words, the even part $\ul\fgl(M|N){}_{\bar0}=\fgl_M\oplus\fgl_N$ acts naturally on the
odd part $\ul\fgl(M|N){}_{\bar1}=\Hom(\BC^M,\BC^N)\oplus\Hom(\BC^N,\BC^M)$, but the
supercommutator $\ul\fgl(M|N){}_{\bar1}\times\ul\fgl(M|N){}_{\bar1}\to\ul\fgl(M|N){}_{\bar0}$
equals zero.

The category of
finite dimensional representations of the corresponding supergroup $\ul\GL(M|N)$ (in vector
superspaces) is denoted $\Rep(\ul\GL(M|N))$, and its bounded derived category is denoted
$SD(\ul\GL(M|N))$.

There is a Koszul equivalence $\varkappa\colon SD(\ul\GL(N|N))\iso
SD_\perf^{\GL_N\times\GL_N}(\fS^\bullet_{1,1})$ (equivariant perfect dg-{\em super}modules over
dg-{\em super}algebra $\fS^\bullet_{1,1}$), see e.g.~\cite{mr}.
It intertwines the usual tensor product of
$\ul\GL(N|N)$-modules with the tensor product $\otimes_{\fS^\bullet_{1,1}}$ on
$SD_\perf^{\GL_N\times\GL_N}(\fS^\bullet_{1,1})$. Composing the Koszul equivalence $\varkappa$ with the
mirabolic Satake equivalence $\Phi^{1,1}\colon SD_\perf^{\GL_N\times\GL_N}(\fS^\bullet_{1,1})\iso
SD_{!*\bG_\bO}(\Gr\times\bV)$ (constructible sheaves of {\em super}vector spaces) we obtain
a super Satake equivalence $SD(\ul\GL(N|N))\iso SD_{!*\bG_\bO}(\Gr\times\bV)$ that intertwines
the usual tensor product of $\ul\GL(N|N)$-modules with the fusion product on
$SD_{!*\bG_\bO}(\Gr\times\bV)$. Moreover, this derived equivalence is exact 
with respect to the tautological $t$-structure on
$SD(\ul\GL(N|N))$ with the heart $\Rep(\ul\GL(N|N))$, resp.  the perverse $t$-structure on
$SD_{!*\bG_\bO}(\Gr\times\bV)$ with the heart $S\Perv_{\bG_\bO}(\Gr\times\bV)$.



\bigskip

Similarly, we construct equivalences
\[SD(\ul\GL(N-1|N))\xrightarrow[\ol\varkappa]{\sim}
SD_\perf^{\GL_{N-1}\times\GL_N}(\ol\fS{}^\bullet_{1,1})
\xrightarrow[\ol\Phi]{\sim}SD_{\GL(N-1,\bO)}(\Gr_{\GL_{N}}),\]
where $\ol\fS{}^\bullet_{1,1}=\Sym\!{}^\bullet\big(\Hom(\BC^{N-1},\BC^N)[-1]\oplus
\Hom(\BC^N,\BC^{N-1})[-1]\big)$.\footnote{Here we view
  both $\Hom(\BC^{N-1},\BC^N)$ and $\Hom(\BC^N,\BC^{N-1})$ as {\em odd} vector spaces, so that
  $\Sym\!{}^\bullet\big(\Hom(\BC^{N-1},\BC^N)[-1]\oplus\Hom(\BC^N,\BC^{N-1})[-1]\big)$
  (with grading disregarded) is really a symmetric (infinite-dimensional)
  algebra, not an exterior algebra.}
The composition is again $t$-exact with respect to the tautological $t$-structure on
$SD(\ul\GL(N-1|N))$ with the heart $\Rep(\ul\GL(N-1|N))$, resp. the perverse $t$-structure on
$SD_{\GL(N-1,\bO)}(\Gr_{\GL_{N}})$  with the heart $S\Perv_{\GL(N-1,\bO)}(\Gr_{\GL_{N}})$.
Moreover, the composition intertwines the usual tensor product of
$\ul\GL(N-1|N)$-modules with the fusion product on $S\Perv_{\GL(N-1,\bO)}(\Gr_{\GL_{N}})$.
Similarly to Section~\ref{qua}, the functor $\ol\Phi$ can be extended to an equivalence
$\ol\Phi{}_\hbar\colon SD_\perf^{\GL_{N-1}\times\GL_N}(\ol\fD{}^\bullet)\iso
SD_{\GL(N-1,\bO)\rtimes\BC^\times}(\Gr_{\GL_{N}})$, where $\ol\fD{}^\bullet$ is the graded Weyl algebra
of shifted differential operators on $\Hom(\BC^{N-1},\BC^N)$ (with $\deg\hbar=2$ and  all the
other generators in degree 1).

\subsection{Gaiotto conjectures}
\label{gaicon}
One may wonder if there is a geometric realization of categories of representations of
{\em nondegenerate} supergroups $\GL(N|N),\ \GL(N-1|N)$. It turns out that such a realization
exists (conjecturally) for the categories of integrable representations of
{\em quantized} algebras $U_q(\fgl(N|N)),\ U_q(\fgl(N-1|N))$. First of all, similarly to the
classical Kazhdan-Lusztig equivalence, it is expected that
$U_q(\fgl(M|N))\on{-mod}\cong\on{KL}_c(\widehat\fgl(M|N))$, where $q=\exp(\pi\sqrt{-1}/c)$,
and $\on{KL}_c(\widehat\fgl(M|N))$ stands for the derived category of
$\GL(M,\bO)\times\GL(N,\bO)$-equivariant
$\widehat\fgl(M|N)$-modules at the level corresponding to the invariant bilinear form
$(X,Y)=c\cdot\on{sTr}(XY)-\frac12\Killing_{\fgl(M|N)}(X,Y)$ on $\fgl(M|N)$.
Second, it is expected that
the category $\on{KL}_c(\widehat\fgl(N-1|N))$ is equivalent to the
$q$-monodromic
$\GL(N-1,\bO)$-equivariant derived constructible category of the complement
$\CL^\bullet$ of the zero section of the determinant line
bundle on $\Gr_{\GL_N}$, and this equivalence takes the standard $t$-structure of
$\on{KL}_c(\widehat\fgl(N-1|N))$ to the perverse $t$-structure.

Further, it is expected that $\on{KL}_c(\widehat\fgl(N|N))$ is equivalent to the 
$q$-monodromic $\GL(N,\bO)$-equivariant derived constructible category of $\CL^\bullet\times\bV$,
and this equivalence takes the standard $t$-structure of $\on{KL}_c(\widehat\fgl(N|N))$ to the
perverse $t$-structure. For $M<N-1$ it is expected that $\on{KL}_c(\widehat\fgl(M|N))$ is
equivalent to the $q$-monodromic $\GL(M,\bO)$-equivariant derived constructible category of
$\CL^\bullet$ with certain Whittaker conditions, cf.~Section~\ref{simple-minded} for more details.
In particular,  the special case $M=0$ of this conjecture follows from the Fundamental Local
Equivalence~\cite{Gait-lur,Gait-quantum,Gait-whit} of the geometric Langlands program.

There are similar expectations for other classical (i.e.\ orthosymplectic) Lie superalgebras; the
interested reader may try to find them in~\cite{gw}.

\subsection{Conjectures of Ben-Zvi, Sakellaridis and Venkatesh}
In an ongoing project of D.~Ben-Zvi, Y.~Sakellaridis and
A.~Venkatesh, the authors propose the Periods---$L$-functions duality conjectures. Their
conjectures predict, among other things, that given a reductive group $\on{G}$ and its spherical
homogeneous variety $X=\on{G}/\on{H}$, there is a subgroup $\on{G}^\vee_X\subset\on{G}^\vee$, its
graded representation $V^\vee_X=\bigoplus_{i\in\BZ}V^\vee_{X,i}[i]$, and an equivalence
$D\Coh(V^\vee_X/\on{G}^\vee_X)=D\Coh\big((\bigoplus_{i\in\BZ}V^\vee_{X,i}[i])/\on{G}^\vee_X\big)
\simeq D_{\on{G}(\bO)}(X(\bF))$.
For a partial list of examples, see the table at the end of~\cite{sa}. The relevant
representations $V^\vee_X$ (constructed in terms of the Luna diagram of $X$) can be
read off from the 4-th column of the table.

It turns out that the equivalences discussed in~Section~\ref{degen} fit into the general
setting outlined in the previous paragraph. Thus the case of Example~13 of~\cite{sa}
corresponds to the equivalence
$\ol{\Phi}\colon D_\perf^{\GL_{N-1}\times\GL_N}(\ol\fS{}^\bullet_{1,1})\iso D_{\GL(N-1,\bO)}(\Gr_{\GL_{N}})$
of~Section~\ref{degen}. To explain this, let $\on{G}:=\GL_{N-1}\times\GL_N$ and
$\on{H}:=\GL_{N-1}$. We view $\on{H}$ as a block-diagonal subgroup of $\on{G}$
and put $X=\on{G}/\on{H}$.
Then loosely speaking we have $D_{\GL(N-1,\bO)}(\Gr_{\GL_{N}})\simeq
D\big(\GL(N-1,\bO)\backslash\GL(N,\bF)/\GL(N,\bO)\big)\simeq
D\big(\on{G}(\bO)\backslash\on{G}(\bF)/\on{H}(\bF)\big)\simeq
D\big(\on{G}(\bO)\backslash X(\bF)\big)\simeq D_{\on{G}(\bO)}(X(\bF))$.
On the other hand, we consider a graded $\on{G}^\vee$-module
$V^\vee_X:=\Hom(\BC^{N-1},\BC^N)[1]\oplus\Hom(\BC^N,\BC^{N-1})[1]$
(similarly to the footnote in~Section~\ref{degen}, we view $V^\vee_X$ as an
  {\em odd} vector space placed in cohomological degree $-1$. Note also that
  $\on{G}^\vee\simeq\on{G}=\GL_{N-1}\times\GL_N$).
Hence, the equivalence $\ol{\Phi}$ of~Section~\ref{degen} takes the form
$D\Coh(V^\vee_X/\on{G}^\vee)\simeq D_{\on{G}(\bO)}(X(\bF))$.

Similarly, Example~14 of~\cite{sa}
matches the Gaiotto conjecture of~Section~\ref{gaicon} for an orthosymplectic Lie superalgebra.

\subsection{Organization of the paper}
In Section~\ref{gaiotto conjectures} (which is not necessary for understanding subsequent sections)
we formulate the Gaiotto conjectures and explain their relation with the geometric Langlands
program; in particular, with the Fundamental Local Equivalence.
Section~\ref{coherent realization} is the technical core of the paper. In this section we
establish a coherent description of the spherical mirabolic affine Hecke category
$D_{\GL(N,\bO)}(\Gr\times\bV)$ with its three monoidal structures. Also, as a preparation
for~Section~\ref{coh realization N-1} we give a coherent description of the restriction
functor $D_{\GL(N,\bO)}(\Gr\times\bV)\to D_{\GL(N,\bO)}(\Gr\times\bV_0)$.\footnote{Recall that
  $\bV_0=V[\![t]\!]$ is the standard lattice in the Tate vector space $\bV$.}
In~Section~\ref{coh realization N-1} we establish a coherent description of the category
$D_{\GL(N-1,\bO)}(\Gr)$ along with its fusion monoidal structure.
In~Section~\ref{loop rotation} we prove the quantum analogues,
that is, counterparts  for 
categories equivariant with respect to the loop rotations, of the  results above.

\subsection{Acknowledgments}
We are grateful to R.~Bezrukavnikov, D.~Ben-Zvi, I.~Entova-Aizenbud, P.~Etingof, B.~Feigin,
D.~Gaiotto, D.~Gaitsgory, J.~Hilburn, D.~Leites, C.~Mautner,
Y.~Sakellaridis, V.~Serganova, A.~Venkatesh
and P.~Yoo for very useful discussions. A.B.\ was partially supported by NSERC.
M.F.\ was partially funded within the framework of the HSE
University Basic Research Program and the Russian Academic Excellence Project `5-100'.
The work of V.G.\ was supported in part by an NSF award DMS-1602111.

\section{Gaiotto conjectures}
\label{gaiotto conjectures}
The purpose of the present Section is to put the main results of this paper into a wider framework which has to do with
the local geometric Langlands correspondence. This section will be used for motivation purposes only.
It can be safely skipped by  the readers who are not interested in the geometric Langlands correspondence.
The ideas of this Section are due to D.~Gaiotto (private communication).
Gaiotto  informed us that his ideas
were motivated, to a large extent, by his discussions with P.~Yoo as well as by \cite{MW}.

\subsection{Reminder on strong actions on categories}
Let $G$ be a connected reductive group over $\CC$ and let $\kap$ denote an invariant symmetric bilinear form on the Lie algebra $\grg$ of $G$ (when $G$ is simple, the vector space of such  bilinear 
forms is 1-dimensional, so we can think of $\kap$ as an element of $\CC$). Then there is a notion of {\em strong action of the group $G_\bF$ on a category $\calC$ of level $\kap$}. We refer the reader
to~\cite{Gait-quantum} for details of the definition. It will be important for us later
 that  this definition is, in some sense,  invariant under integral shifts.
Specifically,
any category $\calC$ 
 endowed with a $G_\bF$-action of level $\kap$
has  a natural $G_\bF$-action  of level $\kap+\kap'$,
where
  $\kap'$ is another form as above which is, moreover, integral 
in the sense that the corresponding quadratic form is integral and even on elements of the coweight lattice of $G$.
Here are two  important examples:

\medskip
\noindent
1) Let $\hat{\grg}_{\kap}$ denote the central extension of $\grg_\bF$ associated with the form $\kap$. 
Let  $\hat{\grg}_{\kap}$-mod be the category of continuous (with respect to $t$-adic topology) modules over 
$\hat{\grg}_{\kap}$ such that the element 1 of the center acts on the module as the identity.
Then, the adjoint action of $G_\bF$ on $\hat{\grg}_{\kap}$ has a natural lift
 to a strong action of $G_\bF$ on $\hat{\grg}_{\kap}$-mod of level~$\kap$. 
More generally, let $\gra=\gra_{\bar0}\oplus \gra_{\bar1}$ be a Lie superalgebra. Assume that:

\begin{i-ii-iii}
\item\label{gai} $G$ acts on $\gra$;

\item\label{gaii} We are given a map $\iota\colon \grg\to \gra_{\bar0}$ such that the corresponding adjoint action of $\grg$ on $\gra$ is equal to the derivative of the $G$-action from (\ref{gai}).

\item\label{gaiii} The algebra $\gra$ is equipped with an invariant symmetric (in the super-sense) bilinear form $\kap_{\gra}$. Let  $\kap_{\grg}$ be its pull-back to $\grg$.
\end{i-ii-iii}

Associated with the form $\kap_{\gra}$, there is a canonical Kac-Moody extension
$\hat{\gra}$ of $\gra_\bF$.
As before, we denote by $\hat{\gra}$-mod the category of continuous modules $M$ over $\hat{\gra}$ 
such that the element 1 of the center acts on $M$ by the identity. This category comes
equipped with a $G_\bF$-action  of level $\kap_{\grg}$.

We will mostly be interested in the following special case of the above construction. 
Fix a pair $M,N$, of non-negative integers.
Let $\gra=\gl(M|N)$, $G=\GL_M\times \GL_N$. Let $\kap_{\gra}(x,y)=c\cdot \sTr(xy)$,
where $\sTr$  stands for `super-trace' and $c\in \CC$.

\medskip
\noindent
2) Assume that the form $\kap$ is integral and even. Then,
this form gives rise to a central extension $\hatG$ of $G_\bF$. Let $X$ be an ind-scheme equipped 
with a $\hatG$-equivariant line bundle $\calL$. Then for any $c\in \CC$ one has the category 
$D_c(X)$-mod of $c$-twisted $D$-modules on $X$. This category has a natural strong $G_\bF$-action 
 of level $c\cdot \kap$.

\subsection{Digression on the local (quantum) geometric Langlands correspondence}
\label{digression}
Let  $G$ and  $\kap$ be as above
and assume in addition that the form $\kap$  is non-degenerate.
Let $G^{\vee}$ denote the Langlands dual group. Since $\kap$ is non-degenerate it gives a similar form 
$\kap^{\vee}$ for $G^{\vee}$.
Further, put  $\kap_{\on{crit}}=-\frac12{\Killing}_\fg$, where ${\Killing}_\fg$ stands for the Killing form
on the Lie algebra $\fg$ of $G$.

The local quantum geometric Langlands conjecture is, roughly speaking, an equivalence of 2-categories
$$
\{ \text{Categories with strong $G_\bF$-action of level $\kap+\kap_{\on{crit}}$}\}$$
and
$$\{ \text{Categories with strong $G^{\vee}_\bF$-action of level $-\kap^{\vee}-\kap^{\vee}_{\on{crit}}$}\}.
$$
The form $\kap_{\on{crit}}$ being integral, the shift by $\kap_{\on{crit}}$ in the above formulation
 is not very essential. However, it is convenient  for many applications to make this shift.

To give a  rigorous meaning to the above conjecture one needs, first of all, to replace all `categories'
by  suitable `dg-categories'. This upgrades each side of the equivalence to an $(\infty,2)$-category. 
Then for generic, i.e. non-rational, $\kap$ the above equivalence is expected to hold as stated. 
For general $\kap$, more corrections are necessary but 
we will not discuss 
this  here since the local geometric Langlands correpondence will only 
serve as a guiding principle.

There is a ``limiting version" of the above conjecture  for $\kap=0$. To explain this,
we use the notion of a `category over a stack $\calS$', 
cf.~\cite[\S6]{Gait-quantum} and references therein. Write $\calD^\circ=\Spec(\bF)$ and 
 let $\LS_{G^{\vee}}(\calD^\circ)$
be the classifying stack of principal $G^{\vee}$-bundles  on $\calD^\circ$ equipped with a connection.
The  ``classical'' local geometric Langlands conjecture  predicts a close relationship between
$$
\{ \text{Categories with strong $G_\bF$-action of level $\kap_{\on{crit}}$}\}
$$
and
$$
\{\text{Categories over $\LS_{G^{\vee}}(\calD^\circ)$}\}.
$$
Here,
$\kap_{\on{crit}}$ can be replaced by $0$,
since $\kap_{\on{crit}}$ is integral.
Again, it is possible to make
 the informal relation above a rigorous mathematical conjecture,
cf.\ again~\cite[\S6]{Gait-quantum} for a more detailed discussion.

In the remaining part of this Section we will pretend that both quantum and classical geometric Langlands conjectures hold as stated. We will write  $\calC\mapsto \calC^{\vee}$ for 
the resulting correspondence .

An important example of such a correspondence is as follows. Fix a non-degenerate form $\kap$
and let  $\calC=D_{\kap+\kap_{\on{crit}}}(\Gr_G)$-mod be the category of
$(\kap-\kap_{\on{crit}})$-twisted $D$-modules on the affine Grassmannian $\Gr_G$ of $G$.
It is expected that in this case the category $\calC^{\vee}$ is 
 the category $D_{-\kap^{\vee}-\kap^{\vee}_{\on{crit}}}(\Gr_{G^{\vee}})$-mod.
In the limiting case $\kap=0$ category $\calC^{\vee}$ is expected to be a
 push-forward of  category $\QCoh(\pt/G^{\vee})$ under the natural map $\pt/G^{\vee}\to \LS_{G^{\vee}}(\calD^\circ)$ 
induced by the imbedding of  the trivial local system. These expectations would imply the following:
\begin{i-ii-iii}
\item If $\kap$ is non-degenerate, then $\calC^{G_\bO}\simeq (\calC^{\vee})^{G^{\vee}_\bO}$ (here $\calC^{G_\bO}$ denotes the category of $G_\bO$-equivariant objects in $\calC$).
\item If $\kap=0$ then $\calC^{G_\bO}$ is equivalent to the pull-back of the category 
$\calC^{\vee}$ under the map $\pt/G^{\vee}\to \LS_{G^{\vee}}(\calD^\circ)$ corresponding to the trivial local system.
\end{i-ii-iii}

\subsection{Whittaker category and the fundamental local equivalence (FLE)} In this
subsection we  discuss another  important example of Langlands dual
categories. We refer to \cite{Gait-quantum} and \cite{Gait-whit} for details.

Let $U$ be a maximal unipotent subgroup of $G$ and let $\chi_0\colon U\to \GG_a$ be a non-degenerate character. We 
define a character
$\chi\colon U_\bF\to \GG_a$ by 
$$
\chi(u(t))=\Res_{t=0} \chi_0(u(t))dt.
$$
Given   a (co-complete, dg)  category  $\calC$ with a strong $G_\bF$-action of a fixed level
one can define a  category $\Whit(\calC)$ of
$(U_\bF,\chi)$-equivariant objects in~$\calC$.

We consider  the category $D_{\kap}(G_\bF)$-mod of $\kap$-twisted $D$-modules on $G_\bF$. 
According to \cite{AG} this category has a natural action of $G_\bF$ of level $\kap$ that comes from left translations, and
also another action of $G_\bF$ of level $-\kap+2\kap_{\on{crit}}$ that comes from right translations.

Let $\Whit^r_{\kap}(G_\bF)$ denote its Whittaker category with respect to the right action. This category
inherits the left action of level $\kap$. One could ask what is its Langlands dual category.

For simplicity, below we will only consider either the case where $\kap$ is not rational
(i.e.\ the value of the corresponding quadratic form on any coroot is not a rational number) or the case $\kap=0$.

Assume first that $\kap$ is non-degenerate and not rational. Then it is expected (cf.~\cite{Gait-quantum}) that
\begin{equation}\label{fle-kap}
  \text{$\Whit^r_{\kap-\kap_{\on{crit}}}(G_\bF)^{\vee}$ is the category
    $\widehat{\grg^{\vee}}\!\!\!_{-\kap^{\vee}+\kap^{\vee}_{\on{crit}}}$-mod.}
\end{equation}
In the case $\kap=0$ the expected answer is the category $\QCoh(\LS_{G^{\vee}}(\calD^\circ))$.

 Let $\Whit_{\kap}(\Gr_G)$ denote the category of $\kap$-twisted $D$-modules on $\Gr_G$ and let $\KL_{\kap}(\grg)=(\hat{\grg}_{\kap}\text{-mod})^{G_\bO}$.
Then, statements (i) and (ii) of Section~\ref{digression} imply
 the following result, which has been  proved rigorously in \cite{Gait-whit} and \cite{Gait-lur} (it goes under the name ``fundamental local equivalence"):
\begin{thm}
Assume that $\kap$ is non-degenerate and not rational.
\begin{equation}\label{whit-kap}
\Whit_{\kap+\kap_{\on{crit}}}(\Gr_G)\simeq \KL_{\kap^{\vee}+\kap^{\vee}_{\on{crit}}}(\grg^{\vee}).
\end{equation}
\begin{equation}\label{whit-cl}
\Whit(\Gr_G)\simeq \Rep(G^{\vee}).
\end{equation}
Moreover, these equivalences hold on the level of abelian categories.
\end{thm}

\begin{rems}
1) The critical shifts in~(\ref{whit-kap}) are not important for irrational $\kap$ if  one only cares
 about both sides as abstract categories. However, we still prefer to keep them, since in this way one can also extend the statement to rational $\kap$; in addition the shifts are important if we keep track of some natural structures on these categories (cf.\ 4) below.

2) Note the absence of a negative sign before $\kap^{\vee}+\kap^{\vee}_{\on{crit}}$ in the RHS of (\ref{whit-kap}). This has to do with the fact that $\Whit_{\kap+\kap_{\on{crit}}}(\Gr_G)$ is actually the category of $G_\bO$-equivariant objects with respect to the {\em left} action of $G_\bO$ on the category of $D$-modules on $G_\bF$ which are Whittaker {\em on the right}. This change of right to left is what is responsible for the change of sign.

3) The fact that the above equivalences respect the natural t-structures
does not follow (to the best of our knowledge)
from any geometric Langlands considerations. In fact, at the level of (unbounded) derived categories the statement holds for all $\kap$, but when $\kap$ is positive rational it is very far from an abelian equivalence.

4) The above equivalences are in fact not only equivalences of abstract categories, but  equivalences of
categories with {\em factorization structure}, a notion   
closely related to the notion of a braided monoidal category (it is also worthwhile to note that as an abstract category, $\KL_{\kap}(\hat{\grg})$ is independent of $\kap$ if $\kap$ is irrational; but it is not so if we take into account the factorization structure).
\end{rems}
\subsection{Gaiotto conjectures: geometric Langlands form in the case $N>M$}
There is one more series of examples which are relevant for the subject of this paper.
Fix two non-negative integers $N$ and $M$ such that $N\geq M$ and set $G_{M,N}=\GL_M\times \GL_N$. Note that $G_{M,N}$ is isomorphic to $G_{M,N}^{\vee}$.

We are going to produce an example of Langlands dual categories for $G_{M,N}$; for $M=0$ we will recover (\ref{fle-kap}).
First, we describe the analog of the RHS. Let $c\in \CC$. Then the category in question will be the category
$\widehat{\gl}(M|N)_c$-mod of modules over the affine Lie super-algebra $\widehat{\gl}(M|N)$ of level $c\cdot\kap_{M,N}+\frac12{\Killing}_{M,N}$, where

1) $\kap_{M,N}(x,y)=\sTr(x y)$.

2) ${\Killing}_{M,N}$ is the restriction of the Killing form of $\gl(M|N)$ to the even part (note that it is degenerate if $M=N$).

This category has a natural action of the group $G_{M,N}(\bF)$ of certain level which is an
integral shift of $(c\cdot\kap_M,-c\cdot\kap_N)$ (here $\kap_N$ denotes the standard
  invariant bilinear form on the Lie algebra $\gl_N$ equal to $\Tr(X\cdot Y)$).
As has been explained above, one can twist the action of $G_{M,N}(\bF)$ on this category so that
the twist becomes equal to $(c\cdot\kap_M,-c\cdot\kap_N)-\kap_{\crit}$ (here by $\kap_{\crit}$ we
mean the critical bilinear form for the Lie algebra $\mathfrak{g}_{M,N}$).
Hence it makes sense to consider its Langlands dual. This should be a category with a strong
action of $G_{M,N}(\bF)$ of level
$c^{-1}\cdot(\kap_M,-\kap_N)-\kap_{\crit}$. However, we can again twist the action and think of it as a
category with an action of $G_{M,N}(\bF)$ of level
$c^{-1}\cdot(\kap_M,-\kap_N)$.
Let us give a conjectural description of the Langlands dual category according to a prediction
of D.~Gaiotto.

Next, let $M<N$. We define a unipotent subgroup $U_{M,N}$ of $\GL_N$ as follows. If $M=N-1$ this subgroup is trivial, 
and if $M=0$ it is the group $U_N$ of unipotent upper-triangular matrices. 
In the general case,  $U_{M,N}$ is a subgroup of $U_N$ defined as follows.

Let $e_{M,N}\in \gl_N$ be the standard upper-triangular Jordan block of size $N-M$, i.e.
\[e_{M,N}=\sum\limits_{i=1}^{N-M-1} E_{i,i+1},\]
where $E_{ij}$ stands for the matrix whose $(i,j)$-entry is equal to 1 and all other entries are equal to 0.
The element $e_{M,N}$ is part of an $\mathfrak{sl}_2$-triple $(e_{M,N},h_{M,N},f_{M,N})$.
Here $f_{M,N}=\sum\limits_{i=1}^{N-M-1} i(N-M-i) E_{i+1,i}$ and $h_{M,N}$ is the diagonal matrix which has
diagonal entries
$(N-M-1,N-M-3,\ldots,-N+M+1,0,\ldots,0)$. For any integer $l$ we let $\mathfrak{g}_l$ denote the
$l$-eigen-space of the adjoint action of $h_{M,N}$ on $\gl_N$. We  set
\[\mathfrak{u}_{M,N}=\bigoplus\limits_{l\geq 2} \mathfrak{g}_l\oplus \mathfrak{g}_1^+,\]
where $\mathfrak{g}_1^+$ is the intersection of $\mathfrak{g}_1$ with the Lie algebra of upper-triangular matrices.
We define a Lie algebra  homomorphism
$\chi^0_{M,N}\colon \mathfrak{u}_{M,N}\to \mathbb{C}$ by sending a matrix
$(u_{ij})$ to $\sum\limits_{i=1}^{N-M-1} u_{i,i+1}$. 
Let  $U_{M,N}$  be  a unipotent subgroup of $\GL_N$ with Lie algebra $\mathfrak{u}_{M,N}$
and let $U_{M,N}\to \mathbb{G}_a$ be the homomorphism induced by $\chi_{M,N}^0$
(which we denote by the same
symbol $\chi_{M,N}^0$).
  
We imbed the group $\GL_M$ into the centralizer of the element $h_{M,N}$ in $\GL_N$ 
(this is the block-diagonal embedding corresponding to rows ${N-M+1},\ldots,N$). It is easy to
see that the group $\GL_M$ normalizes the subgroup $U_{M,N}$ and the homomorphism $\chi^0_{M,N}$
is fixed by the $\GL_M$-action on  $U_{M,N}$ by conjugation.

\begin{rem}
  \label{Roma's approach}
  $U_{M,N}$ is conjugate to the subgroup $U'_{M,N}$ formed by the block-upper-triangular matrices
  $\begin{pmatrix}U_r& {*} &{*}\\0&1_{M+1}& {*}\\0&0& U_s\end{pmatrix}$ where
    \[r=\ceil{(N-M-1)/2},\ s=\floor{(N-M-1)/2};\] in particular $r+s=N-M-1$.
    Here $U_p$ stands for an arbitrary unipotent upper-triangular matrix in $\GL_p$,
    and  the notation `$*$' is used for  arbitrary matrices of an appropriate size.
    Moreover, the conjugation can be chosen so that the character $\chi_{M,N}^0$ corresponds
    to the character $\chi_{M,N}^{(r,s)}$ on $\mathfrak u_{M,N}^{(r,s)}:=\Lie U_{M,N}^{(r,s)}$ given by
    $(u_{ij})\mapsto\sum_{i=1}^{r-1} u_{i,i+1} + u_{rk} + u_{k,N-s+1} + \sum_{i=N-s+1}^{N-1} u_{i,i+1}$
    for any choice of $k\in\{r+1,\ldots,N-s\}$.
    The subgroup $U_{M,N}^{(r,s)}\subset\GL_N$ and the character $\chi_{M,N}^{(r,s)}$ are defined for arbitrary pair $(r,s)\in\BN^2:=\BZ_{\geq0}^2$ with $r+s=N-M-1$. (In the two extreme cases $\{r,s\}=\{0,N-M-1\}$, one of the middle terms in the formula for $\chi_{M,N}^{(r,s)}$ is undefined and should be omitted.) Moreover $\chi_{M,N}^{(r,s)}$ can be replaced by an arbitrary representative of the open $N_{\GL_N}(U_{M,N}^{(r,s)})$-orbit in $\mathfrak u_{M,N}^{(r,s)*}$. 
\end{rem}

\begin{rem}
  \label{classical approach}
  $U_{M,N}^{(0,N-M-1)}$ is also conjugate to the unipotent radical $U_{(M+1,1,\ldots,1)}$ of the standard parabolic
  subgroup $P_{(M+1,1,\ldots,1)}$ of $\GL_N$ corresponding to the partition
  $(M+1,1,\ldots,1)$ of $N$. The character
  $\chi^0_{M,N}$ is conjugate to the restriction of the regular character
  $u\mapsto\sum_{i=1}^{N-1}u_{i,i+1}$ of the upper triangular subgroup to $U_{(M+1,1,\ldots,1)}$,
  cf.~\cite[\S(2.11)]{jps} and~\cite[the beginning of~\S2 of Lecture 5]{c}.
\end{rem}

As before, we  define  a homomorphism $\chi_{M,N}\colon U_{M,N}(\bF)\to \mathbb{G}_a$ 
to be 
$\Res_{t=0}\chi^0_{M,N}$.

For any $c'\in \CC$ we now consider the category $D_{c'}(\GL(N,\bF))^{U_{M,N}(\bF),\chi_{M,N}}$.
By definition this is the derived category of $D$-modules on $\GL(N,\bF)$ twisted by $c'\cdot\kap_N$
that are equivariant {\em on the left} with respect to $(U_{M,N}(\bF),\chi_{M,N})$. This category
has a natural action of $\GL(M,\bF)$ of level $c'\cdot\kap_M$ coming from left multiplication
and an action of $\GL(N,\bF)$ of level $-c'\cdot \kap_M-\Killing_{\fgl_N}$ coming from right
multiplication (recall that $\Killing_{\fgl_N}$ denotes the Killing form on $\gl_N$).
As before, we can twist the second action to make it an action of level $-c'\cdot\kap_M$.

\begin{rem}
  In fact, using Fourier transform for $D$-modules, one can show that replacing
  $(r,s)$ in~Remark~\ref{Roma's approach} with any pair of non-negative integers whose sum equals
  $N-M-1$ produces a category equivalent to $D_{c'}(\GL(N,\bF))^{U_{M,N}(\bF),\chi_{M,N}}$.
\end{rem}

We now take $c'=1/c$. Then according to a conjecture of D.~Gaiotto the category
$D_{1/c}(\GL(N,\bF))^{U_{M,N}(\bF),\chi_{M,N}}$ is Langlands dual to $\widehat{\gl}(M|N)_c$-mod.

We can also consider the limit $c\to\infty$. To simplify the discussion we will not do it now,
but we will discuss it later when we pass to $G_\bO$-equivariant objects.

\subsection{Gaiotto conjectures: geometric Langlands form for $N=M$}
Let us also discuss the case $N=M$. In this case on the left we again take the same category
$\widehat{\gl}(N|N)_c$-mod of modules over the affine Lie super-algebra $\widehat{\gl}(N|N)$ of
level $c\cdot\kap_{N,N}-\frac12{\Killing}_{N,N}$ (note that the Killing form is zero for
${N=M}$
). The Langlands dual category (according to Gaiotto) is the derived category
$D_{1/c}(\GL(N,\bF)\times \bV)$ of $1/c$-twisted $D$-modules on $\GL(N,\bF)\times \bV$ where:

1) The twisting is with respect to the first factor.

2) The two actions of $\GL(N,\bF)$ come from the diagonal action coming from the left action of
$\GL(N,\bF)$ on the first factor and the natural action of $\GL(N,\bF)$ on the 2nd factor, and
the right multiplication action $\GL(N,\bF)$ on the first factor.

\subsection{Gaiotto conjectures: ``simple-minded" form}
\label{simple-minded}
The above statements are not really well-formulated mathematical conjectures, since the local
geometric Langlands duality is not known at present. However, we can turn them into precise
conjectures by using~(i) at the end of subsection~\ref{digression}, i.e.\ we are going to
take $G_{M,N}(\bO)$-invariants on both sides. We get the following conjectures:
\begin{conj}\label{main-c}
Assume that $N>M$ and assume that $c\neq 0$ is not a rational number.
The categories $D_{1/c}^{(\GL(M,\bO)\ltimes U_{M,N}(\bF),\chi_{M,N})}(\Gr_{\GL_N})$  and $\KL_c(\widehat{\gl}(M|N))$
are equivalent as factorization categories.
Here $\KL_c(\widehat{\gl}(M|N))$ is the category of $G_{M,N}(\bO)$-equivariant objects in the
category $\widehat{\gl}(M|N)_c$-mod.

Similarly the category $D_{1/c}^{\GL(N,\bO)}(\Gr_{\GL_N}\times \bV)$ is equivalent to
$\KL_c(\widehat{\gl}(N|N))$.
\end{conj}

We now want to take the limit $c\to \infty$. In this case $1/c$ goes to $0$ and the category of
$1/c$-twisted $D$-modules in Conjecture~\ref{main-c} just becomes the category of usual $D$-modules.
The $c\to \infty$ limit of the category $\KL_c(\widehat{\gl}(M|N))$ is not canonically defined:
one has to choose some nice extension of the corresponding family of categories from $\AA^1$ to
$\PP^1$, cf.~\cite[\S6]{z} or~\cite[\S4]{yifei}.
Naively, one might think that the correct extension is just the category of representations of
the super-group $\GL(M|N)$. However, it turns out that this is not the right choice. Instead, one
needs to consider the category of representations of the group $\ul\GL(M|N))$ defined
in~Section~\ref{degen}.
\medskip
\noindent
With these conventions one gets the following
\begin{conj}\label{main-0}
Assume that $N>M$. Then
the category $D^{(\GL(M,\bO)\ltimes U_{M,N}(\bF),\chi_{M,N})}(\Gr_{\GL_N})$ is equivalent to the category of modules over
the group $\ul\GL(M|N))$.

Similarly, for $N=M$ the category $D^{\GL(N,\bO)}(\Gr_{\GL_N}\times \bV)$ is equivalent to the
category of modules over the group $\ul\GL(N|N))$.

These equivalences should hold for both derived and abelian categories.
\end{conj}

In the present work we prove~Conjecture~\ref{main-0} for $M=N$ and $M=N-1$.

\section{A coherent realization of $D_{\GL(N,\bO)}(\Gr\times\bV)$}
\label{coherent realization}

\subsection{Setup and notation}
\label{setup}
We follow the notation of~\cite{fgt}.
Recall that $\Gr=\Gr_{\GL_N}=\bG_\bF/\bG_\bO=\GL(N,\bF)/\GL(N,\bO)$, where
$\bF=\BC\dprts{t}\supset\BC\dbkts{t}=\bO$. We consider a complex vector space
$V$ with a basis $e_1,\ldots,e_N$. 
We set $\bV=V\otimes\bF\supset V\otimes\bO=\bV_0$.

Recall that the $\bG_\bF$-orbits in $\Gr\times\Gr$ (resp.\
in $\Gr\times\Gr\times{\vphantom{j^{X^2}}\smash{\overset{\circ}{\vphantom{\rule{0pt}{0.55em}}\smash{\mathbf V}}}}$)
are numbered in~\cite[Section~3.1]{fgt} by signatures\footnote{sequences of integers
$\bnu=(\nu_1\geq\ldots\geq\nu_N)$, following the terminology of H.~Weyl.}
(resp.\ by bisignatures, i.e.\ pairs of signatures) in such a way that
the $\bG_\bF$-orbits on $\Gr\times\Gr$ numbered by partitions
$\bnu=(\nu_1\geq\ldots\geq\nu_N\geq0)$ correspond to the pairs of lattices $L^1\subset L^2$.
More precisely, the orbit corresponding to a bisignature $(\blambda,\bmu)$ contains a
point \[\big(L^1=\bO\langle e_1,e_2,\ldots,e_N\rangle,\ L^2=\bO\langle t^{-\lambda_1-\mu_1}e_1,
t^{-\lambda_2-\mu_2}e_2,\ldots,t^{-\lambda_N-\mu_N}e_N\rangle,\ v=\sum_{i=1}^Nt^{-\lambda_i}e_i\big).\]

Irreducible representations of $\GL_N=\GL(V)$ are also numbered by the signatures, so that e.g.\ the
determinant character $\det V$ corresponds to $(1^N)$. To a signature
$\bnu=(\nu_1\geq\ldots\geq\nu_N)$ we associate an irreducible representation $V_\bnu$ with the
highest weight $\bnu$.
The geometric Satake equivalence takes the
irreducible perverse sheaf $\IC_\bnu$ to the irreducible representation $V^*_\bnu$.
Thus if $\bnu$ is a partition (resp.\ a negative partition $(0\geq\nu_1\geq\ldots\geq\nu_N)$), then
$V^*_\bnu$ is an antipolynomial (resp.\ polynomial) representation of $\GL_N$
(a polynomial functor in $V^*,\ V$ respectively).

A word of apology for our weird
  convention is in order. The numbering of $\bG_\bO$-orbits in $\Gr$ such that the orbits of
  {\em sub}lattices $L\subset\bV_0$ are numbered by partitions goes back at least
  to~\cite{lu}. We choose the numbering such that the orbits of sublattices are numbered by
  negative partitions since under this numbering the adjacency order of $\bG_\bO$-orbits
  in $\Gr\times{\vphantom{j^{X^2}}\smash{\overset{\circ}{\vphantom{\rule{0pt}{0.55em}}\smash{\mathbf V}}}}$ goes to Shoji's order~\cite{s},~\cite[Proposition~12]{fgt} on the set of
  bisignatures. Furthermore, we choose the Satake equivalence $\IC_\bnu\mapsto V^*_\bnu$
  (as opposed to $\IC_\bnu\mapsto V_\bnu$) since it makes the statement of our main
  result~Theorem~\ref{main spherical mirabolic} more neat.

\subsection{Constructible mirabolic category and convolutions}
\label{conv constr}
The triangulated category
$D_{\bG_\bO}(\Gr\times{\vphantom{j^{X^2}}\smash{\overset{\circ}{\vphantom{\rule{0pt}{0.55em}}\smash{\mathbf V}}}})$
is defined as in~\cite[Section~2.6]{fgt}. We will denote it by $D_{*\bG_\bO}(\Gr\times\bV)$.
Recall that an object $\CF$ of $D_{*\bG_\bO}(\Gr\times\bV)$ is supported on $\Gr\times t^m\bV_0$
for certain $m\in\BZ$, and there exists $n>m$ and a $\bG_\bO$-equivariant sheaf $\CF_n$ on
$\Gr\times(t^m\bV_0/t^n\bV_0)$
such that $\CF=p_n^*\CF_n$, where \[p_n\colon\Gr\times t^m\bV_0\to\Gr\times(t^m\bV_0/t^n\bV_0)\]
is the natural projection. In other words, $\CF$ is a collection of $\bG_\bO$-equivariant
sheaves $\CF_{n'}$ on $\Gr\times(t^m\bV_0/t^{n'}\bV_0)$ for $n'\geq n$ along with a compatible
system of isomorphisms $p_{n''/n'}^*\CF_{n'}\iso\CF_{n''}$ for $n''\geq n'$, where
\[p_{n''/n'}\colon \Gr\times t^m\bV_0/t^{n''}\bV_0\twoheadrightarrow\Gr\times t^m\bV_0/t^{n'}\bV_0\]
are the natural projections.

If we replace in the above definition $p_n^*$ by $p_n^!$ and $p_{n''/n'}^*$ by $p_{n''/n'}^!$, then
we obtain a triangulated category $D_{!\bG_\bO}(\Gr\times\bV)$. Note that $p_{n''/n'}$ is a smooth
morphism of relative dimension $N(n''-n')$, so we have a canonical isomorphism
$p_{n''/n'}^!\cong p_{n''/n'}^*[2N(n''-n')]$. We also consider the intermediate version
\[p_{n''/n'}^{!*}:=p_{n''/n'}^*[N(n''-n')]=p_{n''/n'}^![-N(n''-n')],\] exact for the perverse $t$-structure.
The corresponding triangulated category will be denoted $D_{!*\bG_\bO}(\Gr\times\bV)$.

\bigskip

We will make use of an identification
$\Gr\times\bV=\bG_\bF\stackrel{\bG_\bO}{\times}\bV=(\bG_\bF\times\bV)/\bG_\bO$
(quotient with respect to the diagonal right-left
action). We will denote the orbit of $(g,v)\in\bG_\bF\times\bV$ by $[g,v]$.
Note that the left diagonal action of $\bG_\bO$ on $\Gr\times\bV$ in terms of the
above identification is $h\cdot[g,v]=[hg,v]$.
We consider the following convolution diagram:
\begin{equation}
  \begin{CD}
    \bG_\bF\times(\Gr\times\bV) @>{\bq}>>
    \bG_\bF\overset{\bG_\bO}{\times}(\Gr\times\bV)\\
    @V{\bp}VV @V{\bm}VV\\
  (\Gr\times\bV)\times(\Gr\times\bV) @.
    \Gr\times\bV,
    \end{CD}
\end{equation}
\[([g_1,g_2v],[g_2,v])\xleftarrow{\bp}(g_1,[g_2,v])\xrightarrow{\bq}[g_1,[g_2,v]]
\xrightarrow{\bm}[g_1g_2,v].\]
Given $\CF_1,\CF_2\in D_{!\bG_\bO}(\Gr\times\bV)$, we define
$\CF_1\srel!\oast\CF_2:=\bm_*(\CF_1\tbx^!\CF_2)$, where
$\CF_1\tbx^!\CF_2\in
D_{!\bG_\bO}\big(\bG_\bF\overset{\bG_\bO}{\times}(\Gr\times\bV)\big)$
is the canonical descent of $\bp^!(\CF_1\boxtimes\CF_2)$ along~$\bq$.
Similarly, given $\CF_1,\CF_2\in D_{*\bG_\bO}(\Gr\times\bV)$, we define
$\CF_1\srel*\oast\CF_2:=\bm_*(\CF_1\tbx^*\CF_2)$, where
$\CF_1\tbx^*\CF_2\in
D_{*\bG_\bO}\big(\bG_\bF\overset{\bG_\bO}{\times}(\Gr\times\bV)\big)$
is the canonical descent of $\bp^*(\CF_1\boxtimes\CF_2)$ along~$\bq$.
a unique sheaf such that $\bp^*(\CF_1\boxtimes\CF_2)=\bq^*(\CF_1\tbx^*\CF_2)$.

We also consider another convolution diagram
\begin{equation}
  \label{**}
  \begin{CD}
    \bG_\bF\times\bV\times(\Gr\times\bV) @>{\fq}>>
    \bG_\bF\overset{\bG_\bO}{\times}(\Gr\times\bV\times\bV)\\
    @V{\fp}VV @V{\fm}VV\\
  (\Gr\times\bV)\times(\Gr\times\bV) @.
    \Gr\times\bV,
    \end{CD}
\end{equation}
\[([g_1,v_1],[g_2,v_2])\xleftarrow{\fp}(g_1,v_1,[g_2,v_2])\xrightarrow{\fq}[g_1,[g_2,g_2^{-1}v_1,v_2]]
\xrightarrow{\fm}[g_1g_2,g_2^{-1}v_1+v_2].\]
Given $\CF_1,\CF_2\in D_{!\bG_\bO}(\Gr\times\bV)$, we define
$\CF_1\srel!*\CF_2:=\fm_!(\CF_1\tbx^!\CF_2)$, where
$\CF_1\tbx^!\CF_2\in
D_{!\bG_\bO}\big(\bG_\bF\overset{\bG_\bO}{\times}(\Gr\times\bV\times\bV)\big)$
is the canonical descent of $\fp^!(\CF_1\boxtimes\CF_2)$ along $\fq$.
Similarly, given $\CF_1,\CF_2\in D_{*\bG_\bO}(\Gr\times\bV)$, we define
$\CF_1\srel**\CF_2:=\fm_*(\CF_1\tbx^*\CF_2)$, where
$\CF_1\tbx^*\CF_2\in
D_{*\bG_\bO}\big(\bG_\bF\overset{\bG_\bO}{\times}(\Gr\times\bV\times\bV)\big)$
is the canonical descent of $\fp^*(\CF_1\boxtimes\CF_2)$ along $\fq$.

\bigskip

If we formally put $v_1=0$ (resp.\ $v_2=0$) in~(\ref{**}), we obtain the convolution diagrams
\begin{equation}
  \label{left conv}
  \begin{CD}
    \bG_\bF\times\Gr\times\bV @>{\fq}>>
    \bG_\bF\overset{\bG_\bO}{\times}(\Gr\times\bV)\\
    @V{\fp_\lef}VV @V{\fm}VV\\
  \Gr\times(\Gr\times\bV) @.
    \Gr\times\bV,
    \end{CD}
\end{equation}
\[([g_1],[g_2,v])\xleftarrow{\fp_\lef}(g_1,[g_2,v])\xrightarrow{\fq}[g_1,[g_2,v]]
\xrightarrow{\fm}[g_1g_2,v],\ \on{resp}.\]
\begin{equation}
  \label{right conv}
  \begin{CD}
    \bG_\bF\times\bV\times\Gr @>{\fq}>>
    \bG_\bF\overset{\bG_\bO}{\times}(\Gr\times\bV)\\
    @V{\fp_\righ}VV @V{\fm}VV\\
  (\Gr\times\bV)\times\Gr @.
    \Gr\times\bV,
    \end{CD}
\end{equation}
\[([g_1,v],[g_2])\xleftarrow{\fp_\righ}(g_1,v,[g_2])\xrightarrow{\fq}[g_1,[g_2,g_2^{-1}v]]
\xrightarrow{\fm}[g_1g_2,g_2^{-1}v].\]
Given $\CP\in D_{\bG_\bO}(\Gr)$ and $\CF\in D_{?\bG_\bO}(\Gr\times\bV)$ (where $?=!,*,!*$),
we define $\CP*\CF:=\fm_*(\CP\tbx\CF)$, where
$\CP\tbx\CF\in D_{?\bG_\bO}(\bG_\bF\overset{\bG_\bO}{\times}(\Gr\times\bV)\big)$ is the canonical
descent of $\fp_\lef^*(\CP\boxtimes\CF)$ along $\fq$. We also define $\CF*\CP:=\fm_*(\CF\tbx\CP)$,
where $\CF\tbx\CP\in D_{?\bG_\bO}(\bG_\bF\overset{\bG_\bO}{\times}(\Gr\times\bV)\big)$
is the canonical descent of $\fp_\righ^*(\CF\boxtimes\CP)$ along $\fq$. Both the left and right
convolutions are bi-exact for the perverse $t$-structures on $D_{\bG_\bO}(\Gr)$ and
$D_{!*\bG_\bO}(\Gr\times\bV)$, see~\cite[Section~3.9]{fgt}.

\subsection{Fusion}
Let $X$ be a smooth curve. For any integer $k>0$, and a collection $x=(x_i)_{i=1}^k$ of
$S$-points of $X$, we denote by $\CalD_x$ the formal neighborhood of the union of graphs
$|x| := \bigcup_{i=1}^k \Gamma_{x_i}\subset S\times X$, and we denote by
$\cD_x^\circ :=\CalD_x \setminus |x|$ the punctured formal neighborhood.
The mirabolic version of the Beilinson-Drinfeld Grassmannian is the
     ind-scheme $\Gr^\mir_{BD,k}$ over $X^k$ parametrizing the following collections of data:
\[ (x_i)_{i=1}^k,\ \CE,\ \phi\colon \CE_\triv|_{\CalD_x^\circ}\iso\CE|_{\CalD_x^\circ},\
v\in \Gamma(\CalD_x^\circ,\cE) , \]
where $\CE$ is a rank $N$ vector bundle on $\CalD_x$.
In case $X=\BA^1$, over the complement to the diagonals we have a canonical isomorphism
\[(\BA^k\setminus\Delta)\x_{\BA^k}\Gr^\mir_{BD,k}\cong(\BA^k\setminus\Delta)\times(\Gr\times\bV)^k.\] We denote the projection
$(\BA^k\setminus\Delta)\times(\Gr\times\bV)^k\to(\Gr\times\bV)^k$ by $\on{pr}_2$.
Given $\CF_1,\CF_2\in D_{?\bG_\bO}(\Gr\times\bV)$ (where $?={}!,*,!*$) we take $k=2$ and define
the fusion
\[\CF_1\star\CF_2:=\pr_{2*}\psi_{x-y}\on{pr}_2^*(\CF_1\boxtimes\CF_2)[1],\]
where $x,y$ are coordinates on $\BA^2$ (so that $x-y=0$ is the equation of the diagonal
$\Delta\subset\BA^2$), and $\psi_{x-y}$ is the nearby cycles functor for the pullback of the
function $x-y$ to $\Gr^\mir_{BD,2}$, normalized so as to preserve the perverse $t$-structure.
Note that the leftmost occurence of $\pr_2$ in the above definition projects
$\BA^1\times\Gr\times\bV$ to $\Gr\times\bV$, while the rightmost occurence of $\pr_2$ projects
$(\BA^2\setminus\Delta)\times(\Gr\times\bV)^2$ to $(\Gr\times\bV)^2$.

\subsection{Coherent mirabolic category and convolutions}
\label{conv coh}
We write $\Pi E$ for an odd vector space obtained from a vector space $E$ by reversing the parity.

We fix a pair of $N$-dimensional vector spaces $V_1\simeq\BC^N\simeq V_2$.
We consider the Lie superalgebra $\fg\fl(N|N)=\fg\fl(V_1\oplus\Pi V_2)$.
We have $\fg\fl(N|N)=\fg_{\bar0}\oplus\fg_{\bar1}$, where
$\fg_{\bar1}=\Pi\Hom(V_1,V_2)\oplus\Pi\Hom(V_2,V_1)$, and
$\fg_{\bar0}=\End(V_1)\oplus\End(V_2)$. We set $G_{\bar0}=\GL(V_1)\times\GL(V_2)$.
We consider the dg-algebra\footnote{We view
  $\fg_{\bar1}$ as an {\em odd} vector space, so that $\Sym(\fg_{\bar1}[-1])$
  (with grading disregarded) is really a symmetric (infinite-dimensional)
  algebra, not an exterior algebra.}
   $\fG_{1,1}^\bullet=\Sym(\fg_{\bar1}[-1])$ with zero differential,
and the triangulated category $D^{G_{\bar0}}_\perf(\fG_{1,1}^\bullet)$ obtained by localization
(with respect to quasi-isomorphisms) of the category of perfect $G_{\bar0}$-equivariant
dg-$\fG_{1,1}^\bullet$-modules.
The category $D^{G_{\bar0}}_\perf(\fG_{1,1}^\bullet)$ is monoidal with respect to
$\CM,\CM'\mapsto\CM\otimes_{\fG_{1,1}^\bullet}\CM'$, see~Section~\ref{super} below.

We will also need two more versions of $\fG_{1,1}^\bullet$, namely
\[\fG_{0,2}^\bullet=\Sym(\Hom(V_1,V_2))\otimes\Sym(\Hom(V_2,V_1)[-2]),\]
\[\fG_{2,0}^\bullet=\Sym(\Hom(V_1,V_2)[-2])\otimes\Sym(\Hom(V_2,V_1)),\] and the
corresponding triangulated categories $D^{G_{\bar0}}_\perf(\fG_{0,2}^\bullet)$ and
$D^{G_{\bar0}}_\perf(\fG_{2,0}^\bullet)$.
Now we will define the monoidal structures on $D^{G_{\bar0}}_\perf(\fG_{0,2}^\bullet)$ and
$D^{G_{\bar0}}_\perf(\fG_{2,0}^\bullet)$.

\bigskip

We consider the variety $\CQ^A$ (resp.\ $\CQ^B$)
of sixtuples \[A\in\Hom(V_1,V_2),\ B\in\Hom(V_2,V_1),\
A'\in\Hom(V'_1,V_2),\] \[B'\in\Hom(V_2,V'_1),\ A''\in\Hom(V_1,V'_1),\ B''\in\Hom(V'_1,V_1),\]
such that \[A=A'A'',\ B'=A''B,\ B''=BA'\ (\on{resp}.\ B=B''B',\ A'=AB'',\ A''=B'A)\] (here
$V'_1$ is a copy of $V_1$).
Clearly, $\CQ^A\simeq\Hom(V_1,V'_1)\times\Hom(V'_1,V_2)\times\Hom(V_2,V_1)$, and
$\CQ^B\simeq\Hom(V'_1,V_1)\times\Hom(V_2,V'_1)\times\Hom(V_1,V_2)$.
\begin{equation}
  \label{right half}
\vcenter{\xymatrix @C=3em @R=0.5em{
V_1\ar@<-0.25ex>@{-->}@/^/[dr]_A \ar@<-0.25ex>[dd]_{A''} \\
& V_2\ar@<-0.25ex>@/_/[ul]_B \ar@<-0.25ex>@{-->}@/^/[dl]_{B'\!} \\
V'_1\ar@<-0.25ex>@/_/[ur]_{A'} \ar@<-0.25ex>@{-->}[uu]_{B''}
}}
\end{equation}
We denote $\Hom(V_1,V_2)\times\Hom(V_2,V_1)\times\Hom(V'_1,V_2)\times\Hom(V_2,V'_1)\times
\Hom(V_1,V'_1)\times\Hom(V'_1,V_1)$ by $\CH$.
We have the natural projections
\[\pr_{12}\colon\CH\to\Hom(V_1,V_2)\times\Hom(V_2,V_1),\ \pr_{1'2}\colon\CH\to
\Hom(V'_1,V_2)\times\Hom(V_2,V'_1),\] \[\pr_{11'}\colon\CH\to\Hom(V_1,V'_1)\times\Hom(V'_1,V_1).\]
The group
$G_\CQ:=\GL(V_1)\times\GL(V'_1)\times\GL(V_2)$ naturally acts on $\CH$:
\[(g_1,g'_1,g_2)(A,A',A'',B,B',B'')=
(g_2Ag_1^{-1},g_2A'g'^{-1}_1,g'_1A''g_1^{-1},g_1Bg_2^{-1},g'_1B'g_2^{-1},g_1B''g'^{-1}_1).\]
The projections $\pr_{12},\pr_{1'2},\pr_{11'}$ are equivariant with respect to the same
named projections from $G_\CQ$ to $\GL(V_1)\times\GL(V_2),\ \GL(V'_1)\times\GL(V_2),\
\GL(V_1)\times\GL(V'_1)$.

Given $\CM_{1'2}\in\Coh^{\GL(V'_1)\times\GL(V_2)}\big(\Hom(V'_1,V_2)\times\Hom(V_2,V'_1)\big)$
and $\CM_{11'}\in\Coh^{\GL(V_1)\times\GL(V'_1)}\big(\Hom(V_1,V'_1)\times\Hom(V'_1,V_1)\big)$ we set
\begin{multline*}\CM_{11'}\srel{A}*\CM_{1'2}:=
  \pr_{12*}(\pr_{11'}^*\CM_{11'}\otimes_{\BC[\CH]}\BC[\CQ^A]\otimes_{\BC[\CH]}
  \pr_{1'2}^*\CM_{1'2})^{\GL(V'_1)}\in\\
  \Coh^{\GL(V_1)\times\GL(V_2)}\big(\Hom(V_1,V_2)\times\Hom(V_2,V_1)\big),
  \end{multline*}
  \begin{multline*}\CM_{11'}\srel{B}*\CM_{1'2}:=
  \pr_{12*}(\pr_{11'}^*\CM_{11'}\otimes_{\BC[\CH]}\BC[\CQ^B]\otimes_{\BC[\CH]}
  \pr_{1'2}^*\CM_{1'2})^{\GL(V'_1)}\in\\
  \Coh^{\GL(V_1)\times\GL(V_2)}\big(\Hom(V_1,V_2)\times\Hom(V_2,V_1)\big).
  \end{multline*}
We will actually need the following modifications of these functors:
\[\srel{B}*\colon D^{G_{\bar0}}_\perf(\fG_{0,2}^\bullet)\times D^{G_{\bar0}}_\perf(\fG_{0,2}^\bullet)\to
D^{G_{\bar0}}_\perf(\fG_{0,2}^\bullet),\
\srel{A}*\colon D^{G_{\bar0}}_\perf(\fG_{2,0}^\bullet)\times D^{G_{\bar0}}_\perf(\fG_{2,0}^\bullet)\to
D^{G_{\bar0}}_\perf(\fG_{2,0}^\bullet)\]
obtained using the dg-algebras with trivial differentials $\BC[\CH]_{0,2}^\bullet$ and
$\BC[\CH]_{2,0}^\bullet$ respectively, where
\begin{multline*}\BC[\CH]_{0,2}^\bullet=\Sym(\Hom(V_1,V_2))\otimes\Sym(\Hom(V_2,V_1)[-2])
  \otimes\Sym(\Hom(V'_1,V_2))\\
  \otimes\Sym(\Hom(V_2,V'_1)[-2])\otimes\Sym(\Hom(V_1,V'_1))\otimes\Sym(\Hom(V'_1,V_1)[-2]),
  \end{multline*}
\begin{multline*}\BC[\CH]_{2,0}^\bullet=\Sym(\Hom(V_1,V_2)[-2])\otimes\Sym(\Hom(V_2,V_1))
  \otimes\Sym(\Hom(V'_1,V_2)[-2])\\
  \otimes\Sym(\Hom(V_2,V'_1))\otimes\Sym(\Hom(V_1,V'_1)[-2])\otimes\Sym(\Hom(V'_1,V_1)).
\end{multline*}
(we identify $\BC[\Hom(U,W)]$ with $\Sym\Hom(W,U)$).

\subsection{Localization, coherent}
\label{loc coh}
We identify \[\Hom(V_1,V_2)=\Hom(V_2,V_1)^*,\ \Hom(V_2,V_1)=\Hom(V_1,V_2)^*,\] so that
\[\Sym(\Hom(V_1,V_2))=\BC[\Hom(V_2,V_1)],\ \Sym(\Hom(V_2,V_1))=\BC[\Hom(V_1,V_2)].\]
We have an open subvariety $\Isom(V_2,V_1)\subset\Hom(V_2,V_1)$, so that
$\BC[\Hom(V_2,V_1)]\subset\BC[\Isom(V_2,V_1)]$. We set
\[\fB^\bullet:=\BC[\Isom(V_2,V_1)]\otimes\Sym(\Hom(V_2,V_1)[-2])\] (a dg-algebra with
trivial differential).
Similarly, we define \[\fA^\bullet:=\Sym(\Hom(V_1,V_2)[-2])\otimes\BC[\Isom(V_1,V_2)].\]
An equivalent formulation of~\cite[Theorem~5]{bf} is an existence of a monoidal
equivalence $D^{G_{\bar0}}_\perf(\fA^\bullet)\cong D_{\bG_\bO}(\Gr)$ (and, changing the roles of
$V_1,V_2$, a monoidal equivalence $D^{G_{\bar0}}_\perf(\fB^\bullet)\cong D_{\bG_\bO}(\Gr)$).

Since $\fA^\bullet$ (resp.\ $\fB^\bullet$) is a localization of $\fG_{2,0}^\bullet$ (resp.\ of
$\fG_{0,2}^\bullet$), we have the restriction of scalars functors
\[\Res_A\colon D^{G_{\bar0}}_\perf(\fA^\bullet)\to\hat{D}{}^{G_{\bar0}}_\perf(\fG_{2,0}^\bullet),\
\Res_B\colon D^{G_{\bar0}}_\perf(\fB^\bullet)\to\hat{D}{}^{G_{\bar0}}_\perf(\fG_{0,2}^\bullet),\]
where $\hat{D}{}^{G_{\bar0}}_\perf(\fG_{?,?}^\bullet)$ stands for the Ind-completion of
$D^{G_{\bar0}}_\perf(\fG_{?,?}^\bullet)$.

However, one can check that for $\CN\in D^{G_{\bar0}}_\perf(\fA^\bullet)$ and
$\CM\in D^{G_{\bar0}}_\perf(\fG_{2,0}^\bullet)$ both convolutions $\Res_A(\CN)\srel{A}*\CM$
and $\CM\srel{A}*\Res_A(\CN)$ lie in
$D^{G_{\bar0}}_\perf(\fG_{2,0}^\bullet)\subset\hat{D}{}^{G_{\bar0}}_\perf(\fG_{2,0}^\bullet)$.
Thus we have the left and right convolution actions
\[\srel{A}*\colon D^{G_{\bar0}}_\perf(\fA^\bullet)\times D^{G_{\bar0}}_\perf(\fG_{2,0}^\bullet)\to
D^{G_{\bar0}}_\perf(\fG_{2,0}^\bullet),\ D^{G_{\bar0}}_\perf(\fG_{2,0}^\bullet)\times D^{G_{\bar0}}_\perf(\fA^\bullet)
\to D^{G_{\bar0}}_\perf(\fG_{2,0}^\bullet),\]
and similarly
\[\srel{B}*\colon D^{G_{\bar0}}_\perf(\fB^\bullet)\times D^{G_{\bar0}}_\perf(\fG_{0,2}^\bullet)\to
D^{G_{\bar0}}_\perf(\fG_{0,2}^\bullet),\ D^{G_{\bar0}}_\perf(\fG_{0,2}^\bullet)\times D^{G_{\bar0}}_\perf(\fB^\bullet)
\to D^{G_{\bar0}}_\perf(\fG_{0,2}^\bullet).\]

\subsection{Renormalizations}
\label{renorm}
The action of the centre $Z(\GL(V_1))\cong\BG_m$ on an object
$\CM\in D^{G_{\bar0}}_\perf(\fG_{1,1}^\bullet)$ defines a grading, and the corresponding degrees
will be denoted by $\deg_1$. Similarly, the action of $Z(\GL(V_2))\cong\BG_m$ gives rise to
another grading with degrees denoted by $\deg_2$. The cohomological degrees will be denoted
simply by $\deg$. Clearly, the degrees of the generators are as follows:
\begin{multline*}
  \deg_2(\Hom(V_1,V_2))=\deg_1(\Hom(V_2,V_1))=1,\\
  \deg_1(\Hom(V_1,V_2))=\deg_2(\Hom(V_1,V_2))=-1.
\end{multline*}
Hence changing cohomological degrees by the formula $\deg\leadsto\deg+\deg_1$,
resp.\ $\deg\leadsto\deg-\deg_2$, yields equivalences
\[D^{G_{\bar0}}_\perf(\fG_{2,0}^\bullet)\xrightarrow{\varrho_\righ}D^{G_{\bar0}}_\perf(\fG_{1,1}^\bullet)
\xrightarrow{\varrho_\righ}D^{G_{\bar0}}_\perf(\fG_{0,2}^\bullet),\]
\[D^{G_{\bar0}}_\perf(\fG_{2,0}^\bullet)\xrightarrow{\varrho_\lef}D^{G_{\bar0}}_\perf(\fG_{1,1}^\bullet)
\xrightarrow{\varrho_\lef}D^{G_{\bar0}}_\perf(\fG_{0,2}^\bullet),\]
respectively. The notation is due to the fact that $\varrho_\lef$ commutes with the left
action of $\Rep(\GL(V_1))$ on our categories, while $\varrho_\righ$ commutes with the
right action of $\Rep(\GL(V_2))$ on our categories.

Now recall the notation in the definition of categories $D_{?\bG_\bO}(\Gr\times\bV)$
(where $?={!},*,!*$) of~Section~\ref{setup}. Given $\CF=(\CF_n)_{n>m}\in D_{!\bG_\bO}(\Gr\times\bV)$
we define $\varrho_\righ\CF:=(\CF_n[-Nn])_{n>m}\in D_{!*\bG_\bO}(\Gr\times\bV)$. Similarly,
given $\CF=(\CF_n)_{n>m}\in D_{!*\bG_\bO}(\Gr\times\bV)$
we define $\varrho_\righ\CF:=(\CF_n[-Nn])_{n>m}\in D_{*\bG_\bO}(\Gr\times\bV)$.
The functors $\varrho_\righ$ commute with the action of the monoidal category
$D_{\bG_\bO}(\Gr)$ by the right convolutions.
Recall also that the affine Grassmannian is a union of connected components
$\Gr=\bigsqcup_{k\in\BZ}\Gr^{(k)}$, where $\Gr^{(k)}$ parametrizes the lattices of virtual
dimension $k$ (e.g.\ $\dim(t\bV_0)=-N$). For $\CF$ supported on $\Gr^{(k)}\times\bV$ we set
$\varrho_\lef(\CF):=\varrho_\righ(\CF)[-k]$. Then the functors \[\varrho_\lef\colon
D_{!\bG_\bO}(\Gr\times\bV)\to D_{!*\bG_\bO}(\Gr\times\bV),\ D_{!*\bG_\bO}(\Gr\times\bV)\to
D_{*\bG_\bO}(\Gr\times\bV)\] commute with the action of the monoidal category
$D_{\bG_\bO}(\Gr)$ by the left convolutions.

\bigskip

Our goal is the following

\begin{thm}
  \label{main spherical mirabolic}
  There exist monoidal \footnote{see Section~\ref{super} for the definition of the left middle
    monoidal structure.} equivalences of triangulated categories
  \begin{equation}
    \begin{CD}
      \big(D^{G_{\bar0}}_\perf(\fG_{2,0}^\bullet),\ \srel{A}*\big) @>{\sim}>{\Phi^{2,0}}>
      \big(D_{!\bG_\bO}(\Gr\times\oV),\ \srel!\oast\big)\\
      @V{\wr}V{\varrho_\righ}V @V{\wr}V{\varrho_\righ}V\\
      \big(D^{G_{\bar0}}_\perf(\fG_{1,1}^\bullet),\ \otimes_{\fG_{1,1}^\bullet}\big) @>{\sim}>{\Phi^{1,1}}>
      \big(D_{!*\bG_\bO}(\Gr\times\oV),\ \star\big)\\
      @V{\wr}V{\varrho_\righ}V @V{\wr}V{\varrho_\righ}V\\
      \big(D^{G_{\bar0}}_\perf(\fG_{0,2}^\bullet),\ \srel{B}*\big) @>{\sim}>{\Phi^{0,2}}>
      \big(D_{*\bG_\bO}(\Gr\times\oV),\ \srel**\big).
    \end{CD}
  \end{equation}
  (the vertical equivalences are {\em not} monoidal). The squares are commutative.
  The horizontal equivalences commute with the actions of the monoidal spherical
  Hecke category $\Perv_{\bG_\bO}(\Gr)\cong\Rep(\GL_N)$ by the left and right convolutions.
\end{thm}

The proof will be given in~Section~\ref{proof main} after a necessary preparation.

\subsection{Super}
\label{super}
Strictly speaking, in accordance with the footnote at the beginning of~Section~\ref{conv coh},
we should consider the category $SD^{G_{\bar0}}_\perf(\fG_{1,1}^\bullet)$ of {\em super} dg-modules
over the superalgebra $\fG_{1,1}^\bullet=\Sym(\fg_{\bar1}[-1])$. The latter superalgebra
is super-commutative, hence we have a symmetric monoidal structure
$\otimes_{\fG_{1,1}^\bullet}$ on the category $SD^{G_{\bar0}}_\perf(\fG_{1,1}^\bullet)$. The equivalence
$\Phi^{1,1}$ of~Theorem~\ref{main spherical mirabolic} can be upgraded to a monoidal equivalence
$S\Phi^{1,1}\colon \big(SD^{G_{\bar0}}_\perf(\fG_{1,1}^\bullet),\ \otimes_{\fG_{1,1}^\bullet}\big)\iso
\big(SD_{!*\bG_\bO}(\Gr\times\oV),\ \star\big)$ to the derived category of sheaves with
coefficients in {\em super} vector spaces.

However, the action of the central element $(\Id_{V_1},-\Id_{V_2})\in G_{\bar0}$ on an object of
$D^{G_{\bar0}}_\perf(\fG_{1,1}^\bullet)$ equips this object with an extra $\BZ/2\BZ$-grading, and thus
defines a fully faithful functor
$D^{G_{\bar0}}_\perf(\fG_{1,1}^\bullet)\to SD^{G_{\bar0}}_\perf(\fG_{1,1}^\bullet)$ of a ``superization'',
such that its essential image is
closed under the monoidal structure $\otimes_{\fG_{1,1}^\bullet}$. This defines the desired
monoidal structure $\otimes_{\fG_{1,1}^\bullet}$ on the category $D^{G_{\bar0}}_\perf(\fG_{1,1}^\bullet)$.

\subsection{Koszul equivalence}
\label{koszul}
We consider the following complex $H^\bullet$ of {\em odd} vector spaces living in degrees $0,1\colon
\fg_{\bar1}\stackrel{\Id}{\longrightarrow}\fg_{\bar1}$.
We define the Koszul complex $K^\bullet$ as the symmetric algebra $\Sym(H^\bullet)$. The degree
zero part \[K^0=\Lambda\big(\Hom(V_1,V_2)\oplus\Hom(V_2,V_1)\big)=:\Lambda\] (as a vector space,
with a super-structure disregarded). We turn $K^\bullet$ into a  dg-$\fG^\bullet_{1,1}-\Lambda$-bimodule
by letting $\fG^\bullet_{1,1}$ act by multiplication, and $\Lambda$ by differentiation.
Note that $K^\blt$ is quasi-isomorphic to $\BC$ in degree 0 as a complex of vector spaces,
but {\em not} as a dg-$\fG^\bullet_{1,1}-\Lambda$-bimodule.
We consider the derived category $D_\fid^{G_{\bar0}}(\Lambda)$ of finite dimensional complexes of
$G_{\bar0}\ltimes\Lambda$-modules. If we remember the super-structure of $\Lambda$, we obtain
the corresponding category of super dg-modules $SD_\fid^{G_{\bar0}}(\Lambda)$.
We have the Koszul equivalence functors
\[\varkappa\colon D_\fid^{G_{\bar0}}(\Lambda)\iso D^{G_{\bar0}}_\perf(\fG_{1,1}^\bullet),\
SD_\fid^{G_{\bar0}}(\Lambda)\iso SD^{G_{\bar0}}_\perf(\fG_{1,1}^\bullet),\
\CN\mapsto K^\bullet\otimes_\Lambda\CN.\]
Here is an equivalent definition of the category $SD_\fid^{G_{\bar0}}(\Lambda)$. We consider
the following degeneration $\ul\fgl(N|N)$ of the Lie superalgebra $\fgl({N|N})$:
the supercommutator of the even elements (with even or odd elements) remains intact, but the
supercommutator of any two odd elements is set to be zero. Let $SD_{\on{int}}(\ul\fgl(N|N))$
denote the derived category of bounded complexes of integrable $\ul\fgl(N|N)$-modules
(note that the even part of $\ul\fgl(N|N)$ is just $\fg_{\bar0}$, and the integrability is
nothing but $\fg_{\bar0}$-integrability, i.e.\ $G_{\bar0}$-equivariance). Then
$SD_{\on{int}}(\ul\fgl(N|N))\cong SD_\fid^{G_{\bar0}}(\Lambda)$ tautologically.
The resulting Koszul equivalence
$\varkappa\colon SD_{\on{int}}(\ul\fgl(N|N))\iso SD^{G_{\bar0}}_\perf(\fG_{1,1}^\bullet)$ is monoidal
with respect to the usual tensor structure on the LHS and $\otimes_{\fG^\bullet_{1,1}}$ on the RHS.

As in~Section~\ref{super}, the action of $(\Id_{V_1},-\Id_{V_2})\in G_{\bar0}$ gives rise to a
``superization'' fully faithful functor $D_\fid^{G_{\bar0}}(\Lambda)\to SD_\fid^{G_{\bar0}}(\Lambda)\cong
SD_{\on{int}}(\ul\fgl(N|N))$ with the essential image closed under the tensor structure.
This defines the tensor structure on $D_\fid^{G_{\bar0}}(\Lambda)$ such that the Koszul
equivalence $\varkappa\colon D_\fid^{G_{\bar0}}(\Lambda)\iso D^{G_{\bar0}}_\perf(\fG_{1,1}^\bullet)$ is
monoidal.

\begin{cor}[of Theorem~\ref{main spherical mirabolic}]
  \label{exact spherical}
  \textup{(a)} The composed equivalence
  \[\Phi^{1,1}\circ\varkappa\colon
  D_\fid^{G_{\bar0}}(\Lambda)\iso D_{!*\bG_\bO}(\Gr\times\oV)\]
  is exact with respect to the tautological $t$-structure on
  $D_\fid^{G_{\bar0}}(\Lambda)$
  and the perverse $t$-structure on $D_{!*\bG_\bO}(\Gr\times\oV)$.

  \textup{(b)} This equivalence is monoidal
  with respect to the tensor structure on 
  $D_\fid^{G_{\bar0}}(\Lambda)$ and the fusion $\star$
  on $D_{!*\bG_\bO}(\Gr\times\oV)$.

  \textup{(c)} The equivariant derived category $D_{!*\bG_\bO}(\Gr\times\bV)$ is equivalent to the
  bounded derived category of the abelian category $\Perv_{\bG_\bO}(\Gr\times\bV)$.
\end{cor}
\begin{proof}
  We consider an irreducible $G_{\bar0}$-module $V_{1,\blambda}\otimes V_{2,\bmu}$ as a
  $G_{\bar0}\ltimes\Lambda$-module with the trivial action of $\Lambda$. Then
  $\varkappa(V_{1,\blambda}\otimes V_{2,\bmu})=V_{1,\blambda}\otimes\fG_{1,1}^\bullet\otimes V_{2,\bmu}$,
  and $\Phi^{1,1}(V_{1,\blambda}\otimes\fG_{1,1}^\bullet\otimes V_{2,\bmu})=\IC_{(\blambda^*,\bmu^*)}$
  by construction of $\Phi^{1,1}$. Here for a signature $\bnu=(\nu_1\geq\ldots\geq\nu_N)$ we set
  $\bnu^*:=(-\nu_N\geq\ldots\geq-\nu_1)$.

  Finally, since $D_\fid^{G_{\bar0}}(\Lambda)$ is equivalent to the bounded derived category of
  its heart $\Rep(G_{\bar0}\ltimes\Lambda)$, (c) follows from (a).
  \end{proof}

\subsection{Deequivariantized Ext algebra}
\label{deeq ext}
Recall from~\cite[Proposition~8]{fgt} that the $\bG_\bO$-orbits in $\Gr\times{\vphantom{j^{X^2}}\smash{\overset{\circ}{\vphantom{\rule{0pt}{0.55em}}\smash{\mathbf V}}}}$
are numbered by bisignatures $(\blambda,\bmu)$ where both $\blambda$ and $\bmu$ have
length $N$. The IC-extension of the constant 1-dimensional local system on
such an orbit is denoted by $\IC_{(\blambda,\bmu)}$. In particular, $\IC_{(0^N,0^N)}$ is the
constant sheaf on $\Gr^0\times\bV_0$, to be denoted by $E_0$ for short.
Also recall from~\cite[Section~3]{fgt} that the left and right actions of the monoidal Satake
category $\Perv_{\bG_\bO}(\Gr)\cong\Rep(\GL_N)$ on $D_{!*\bG_\bO}(\Gr\times\oV)$ by convolutions
respect the perverse $t$-structure with the heart
$\Perv_{\bG_\bO}(\Gr\times\oV)\subset D_{!*\bG_\bO}(\Gr\times\oV)$.
As has been mentioned in~Section~\ref{renorm}, the right actions of $D_{\bG_\bO}(\Gr)$ on
$D_{?\bG_\bO}(\Gr\times\bV)$ commute with the equivalences $\varrho_\righ$, but
the left actions only commute with $\varrho_\righ$ up to cohomological shifts depending
on the connected components of $\Gr$.

We restrict the left and right actions of $D_{\bG_\bO}(\Gr)$ on
$D_{?\bG_\bO}(\Gr\times\bV)$ to the left and right actions of $\Perv_{\bG_\bO}(\Gr)\cong\Rep(\GL_N)$.
Thus we obtain the action of $\Rep(\sG)$ for $\sG=\GL_N\times\GL_N$. 
Let $D^{\deeq}_{?\bG_\bO}(\Gr\times\oV)$ denote the corresponding de\-equivariantized
category (see~\cite{ag} in the setting of abelian categories and~\cite{1-aff} in the setting
of dg-categories). Recall that to construct $D^{\deeq}_{?\bG_\bO}(\Gr\times\oV)$, first we consider
objects of the ind-completion of $D_{?\bG_\bO}(\Gr\times\bV)$ endowed with an action of $R_\sG$
where $R_\sG$ is the regular representation of $\sG$ considered as a ring-object in $\Rep(\sG)$.
Then we take compact objects of the resulting category. More explicitly,
an object of $D^{\deeq}_{?\bG_\bO}(\Gr\times\oV)$ is an object $\CF$ of $D_{?\bG_\bO}(\Gr\times\bV)$
together with a system of isomorphisms $\IC_\blambda*\CF*\IC_\bmu\iso V^*_\blambda\otimes V^*_\bmu\otimes\CF$
(recall that the geometric Satake equivalence takes $\IC_\blambda$ to $V^*_\blambda$) for any
signatures $\blambda,\bmu$, satisfying some natural compatibilities with respect to direct sums and
tensor products. By definition we have a natural forgetful functor
$D^{\deeq}_{?\bG_\bO}(\Gr\times\oV)\to D_{?\bG_\bO}(\Gr\times\bV)$. This functor admits a left adjoint that
sends an object $\CF$ to
$R_\sG*\CF=\bigoplus_{\blambda,\bmu}V_\blambda\otimes V_\bmu\otimes\IC_\blambda*\CF*\IC_\bmu$.

Thus, given $\CF_1,\CF_2\in D_{?\bG_\bO}(\Gr\times\bV)$, we denote the corresponding objects of the
deequivariantized category by the same symbols, and we have
\begin{equation}
  \label{def deeq}
  \RHom_{D^{\deeq}_{?\bG_\bO}(\Gr\times\ooV)}(\CF_1,\CF_2)=
\bigoplus_{\blambda,\bmu}\RHom_{D_{?\bG_\bO}(\Gr\times\ooV)}(\CF_1,\IC_\blambda*\CF_2*\IC_\bmu)\otimes
V_\blambda\otimes V_\bmu.
\end{equation}

\begin{lem}
  \label{formality}
  The dg-algebra $\RHom_{D^{\deeq}_{!*\bG_\bO}(\Gr\times\ooV)}(E_0,E_0)$ is formal, i.e.\ it is
  quasiisomorphic to the graded algebra $\Ext^\bullet_{D^{\deeq}_{!*\bG_\bO}(\Gr\times\ooV)}(E_0,E_0)$ with
  trivial differential.
\end{lem}

\begin{proof}
  We change the setting to the base field $\BF_q$. Then a $\bG_\bO$-equivariant
  irreducible perverse sheaf $\IC_{(\blambda,\bmu)}$ on $\Gr\times\bV$ carries a natural Weil
  structure, see~\cite[Section~2.6]{fgt}, and
  $\Ext^\bullet_{D_{!*\bG_\bO}(\Gr\times\ooV)}(E_0,\IC_{(\blambda,\bmu)})$
  is pure. Indeed, if $\bar\imath_0$ denotes the closed embedding
  $\bV_0\cong\Gr^0\times\bV_0\hookrightarrow\Gr\times\bV$, then
  $\Ext^\bullet_{D_{!*\bG_\bO}(\Gr\times\ooV)}(E_0,\IC_{(\blambda,\bmu)})=
  H^\bullet_{\bG_\bO}(\bV_0,\bar\imath_0^!\IC_{(\blambda,\bmu)})$.
  Let us consider the loop rotation action $\BG_m\curvearrowright(\Gr\times\bV)$.
  We have an embedding $\BA^1\hookrightarrow\bV$ (constant Laurent series), and
  $\Gr^0\times\BA^1$ is a fixed point component of the $\BG_m$-action, so that
  $H^\bullet_{\bG_\bO}(\bV_0,\bar\imath_0^!\IC_{(\blambda,\bmu)})$ is the hyperbolic restriction to this
  component (more precisely, the hyperbolic restriction is a geometrically
  constant complex on $\BA^1$ with the above stalks)~\cite{bra,dg}. But the
  hyperbolic restriction preserves purity.

  Now given the purity of $\Ext^\bullet_{D_{!*\bG_\bO}(\Gr\times\ooV)}(E_0,\IC_{(\blambda,\bmu)})$, the 
  desired result follows by an application of~\cite[Lemma~15]{bf}
  (and then of~\cite[Proposition~5]{bf}).
  \end{proof}

We denote the dg-algebra
$\Ext^\bullet_{D^{\deeq}_{!*\bG_\bO}(\Gr\times\ooV)}(E_0,E_0)$ (with trivial differential)
by $\fE^\bullet$. Since it is an Ext-algebra in the deequivariantized category between objects
induced from the original category, it is automatically
equipped with an action of $\GL_N\times\GL_N=\GL(V_1)\times\GL(V_2)=G_{\bar0}$, and we can
consider the corresponding triangulated category $D^{G_{\bar0}}_\perf(\fE^\bullet)$.

\begin{lem}
  \label{purity}
There is a canonical equivalence $D^{G_{\bar0}}_\perf(\fE^\bullet)\iso D_{!*\bG_\bO}(\Gr\times\oV)$.
\end{lem}

\begin{proof}
  The desired functor is constructed as in~\cite[Section~6.5,~Propositions~5,6]{bf}.
  Since $E_0$ generates the triangulated category $D_{!*\bG_\bO}(\Gr\times\oV)$ (with respect
  to the left and right actions of the Satake category), the claim follows
  from~Lemma~\ref{formality}.
\end{proof}

We also consider the left and right actions of the monoidal Satake category
$\Perv_{\bG_\bO}(\Gr)\cong\Rep(\GL_N)$ on $D_{\bG_\bO}(\Gr)$ by convolutions.
Let $D^{\deeq}_{\bG_\bO}(\Gr)$ denote the corresponding deequivariantized
category. Then the dg-algebra $\RHom_{D^{\deeq}_{\bG_\bO}(\Gr)}(\IC_0,\IC_0)$ is formal, i.e.\ it is
quasiisomorphic to the graded algebra $\Ext^\bullet_{D^{\deeq}_{\bG_\bO}(\Gr)}(\IC_0,\IC_0)$ with
trivial differential. Furthermore, it follows from~\cite[Theorem~5]{bf} that there is a
natural isomorphism $\Ext^\bullet_{D^{\deeq}_{\bG_\bO}(\Gr)}(\IC_0,\IC_0)\cong\fA^\bullet$
(notation of~Section~\ref{loc coh}).

\subsection{Localization, constructible}
\label{localization}
We have an automorphism
$\alpha\colon \Gr\times\oV\iso\Gr\times\oV,\ (L,v)\mapsto (L,tv)$.
We have a morphism of endofunctors
$\alpha^*\to\Id\colon D_{*\bG_\bO}(\Gr\times\bV)\to D_{*\bG_\bO}(\Gr\times\bV)$
constructed as follows. We consider a family of automorphisms
$\boldsymbol{\alpha}\colon \BA^1\times\Gr\times\oV\hookrightarrow\Gr\times\oV,\
(c,L,v)\mapsto (L,(c+t)v)$,
so that $\boldsymbol{\alpha}_0=\alpha$. Note that $\boldsymbol{\alpha}_c^*\cong\Id$
on $D_{*\bG_\bO}(\Gr\times\bV)$ for  $c\ne0$.
Now the desired morphism $\alpha^*\to\Id$ is just the cospecialization\footnote{terminology
of~\cite[6.2.7]{sch}} morphism from the stalk at $0\in\BA^1$ to the nearby stalk.
For $\CF_1,\CF_2\in D_{*\bG_\bO}(\Gr\times\oV)$ we have an inductive system
\begin{equation*}\ldots\to\RHom_{D_{*\bG_\bO}(\Gr\times\ooV)}((\alpha^{n-1})^*\CF_1,\CF_2)\to
  \RHom_{D_{*\bG_\bO}(\Gr\times\ooV)}((\alpha^n)^*\CF_1,\CF_2)\to\ldots
\end{equation*}
Note that it stabilizes since by definition of $D_{*\bG_\bO}(\Gr\times\oV)$, the restriction
of $(\alpha^n)^*\CF_1$ to the support of $\CF_2$ becomes the pullback of an
appropriate sheaf in $D_{\bG_\bO}(\Gr)$ for $n\gg0$.

We define the localized category $D^\loc_{*\bG_\bO}(\Gr\times\oV)$ as the category with the same
objects as $D_{*\bG_\bO}(\Gr\times\oV)$, and with morphisms
\begin{equation*}\RHom_{D^\loc_{*\bG_\bO}(\Gr\times\ooV)}(\CF_1,\CF_2):=
  \varinjlim\RHom_{D_{*\bG_\bO}(\Gr\times\ooV)}((\alpha^n)^*\CF_1,\CF_2).
\end{equation*}
The tautological functor $D_{*\bG_\bO}(\Gr\times\oV)\to D^\loc_{*\bG_\bO}(\Gr\times\oV)$
is denoted $u_0^*$ (``restriction to $v\thickapprox0$'').

\bigskip

We also have the Verdier dual (to the above $\alpha^*\to\Id$) morphism of endofunctors
$\Id\to\alpha^!
\colon D_{!\bG_\bO}(\Gr\times\bV)\to D_{!\bG_\bO}(\Gr\times\bV)$.
It gives rise to an inductive system for
$\CF_1,\CF_2\in D_{!\bG_\bO}(\Gr\times\oV)$
\begin{multline*}\ldots\to\RHom_{D_{!\bG_\bO}(\Gr\times\ooV)}(\CF_1,(\alpha^{n-1})^!\CF_2)\\
\to\RHom_{D_{!\bG_\bO}(\Gr\times\ooV)}(\CF_1,(\alpha^n)^!\CF_2)\to\ldots
\end{multline*}
stabilizing for the reasons similar to above.
We define the localized category $D^{\loc}_{!\bG_\bO}(\Gr\times\oV)$ as the category with the same
objects as $D_{!\bG_\bO}(\Gr\times\oV)$, and with morphisms
\begin{equation*}\RHom_{D^{\loc}_{!\bG_\bO}(\Gr\times\ooV)}(\CF_1,\CF_2):=
  \varinjlim\RHom_{D_{!\bG_\bO}(\Gr\times\ooV)}(\CF_1,(\alpha^n)^!\CF_2).
\end{equation*}
The tautological functor $D_{!\bG_\bO}(\Gr\times\oV)\to D^{\loc}_{!\bG_\bO}(\Gr\times\oV)$
is denoted $u_0^!$ (``corestriction to $v\thickapprox0$'').

We also have the projection $\pr\colon\Gr\times\bV_0\to\Gr$, and the corresponding
pullbacks $\pr^*\colon D_{\bG_\bO}(\Gr)\to D_{*\bG_\bO}(\Gr\times\oV),\
\pr^!\colon D_{\bG_\bO}(\Gr)\to D_{!\bG_\bO}(\Gr\times\oV)$ (with the essential images
supported on $\Gr\times\bV_0\subset\Gr\times\bV$).

\begin{lem}
  \label{restriction to zero}
  The compositions
  \begin{equation*}u_0^*\circ\pr^*\colon D_{\bG_\bO}(\Gr)\to D^\loc_{*\bG_\bO}(\Gr\times\oV),\
  u_0^!\circ\pr^!\colon D_{\bG_\bO}(\Gr)\to D^{\loc}_{!\bG_\bO}(\Gr\times\oV)\end{equation*}
  are equivalences of categories sending $\IC_0$ to $\varrho_\righ E_0,\ \varrho_\righ^{-1}E_0$
  respectively.
\end{lem}

\begin{proof}
  Clear.
\end{proof}

The localizations of the deequivariantized categories
$D^{\deeq}_{*\bG_\bO}(\Gr\times\oV)$, $D^{\deeq}_{!\bG_\bO}(\Gr\times\oV)$ will be denoted
  $D^{\loc,\deeq}_{*\bG_\bO}(\Gr\times\oV)$ and
  $D^{\loc,\deeq}_{!\bG_\bO}(\Gr\times\oV)$ respectively.

  Recall the dg-algebras $\fA^\bullet,\fB^\bullet$ introduced in Section~\ref{loc coh}.

\begin{cor}
  \label{localized ext}
  There are canonical isomorphisms

  \textup{(a)} $\Ext^\bullet_{D^{\loc,\deeq}_{!\bG_\bO}(\Gr\times\ooV)}(\varrho_\righ^{-1}E_0,
  \varrho_{\righ}^{-1}E_0)\cong\fA^\bullet,$

  \textup{(b)} $\Ext^\bullet_{D^{\loc,\deeq}_{*\bG_\bO}(\Gr\times\ooV)}(\varrho_\righ E_0,
  \varrho_\righ E_0)\cong\fB^\bullet$.\hfill{$\Box$}
\end{cor}

Note that the dg-algebra $\RHom_{D^{\deeq}_{*\bG_\bO}(\Gr\times\ooV)}(\varrho_\righ E_0,\varrho_\righ E_0)$
(resp.\ $\RHom_{D^{\deeq}_{!\bG_\bO}(\Gr\times\ooV)}(\varrho_\righ^{-1}E_0,\varrho_\righ^{-1}E_0)$)
is also formal, i.e.\ it is quasiisomorphic to
$\Ext^\bullet_{D^{\deeq}_{*\bG_\bO}(\Gr\times\ooV)}(\varrho_\righ E_0,\varrho_\righ E_0)$
(resp.\ $\Ext^\bullet_{D^{\deeq}_{!\bG_\bO}(\Gr\times\ooV)}(\varrho_\righ^{-1}E_0,\varrho_\righ^{-1}E_0)$)
with trivial differential. If we disregard their gradings, they are both isomorphic to
$\Ext^\bullet_{D^{\deeq}_{!*\bG_\bO}(\Gr\times\ooV)}(E_0,E_0)$.
We will denote the algebra $\Ext^\bullet_{D^{\deeq}_{!*\bG_\bO}(\Gr\times\ooV)}(E_0,E_0)$
with grading forgotten by $\Ext_{D^{\deeq}_{!*\bG_\bO}(\Gr\times\ooV)}(E_0,E_0)$; the same
applies to $\fA^\bullet,\fB^\bullet,\fE^\bullet,\fG^\bullet$.
Thus we have equalities
$\Ext_{D^{\deeq}_{*\bG_\bO}(\Gr\times\ooV)}(\varrho_\righ E_0,\varrho_\righ E_0)=
\Ext_{D^{\deeq}_{!*\bG_\bO}(\Gr\times\ooV)}(E_0,E_0)=
\Ext_{D^{\deeq}_{!\bG_\bO}(\Gr\times\ooV)}(\varrho_\righ^{-1}E_0,\varrho_\righ^{-1}E_0)$.

\begin{lem}
  \label{injective}
  The natural morphism
  \begin{equation*}\Ext^\bullet_{D^{\deeq}_{!\bG_\bO}(\Gr\times\ooV)}(\varrho_\righ^{-1}E_0,\varrho_\righ^{-1}E_0)
    \to\Ext^\bullet_{D^{\loc,\deeq}_{!\bG_\bO}(\Gr\times\ooV)}(\varrho_\righ^{-1}E_0,\varrho_\righ^{-1}E_0)
  \end{equation*}
  is injective.
\end{lem}

\begin{proof}
  We have $\alpha^!\varrho_\righ^{-1}\IC_{(-1^N,1^N)}[N]\simeq\varrho_\righ^{-1}E_0$,
  and we need to check that for any $\blambda,\bmu$
  the natural morphism
  \begin{multline}
    \label{inj}
    \vartheta\colon\RHom_{D_{!\bG_\bO}(\Gr\times\ooV)}(\varrho_\righ^{-1}E_0,
    \varrho_\righ^{-1}\IC_{(\blambda,\bmu)})\\
    \to\RHom_{D_{!\bG_\bO}(\Gr\times\ooV)}(\varrho_\righ^{-1}\IC_{(-1^N,1^N)}[-N],
    \varrho_\righ^{-1}\IC_{(\blambda,\bmu)})
  \end{multline}
  is injective.
  Indeed, since $\alpha^!$ is a Hecke transformation endofunctor (namely, it is given by
  convolution $\IC_{(1^N)}*?*\IC_{(-1)^N}$ with invertible objects of the Satake category),
  the morphism $\vartheta$ is an endomorphism of the identity endofunctor of the deequivariantized
  category $D^{\deeq}_{!\bG_\bO}(\Gr\times\ooV)$. In other words, $\vartheta$ is an element of the
  center of this category. But the natural morphism from the center to its localization is
  injective iff the multiplication by $\vartheta$ is injective.
  
  We change the setting to the base field $\BF_q$ as in the proof of~Lemma~\ref{formality}.
  Then all the IC sheaves in question carry a natural Weil structure, and it was proved
  in {\em loc.\ cit.} that both the LHS and the RHS of~(\ref{inj}) are pure; it is immediate
  to see that they are pure of the same weight $w$. The cone of $\vartheta$ is
  $H^\bullet_{\bG_\bO}(\Omega_{(0^N,0^N)},\imath_0^!\varrho_\righ^{-1}\IC_{(\blambda,\bmu)})$
  where $\imath_0$ stands for the
  locally closed embedding of the $\bG_\bO$-orbit
  $\Omega_{(0^N,0^N)}=\Gr^0\times(\bV_0\setminus t\bV_0)\hookrightarrow\Gr\times\bV$.
  Due to the pointwise purity of $\IC_{(\blambda,\bmu)}$~\cite[Section~3]{fgt},
  $\imath_0^!\IC_{(\blambda,\bmu)}$ is a pure local system on $\Omega_{(0^N,0^N)}$; hence
  $H^\bullet_{\bG_\bO}(\Omega_{(0^N,0^N)},\imath_0^!\IC_{(\blambda,\bmu)})$ is also pure of weight $w$.
  It follows that the kernel of $\vartheta$ vanishes, and $\vartheta$ is injective.
\end{proof}

\begin{cor}\label{comm-int}
  The algebra $\fE=\Ext_{D^{\deeq}_{!*\bG_\bO}(\Gr\times\ooV)}(E_0,E_0)$ is a commutative integrally closed
  domain.
\end{cor}

\subsection{Calculation of the Ext algebra}
\label{calculation}
Recall that the first fundamental coweight of $\GL_N$ is $\omega_1=(1,0,\ldots,0)$, and
$\omega_1^*=(0,\ldots,0,-1)$. We have
$\RHom_{D_{!*\bG_\bO}(\Gr\times\ooV)}(\IC_{\omega_1}*E_0,E_0*\IC_{\omega_1})=
\RHom_{D_{!*\bG_\bO}(\Gr\times\ooV)}(E_0,\IC_{\omega_1^*}*E_0*\IC_{\omega_1})$,
$\RHom_{D_{!*\bG_\bO}(\Gr\times\ooV)}(E_0*\IC_{\omega_1},\IC_{\omega_1}*E_0)=
\RHom_{D_{!*\bG_\bO}(\Gr\times\ooV)}(E_0,\IC_{\omega_1}*E_0*\IC_{\omega_1^*})$.
Now $\IC_{\omega_1}*E_0$ is the constant IC-sheaf of the stratum closure formed by all the pairs
$(L,v)$ such that the lattice $L$ contains $\bV_0$ as a hyperplane, and $v\in L$.
Furthermore, $E_0*\IC_{\omega_1}$ is the constant IC-sheaf of the stratum closure formed by all
the pairs $(L,v)$ such that the lattice $L$ contains $\bV_0$ as a hyperplane, and $v\in\bV_0$.
In particular, the latter stratum closure is a smooth divisor in the former stratum closure,
so we have canonical elements $h\in\Ext^1_{D_{!*\bG_\bO}(\Gr\times\ooV)}(\IC_{\omega_1}*E_0,E_0*\IC_{\omega_1})$
and $h^*\in\Ext^1_{D_{!*\bG_\bO}(\Gr\times\ooV)}(E_0*\IC_{\omega_1},\IC_{\omega_1}*E_0)$.
Hence we obtain the subspaces
$h\otimes V_{\omega_1^*}\otimes V_{\omega_1}\subset\fE^1$ and
$h^*\otimes V_{\omega_1}\otimes V_{\omega_1^*}\subset\fE^1$ (see~Section~\ref{deeq ext} for
the definition of $\fE^\bullet$ and~(\ref{def deeq})). We identify the former subspace with
$\Hom(V_1,V_2)$ and the latter one with $\Hom(V_2,V_1)$. Thus we obtain a homomorphism
$\phi^\bullet\colon\Sym\big(\Pi\Hom(V_1,V_2)[-1]\oplus\Pi\Hom(V_2,V_1)[-1]\big)=
\Sym(\fg_{\bar1}[-1])=\fG_{1,1}^\bullet\to\fE^\bullet$
(due to commutativity of $\fE$).

\begin{lem}
  \label{phi isom}
  $\phi^\bullet$ is an isomorphism.
\end{lem}

\begin{proof}
  We can and will disregard the grading. The morphism $\phi$ induces the morphism
  $\phi^*\colon\Spec\fE\to\fg^*_{\bar1}$ that is an isomorphism over the open subset
  $\Isom(V_2,V_1)\times\Hom(V_1,V_2)\subset\fg^*_{\bar1}$ due to~Corollary~\ref{localized ext}(a).
  Similarly, $\phi^*$ is an isomorphism over the open subset
  $\Hom(V_2,V_1)\times\Isom(V_1,V_2)$ due to~Corollary~\ref{localized ext}(b).

  Since the complement to the union of these two open subsets has codimension~2 in
  $\fg_{\bar1}^*$, we can apply~Lemma~\ref{Roma's} below.
  Note that the irreducibility of $\Spec\fE$ is guaranteed by Corollary~\ref{comm-int}.
  It remains only to check that the ratio of the above isomorphisms
  is the identity birational isomorphism between the varieties $\Isom(V_2,V_1)\times\Hom(V_1,V_2)$ and
  $\Hom(V_2,V_1)\times\Isom(V_1,V_2)$.

  The composition $h\circ h^*\in\Ext^2_{D_{!*\bG_\bO}(\Gr\times\ooV)}(E_0*\IC_{\omega_1},E_0*\IC_{\omega_1})$
  is the multiplication by the first Chern class of the normal line bundle $\CN$ to the divisor
  $\on{supp}(E_0*\IC_{\omega_1})$ in $\on{supp}(\IC_{\omega_1}*E_0)$. Recall that
  $\Gr^{\omega_1}\simeq\BP^{N-1}$. The line bundle $\CN$ is pulled back from the line bundle
  $\CO(1)$ on $\BP^{N-1}\simeq\Gr^{\omega_1}$. Recall also that the restriction of the determinant
  line bundle $\CL$ from $\Gr$ to $\Gr^{\omega_1}$ is also isomorphic to $\CO(1)$.
  We conclude that $h\circ h^*=c_1(\CL)$.

  On the other hand, in the equivariant Satake category
  \[D_{\bG_\bO}(\Gr)\cong D_\perf^{\GL_N}(\Sym(\fg\fl_N[-2]))\cong D_\perf^{G_{\bar0}}(\fA^\bullet),\]
  the first Chern class
  $c_1(\CL)\in\Ext^2_{D_{\bG_\bO}(\Gr)}(\IC_0*\IC_{\omega_1},\IC_0*\IC_{\omega_1})\subset\fA^2$
  corresponds to 
  the identity element (shifted by 2) $\Id\in\Hom(V_2,V_1)^*\otimes\Hom(V_2,V_1)$.

The lemma is proved.
\end{proof}

\begin{lem}
  \label{Roma's}
Let $\pi\colon X\to\BA^n$ be a morphism from an irreducible affine algebraic variety to an affine space.
  Let $f,g\in \BC[\BA^n]$ be such that
the  codimension of the closed subvariety $\BA^n\setminus (U_f\cup U_g)$ in
$\BA^n$ is at least $2$, where $U_f=\{ u\in \BA^n\mid f(u)\neq0\}$.
Assume moreover that each of the morphisms
$\pi^{-1}(U_f)\to U_f$ and $\pi^{-1}(U_g)\to U_g$, induced by $\pi$,
is an isomorphism. Then $\pi$ is an isomorphism.
\end{lem}

\begin{proof}
Let
$X_f=\pi^{-1}(U_f)$, resp. $X_g=\pi^{-1}(U_g)$, and write
$j\colon U_f\cup U_g\hookrightarrow \BA^n$,
resp. $j_X\colon X_f\cup X_g\hookrightarrow X$,  for the
open imbedding.
We have the following commutative diagram

\[
\xymatrix{
\BC[\BA^n]\ \ar@{^{(}->}[r]^<>(0.5){j^*}\ar[d]^<>(0.5){\pi^*}
&\ \Gamma(U_f\cup U_g, {\mathcal O}_{U_f\cup U_g})\ar@{=}[r]&
\{u\oplus v\in \BC[U_f]\oplus \BC[U_g]\mid u|_{U_f\cap U_g}=v|_{U_f\cap U_g}\}
\ar[d]^<>(0.5){(\pi|_{X_f})^*\oplus (\pi|_{X_g})^*}\\
\BC[X]\ar@{^{(}->}[r]^<>(0.5){j_X^*}&
\Gamma(X_f\cup X_g, {\mathcal O}_{X_f\cup X_g})\ar@{=}[r]&
\{\tilde u\oplus \tilde v\in \BC[X_f]\oplus \BC[X_g]\mid
\tilde u|_{X_f\cap X_g}=\tilde v|_{X_f\cap X_g}\}
}
\]

The map $j^*$ in this diagram is an isomorphism by the codimension $\geq 2$ assumption.
The map $j_X^*$  is injective since $X$ is irreducible.
The assumptions imply also that the vertical map $(\pi|_{X_f})^*\oplus (\pi|_{X_g})^*$ on the right is  an isomorphism.
It follows that  the vertical map $\pi^*$ on the left must be an isomorphism, as required.
\end{proof}

\subsection{Restriction to $\Gr\times\bV_0$}
\label{Restriction}
The existence of the desired equivalence $\Phi^{1,1}$ of~Theorem~\ref{main spherical mirabolic}
follows from~Lemma~\ref{purity} and~Lemma~\ref{phi isom}. The equivalences $\Phi^{0,2}$ and
$\Phi^{2,0}$ are obtained by conjugating with $\varrho_\righ^{\pm1}$. It remains to check their
compatibility with monoidal structures.

We denote by $\bar\jmath_0$ (resp.\ $\bar\jmath_{-1}$) the closed embedding
$\Gr\times\bV_0\hookrightarrow\Gr\times\bV$ (resp.\
$\Gr\times t\bV_0\hookrightarrow\Gr\times\bV$). We also denote by $\jmath_0$ the locally closed
embedding $\Gr\times(\bV_0\setminus t\bV_0)\hookrightarrow\Gr\times\bV$.
Our goal in this section is a description in terms of $\Phi^{1,1}$ of the endofunctors
$\bar\jmath_{0*}\bar\jmath{}_0^!,\ \jmath_{0*}\jmath_0^!\colon D_{!*\bG_\bO}(\Gr\times\bV)\to
D_{!*\bG_\bO}(\Gr\times\bV)$.

Recall the setup and notation of~Section~\ref{conv coh}. We have the natural morphisms
\[p\colon\CQ^A\to\Hom(V_1,V_2)\times\Hom(V_2,V_1)=\Pi\fg^*_{\bar1},\]
\[q\colon\CQ^A\to\Hom(V'_1,V_2)\times\Hom(V_2,V'_1)=\Pi\fg^*_{\bar1};\]
\[p(A,A',A'',B,B',B'')=(A,B),\ q(A,A',A'',B,B',B'')=(A',B').\] We also have the
natural morphisms \[p,q\colon G_\CQ\to G_{\bar0},\ p(g_1,g'_1,g_2)=(g_1,g_2),\
q(g_1,g'_1,g_2)=(g'_1,g_2).\] Clearly, $p,q\colon\CQ^A\to\Pi\fg^*_{\bar1}$ are equivariant with
respect to $p,q\colon G_\CQ\to G_{\bar0}$. Hence we have the convolution functor
\[p_*q^*\colon\Coh^{G_{\bar0}}(\Pi\fg^*_{\bar1})=\Coh(G_{\bar0}\backslash\Pi\fg^*_{\bar1})\xrightarrow{q^*}
\Coh(G_\CQ\backslash\CQ^A)\xrightarrow{p_*}\Coh(G_{\bar0}\backslash\Pi\fg^*_{\bar1})=
\Coh^{G_{\bar0}}(\Pi\fg^*_{\bar1})\] (in particular, $p_*$ involves taking $\GL(V'_1)$-invariants).
We will actually need the same named functor $p_*q^*\colon
D_\perf^{G_{\bar0}}(\fG_{1,1}^\bullet)\to D_\perf^{G_{\bar0}}(\fG_{1,1}^\bullet)$
defined similarly using
the dg-algebra with trivial differential $\fG_{1,1}^\bullet\otimes\BC[\Hom(V_1,V'_1)]$
(the grading on $\BC[\Hom(V_1,V'_1)]$
is trivial, and if we disregard the grading, then
$\fG_{1,1}^\bullet\otimes\BC[\Hom(V_1,V'_1)]\simeq\BC[\CQ^A]$).

We also have a $G_\CQ$-invariant subvariety $\CQ^A_0\subset\CQ^A$ given by the equation that
$A''$ is noninvertible.
The restriction of $p,q$ to $\CQ^A_0$ will be denoted by $p_0,q_0$. As above, we obtain the
functor $p_{0*}q_0^*\colon D_\perf^{G_{\bar0}}(\fG_{1,1}^\bullet)\to D_\perf^{G_{\bar0}}(\fG_{1,1}^\bullet)$.

\begin{prop}
  \label{corestriction}
  \textup{(a)}  There is an isomorphism of functors
  \[\bar\jmath_{0*}\bar\jmath{}_0^!\circ\Phi^{1,1}\simeq
  \Phi^{1,1}\circ p_*q^*\colon D^{G_{\bar0}}_\perf(\fG_{1,1}^\bullet)\to D_{!*\bG_\bO}(\Gr\times\oV).\]

  \textup{(b)} The isomorphism in \textup{(a)} can be extended to the following commutative diagram
  of morphisms:
  \[
  \xymatrix{
    \bar\jmath_{-1*}\bar\jmath{}_{-1}^!\circ\Phi^{1,1} \ar[r]^-\sim \ar[d] \ar@/_2pc/[dd]   &
    \Phi^{1,1}\circ\big(\det V_1\otimes(p_*q^*)\circ(\det\!{}^{-1}V'_1\otimes{-})\big)
    \ar[d] \ar@/^2pc/[dd]^{\cdot\det A''} \\
    \Phi^{1,1} \ar[r]^\Id & \Phi^{1,1}   \\
    \bar\jmath_{0*}\bar\jmath{}_0^!\circ\Phi^{1,1} \ar[r]^\sim \ar[u] & \Phi^{1,1}\circ p_*q^*. \ar[u]
  }
  \]

  \textup{(c)} There is an isomorphism of functors \[\jmath_{0*}\jmath_0^!\circ\Phi^{1,1}\simeq
  \Phi^{1,1}\circ p_{0*}q_0^*\colon D^{G_{\bar0}}_\perf(\fG_{1,1}^\bullet)\to D_{!*\bG_\bO}(\Gr\times\oV).\]
\end{prop}

\begin{proof}
  (a) The support of $\IC_{(\bnu,\bmu)}$ lies in $\Gr\times\bV_0$ iff $\bnu$ is a negative
  partition: $(0\geq\nu_1\geq\ldots\geq\nu_N)$ (see~\cite[Proof of Proposition~8]{fgt}).
  So $D_{!*\bG_\bO}(\Gr\times\bV_0)$ is generated by the collection of objects
  $\IC_{(\bnu,\bmu)}$ where $\bnu$ is a negative partition.
  Thus we see that $\bar\jmath_{0*}$ is a fully faithful functor whose image is generated by
  $\{\IC_{(\bnu,\bmu)}\mid\bnu\le0\}$ and
  $\bar\jmath_0^!$ is the right adjoint of $\bar\jmath_{0*}$.
 Recall that $V^*_\bnu$ denotes
  an irreducible representation of $\GL(V)$ obtained by applying the corresponding Schur
  functor to $V^*$. The result of application of the same Schur functor to $V^*_1$
  (resp.\ $V^{\prime*}_1,V^*_2$) will be denoted $V^*_{1,\bnu}$
  (resp.\ $V^{\prime*}_{1,\bnu},V^*_{2,\bnu}$). Since
  \[\IC_{(\bnu,\bmu)}=\Phi^{1,1}(V_{1,\bnu}^*\otimes\fG_{1,1}^\bullet\otimes V_{2,\bmu}^*)=
  \Phi^{1,1}(V_{1,\bnu^*}\otimes\fG_{1,1}^\bullet\otimes V_{2,\bmu^*})\] (where
  $\bnu^*=(-\nu_N,-\nu_{N-1},\ldots,-\nu_2,-\nu_1)$ for $\bnu=(\nu_1,\ldots,\nu_N)$),
  we have to show that $p_*q^*$ lands in the subcategory
  $D_\perf^{G_{\bar0},\geq}(\fG_{1,1}^\bullet)\subset D_\perf^{G_{\bar0}}(\fG_{1,1}^\bullet)$ generated by objects
  $\{V_{1,\blambda}\otimes\fG_{1,1}^\bullet\otimes V_{2,\bmu}\mid\bla\geq0\}$ and to construct the
  adjunction  isomorphism
  \begin{multline*}
    \Hom_{D_\perf^{G_{\bar0}}(\fG_{1,1}^\bullet)}(V_{1,\blambda}\otimes\fG_{1,1}^\bullet\otimes V_{2,\bmu},
  V_{1,\blambda'}\otimes\fG_{1,1}^\bullet\otimes V_{2,\bmu'})\\ \iso
  \Hom_{D_\perf^{G_{\bar0}}(\fG_{1,1}^\bullet)}(V_{1,\blambda}\otimes\fG_{1,1}^\bullet\otimes V_{2,\bmu},
  p_*q^*(V'_{1,\blambda'}\otimes\fG_{1,1}^\bullet\otimes V_{2,\bmu'}))
  \end{multline*}
  for a partition $\blambda$. Equivalently, we have to construct an isomorphism
  \begin{multline*}\big((V_{1,\blambda}^{\prime*}\otimes
    V'_{1,\blambda'})\otimes\BC[\Hom(V'_1,V_2)\times\Hom(V_2,V'_1)]
  \otimes(V_{2,\bmu}^*\otimes V_{2,\bmu'})\big)^{\GL(V'_1)\times\GL(V_2)}\\ \iso
  \big(V_{1,\blambda}^*\otimes
  p_*q^*(V'_{1,\blambda'}\otimes\BC[\Hom(V'_1,V_2)\times\Hom(V_2,V'_1)]\otimes V_{2,\bmu'})
  \otimes V_{2,\bmu}^*\big)^{\GL(V_1)\times\GL(V_2)}\\
  :=\big(V^*_{1,\blambda}\otimes
  V'_{1,\blambda'}\otimes\BC[\CQ^A]\otimes(V^*_{2,\bmu}\otimes V_{2,\bmu'})\big)^{G_\CQ}.
  \end{multline*}
  Recall that $\CQ^A=\Hom(V_1,V'_1)\times\Hom(V'_1,V_2)\times\Hom(V_2,V_1)$. We
  apply~Lemma~\ref{IT}(c) below to $U_1=V_2,\ U_2=V_1,\ U_3=V'_1,\ \bnu=\blambda$
  (notation of~\ref{inth}) to obtain an isomorphism
  \[V^{\prime*}_{1,\blambda}\otimes\BC[\Hom(V_2,V'_1)]\iso\big(\BC[\Hom(V_2,V_1)\otimes
    V^*_{1,\blambda}\otimes\BC[\Hom(V_1,V'_1)]\big)^{\GL(V_1)}\]
    whose inverse induces the desired adjunction isomorphism.

    We still have to check that the essential image of $p_*q^*$ lies in the subcategory
    $D_\perf^{G_{\bar0},\geq}(\fG_{1,1}^\bullet)\subset D_\perf^{G_{\bar0}}(\fG_{1,1}^\bullet)$ generated by
    $\{V_{1,\blambda}\otimes\fG_{1,1}^\bullet\otimes V_{2,\bmu}\mid\bla\geq0\}$. We consider the
    homomorphism $\fG_{1,1}^\bullet\to\BC$ killing all the generators, and for
    $\CM\in D^{G_{\bar0}}_\perf(\fG_{1,1}^\bullet)$ we set
    $z_0^*\CM:=\BC\otimes_{\fG_{1,1}^\bullet}^L\!\!\CM\in D^{G_{\bar0}}(\BC)$
    (``the fiber at $0\in\Hom(V_1,V_2)\times\Hom(V_2,V_1)$''). Note that $z_0^*p_*q^*$ lands
    in the category generated by $\{V_{1,\blambda}\otimes V_{2,\bmu}\mid\bla\geq0\}$, i.e.\ the
    category of modules with polynomial action of $\GL(V_1)$. Indeed,
    recall that $G_\CQ$ acts on $\CQ^A$ via
    \[(g_1,g_1',g_2)(A',A'',B)=(g_2A'g_1^{\prime-1},g'_1A''g_1^{-1},g_1Bg_2^{-1}).\]
    Since we impose the conditions $B=0=A:=A'A''$, the action of $\GL(V_1)$ on
    $z_0^*p_*q^*\CN$ (for a free dg-$\fG_{1,1}^\bullet$-module $\CN$) comes from its action on
    functions of $A''$, and the latter action is polynomial.

    Finally, we claim that if the action of $\GL(V_1)$ on $z_0^*\CM$ is polynomial, then
    $\CM$ lies in the subcategory generated by
    $\{V_{1,\blambda}\otimes\fG_{1,1}^\bullet\otimes V_{2,\bmu}\mid\bla\geq0\}$. To this end we
    apply the Koszul equivalence
    $\varkappa\colon D_\fid^{G_{\bar0}}(\Lambda)\iso D^{G_{\bar0}}_\perf(\fG_{1,1}^\bullet)$
    of~Section~\ref{koszul}. It is easy to see that if the action of $\GL(V_1)$ on the
    total cohomology of $\CK\in D_\fid^{G_{\bar0}}(\Lambda)$ is polynomial, then
    $\CK$ lies in the subcategory $\on{Pol}\subset D_\fid^{G_{\bar0}}(\Lambda)$ generated
    by the $G_{\bar0}\ltimes\Lambda$-modules with polynomial action of $\GL(V_1)$ and trivial
    action of $\Lambda$. Now $\varkappa(\on{Pol})$ is the subcategory
    $D_\perf^{G_{\bar0},\geq}(\fG_{1,1}^\bullet)\subset D_\perf^{G_{\bar0}}(\fG_{1,1}^\bullet)$. And if
    the action of $\GL(V_1)$ on $z_0^*\CM$ is polynomial, then $\CM\simeq\varkappa(\CK)$, where
    the action of $\GL(V_1)$ on the total cohomology of $\CK$ is polynomial.
    This completes the proof of~(a).

    \medskip

    (b)   
    As $\alpha^*\bar\jmath_{-1*}\bar\jmath{}_{-1}^!\simeq\bar\jmath_{0*}\bar\jmath{}_0^!\alpha^*$
  (notation of~Section~\ref{localization}), we deduce an isomorphism of functors
  \[\bar\jmath_{-1*}\bar\jmath{}_{-1}^!\circ\Phi^{1,1}\simeq\Phi^{1,1}\circ\big(\det V_1\otimes
  (p_*q^*)\circ(\det\!{}^{-1}V'_1\otimes{-})\big)\colon D^{G_{\bar0}}_\perf(\fG_{1,1}^\bullet)\to
  D_{!*\bG_\bO}(\Gr\times\oV).\] Thus the upper and lower rectangles of the diagram in (b)
  are commutative by construction. We have to prove
  the commutativity of the big curved quadrangle.
    The endofunctors $\bar\jmath_{0*}\bar\jmath{}_0^!,\ \bar\jmath_{-1*}\bar\jmath{}_{-1}^!$
    of $D_{!*\bG_\bO}(\Gr\times\bV)$ are equipped with the structure of idempotent comonads.
    The desired commutativity follows from the fact that given two idempotent comonads
    $T_0,T_1\colon{\mathcal C}\to{\mathcal C}$, there is at most one morphism of functors
    $\chi\colon T_1\to T_0$ such that $\varepsilon_1=\varepsilon_0\circ\chi$ for the
    counits $\varepsilon_i\colon T_i\to\Id_{\mathcal C}$. Indeed,
    $\chi=T_0(\varepsilon_1)\circ\chi_1$, where $\chi_1$ is defined as the composition
    \[T_1\cong T_1\circ T_1\xrightarrow{\chi\circ T_1}T_0\circ T_1.\]
    We claim that $\chi_1$ is an isomorphism uniquely defined as the inverse to the morphism
    $\varepsilon_0\circ T_1\colon T_0\circ T_1\iso T_1$. In effect, the composition
    \[T_1\cong T_1\circ T_1\xrightarrow{\chi\circ T_1}T_0\circ T_1
    \xrightarrow{\varepsilon_0\circ T_1} T_1\] equals the composition
    \[T_1\cong T_1\circ T_1\xrightarrow{\varepsilon_1\circ T_1}\Id_{\mathcal C}\circ T_1=T_1\]
    that in turn equals $\Id_{T_1}$. Conversely, the composition
    \[T_0\circ T_1\xrightarrow{\varepsilon_0\circ T_1} T_1\cong T_1\circ T_1
    \xrightarrow{\chi\circ T_1}T_0\circ T_1\] equals the composition
    \[T_0\circ T_1\cong T_0\circ T_1\circ T_1\xrightarrow{T_0\circ\chi\circ T_1}
    T_0\circ T_0\circ T_1\cong T_0\circ T_1\] that in turn equals $\Id_{T_0\circ T_1}$.

    This completes the proof of (b), but we would like to give one more independent argument
    that will prove useful later on.

\medskip

  Recall that we have to prove
  the commutativity of the big curved quadrangle. To this end we change the setting to the
  base field $\BF_q$ as in the proof of Lemma~\ref{formality}. That is we replace
  $D_{!*\bG_\bO}(\Gr\times\oV)$ by the equivalent equivariant derived category of sheaves on
  $(\Gr\times\oV)_{\ol\BF_q}$ as in~\cite[Proposition~5]{bf} (in particular, choosing an
  isomorphism $\BC\simeq\ol\BQ_\ell$). However, we preserve the notation $D_{!*\bG_\bO}(\Gr\times\oV)$
  for this category in order not to overload our notation (anyway, it will only be used during
  the current proof). All the irreducible perverse
  sheaves $\IC_{(\blambda,\bmu)}$ carry a natural Tate Weil structure by~\cite[Proposition~11]{fgt}.
  They (along with their Tate twists) generate a subcategory\footnote{closed with respect to
  taking cones and direct summands} $\widehat{D}_{!*\bG_\bO}(\Gr\times\oV)$
  of the mixed version $D^{\on{mix}}_{!*\bG_\bO}(\Gr\times\oV)$ of $D_{!*\bG_\bO}(\Gr\times\oV)$.
  We will use a particular dg-model of $\widehat{D}_{!*\bG_\bO}(\Gr\times\oV)$ viewed as a
  category enriched over complexes equipped with an action of the Frobenius automorphism $\on{Fr}$.
  Note that the absolute values of the eigenvalues of $\on{Fr}$ lie in $\sqrt{q}^\BZ$, and
  hence our complexes carry an additional grading according to the absolute values of the
  eigenvalues of $\on{Fr}$. If we forget the mixed structure and remember only this additional
  grading, we obtain a category $\widetilde{\ul D}{}_{!*\bG_\bO}(\Gr\times\oV)$ enriched over complexes
  equipped with an additional grading. Its localization with respect to quasiisomorphisms
  will be denoted $\widetilde{D}_{!*\bG_\bO}(\Gr\times\oV)$.

  On the other hand, we consider the category $\ul{D}{}_\perf^{G_{\bar0}\times\BC^\times}(\fG^\bullet_{1,1})$
  of perfect $G_{\bar0}\times\BC^\times$-equivariant dg-$\fG^\bullet_{1,1}$-modules and its
  localization (with respect to quasi-isomorphisms) $D_\perf^{G_{\bar0}\times\BC^\times}(\fG^\bullet_{1,1})$.
  Here all the generators of $\fG^\bullet_{1,1}$ have weight 1 with respect to the action of
  $\BC^\times$. Then the standard modification of our construction of the equivalence
  $\Phi^{1,1}\colon D^{G_{\bar0}}_\perf(\fG_{1,1}^\bullet)\iso D_{!*\bG_\bO}(\Gr\times\oV)$ produces a functor
  $\widetilde{\ul\Phi}{}^{1,1}\colon \ul{D}{}^{G_{\bar0}\times\BC^\times}_\perf(\fG_{1,1}^\bullet)\to
  \widetilde{\ul D}{}_{!*\bG_\bO}(\Gr\times\oV)$ and its localization
  $\widetilde{\Phi}{}^{1,1}\colon D^{G_{\bar0}\times\BC^\times}_\perf(\fG_{1,1}^\bullet)\iso
  \widetilde{D}{}_{!*\bG_\bO}(\Gr\times\oV)$.

  We have the similar diagram of morphisms of functors
  $D^{G_{\bar0}\times\BC^\times}_\perf(\fG_{1,1}^\bullet)\to\widetilde{D}{}_{!*\bG_\bO}(\Gr\times\oV)$
  with commutative upper and lower rectangles
  \[
  \xymatrix{
  \bar\jmath_{-1*}\bar\jmath{}_{-1}^!\circ\widetilde{\Phi}{}^{1,1} \ar[r]^-\sim \ar[d] \ar@/_2pc/[dd]
  &  \widetilde{\Phi}{}^{1,1}\circ\big(\det V_1\otimes(p_*q^*)\circ(\det\!{}^{-1}V'_1\otimes{-})\big)
    \ar[d] \ar@/^2pc/[dd]^{\cdot\det A''} \\
    \widetilde{\Phi}{}^{1,1} \ar[r]^\Id & \widetilde{\Phi}{}^{1,1}   \\
    \bar\jmath_{0*}\bar\jmath{}_0^!\circ\widetilde{\Phi}{}^{1,1} \ar[r]^\sim \ar[u] &
    \widetilde{\Phi}{}^{1,1}\circ p_*q^*, \ar[u]
  }
  \]
  and we have to prove the commutativity of the big curved quadrangle.
  The endofunctors $\det V_1\otimes(p_*q^*)\circ(\det\!{}^{-1}V'_1\otimes{-}),\ \Id,\ p_*q^*$ of
  $D^{G_{\bar0}\times\BC^\times}_\perf(\fG_{1,1}^\bullet)$ are given by their respective kernels
  $K_1,K_2,K_3$ in
  $\ul{D}{}^{G_{\bar0}^2\times\BC^\times}_\perf(\fG_{1,1}^\bullet\otimes\fG_{1,1}^\bullet)$.
  Note that the equivariance with respect to $\BC^\times$ (as opposed to $(\BC^\times)^2$)
  suffices since all the three functors under consideration commute with the shifts of
  the additional grading. Similarly, all the three functors on the constructible side
  commute with the Tate twists. All the three
  kernels are {\em pure} of weight 0, that is, their additional gradings coincide with their
  cohomological gradings. The category of pure weight 0 objects in
  $D^{G_{\bar0}^2\times\BC^\times}_\perf(\fG_{1,1}^\bullet\otimes\fG_{1,1}^\bullet)$ is equivalent to the
  abelian category
  of $G_{\bar0}^2\times\BC^\times$-equivariant $\fG_{1,1}\otimes\fG_{1,1}$-modules
  (the equivalence being obtained by taking cohomology).
  Therefore, the morphisms of functors
  $(\det V_1\otimes{-})\circ(p_*q^*)\circ(\det\!{}^{-1}V'_1\otimes{-})\to\Id$ and $p_*q^*\to\Id$ arise
  from the morphisms between the respective kernels that are injective as morphisms of
  $\fG_{1,1}\otimes\fG_{1,1}$-modules. We have to compare certain morphisms to
  $\widetilde{\ul\Phi}{}^{1,1}K_3$, and we know that their compositions with the
  monomorphism $\widetilde{\ul\Phi}{}^{1,1}K_3\to\widetilde{\ul\Phi}{}^{1,1}K_2$ coincide,
  hence the desired equality of morphisms. This completes our second proof of (b).

  \medskip

(c) follows from the comparison of the distinguished triangles
  \[(\bar\jmath_{-1*}\bar\jmath{}_{-1}^!\to\bar\jmath_{0*}\bar\jmath{}_0^!\to\jmath_{0*}\jmath_0^!)
  \circ\Phi^{1,1}\] and
  \[\Phi^{1,1}\circ\big(\det V_1\otimes (p_*q^*)\circ(\det\!{}^{-1}V'_1\otimes{-})
  \to p_*q^*\to p_{0*}q_0^*\big).\]

  The proposition is proved.
\end{proof}

\subsection{Some invariant theory}
\label{inth}
Let $U_1,U_2,U_3$ be vector spaces of dimensions $n_1,n_2,n_3$. The irreducible polynomial
(resp.\ antipolynomial) representations of $\GL(U_i)$ are realized in the Schur spaces
$\BS_\blambda U_i$ (resp.\ $\BS_\blambda U_i^*$), where
$\blambda$ is a partition with $\ell(\blambda)\leq n_i$. We will also write
$\BS_{\blambda^*}U_i$ for $\BS_\blambda U_i^*$, where
\[\blambda^*=-w_0\blambda=(-\lambda_{n_i},-\lambda_{n_i-1},\ldots,-\lambda_2,-\lambda_1)\] for
$\blambda=(\lambda_1,\ldots,\lambda_{n_i})$. We also set $\BS_\blambda U_i=0=\BS_{\blambda^*}U_i$
for $\ell(\blambda)>n_i$. We denote by $\Hom_{\leq n_2}(U_1,U_3)\subset\Hom(U_1,U_3)$ the
subvariety formed by all the homomorphisms of rank $\leq n_2$.

\begin{lem}
  \label{IT}
  The composition of homomorphisms induces the following isomorphisms of
  $\GL(U_1)\times\GL(U_3)$-modules:

  \textup{(a)} $\BC[\Hom_{\leq n_2}(U_1,U_3)]\iso
  \big(\BC[\Hom(U_1,U_2)]\otimes\BC[\Hom(U_2,U_3)]\big)^{\GL(U_2)}$.

  \textup{(b)} In case $n_1\leq n_2$, for a partition $\bnu$,
  \[\BS_\bnu U_1\otimes\BC[\Hom(U_1,U_3)]\iso
  \big(\BC[\Hom(U_1,U_2)]\otimes\BS_\bnu U_2\otimes\BC[\Hom(U_2,U_3)]\big)^{\GL(U_2)}.\]

  \textup{(c)} In case $n_3\leq n_2$, for a partition $\bnu$,
  \[\BC[\Hom(U_1,U_3)]\otimes\BS_\bnu U^*_3\iso
  \big(\BC[\Hom(U_1,U_2)]\otimes\BS_\bnu U^*_2\otimes\BC[\Hom(U_2,U_3)]\big)^{\GL(U_2)}.\]
\end{lem}

\begin{proof}
  (a) We have $\BC[\Hom(U_i,U_j)]=\bigoplus_{m\geq0}\Sym^m(U_j^*\otimes U_i)=
  \bigoplus_\blambda\BS_\blambda U_i\otimes\BS_{\blambda^*} U_j$ as a $\GL(U_i)\times\GL(U_j)$-module.
  Also, $\BC[\Hom_{\leq r}(U_i,U_j)]=
  \bigoplus_{\ell(\blambda)\leq r}\BS_\blambda U_i\otimes\BS_{\blambda^*} U_j$ as a
  $\GL(U_i)\times\GL(U_j)$-module. Clearly,
  $(\BS_{\blambda^*}U_2\otimes\BS_\bmu U_2)^{\GL(U_2)}=\BC^{\delta_{\blambda\bmu}}$. So the two sides
  of~(a) are isomorphic as $\GL(U_1)\times\GL(U_3)$-modules. On the other hand, the morphism in question
  is injective since the composition morphism $\Hom(U_1,U_2)\times\Hom(U_2,U_3)\to\Hom_{\leq n_2}(U_1,U_3)$
  is dominant. Hence the morphism in question is an isomorphism.

  (b) We consider a copy $U'_2$ of $U_2$, we tensor both sides of (b) with $\BS_{\bnu^*}U'_2$,
  and take direct sum over all partitions $\bnu$ with $\ell(\bnu)\leq n_2$. Then we have
  to prove that the morphism (induced by the composition of arrows of the $D_4$-quiver
  in~(\ref{D4}))
  \begin{multline*}
    \label{red lemma}
  \gamma\colon \BC[\Hom(U_1,U_2')]\otimes\BC[\Hom(U_1,U_3)]\to\\
  \big(\BC[\Hom(U_1,U_2)]\otimes \BC[\Hom(U_2,U_2')]\otimes\BC[\Hom(U_2,U_3)]\big)^{\GL(U_2)}
  \end{multline*}
  is an isomorphism. 
  
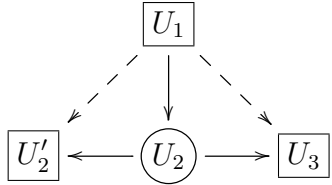
\begin{equation}
  \label{D4}
  \xymatrix{
    & \boxed{U_1} \ar@{-->}[dl] \ar[d] \ar@{-->}[dr]\\
    \boxed{U'_2} &*+=+<8pt>[o][F]+ {U_2} \ar[l] \ar[r] & \boxed{U_3}
}
  \end{equation}
  Now the statement can be reduced to~(a) using the substitution $U_3 \rightsquigarrow U_3\oplus U'_2$.
  Alternatively,
  the condition $n_1\leq n_2$ guarantees that the morphism from the representation space of the
  $D_4$-quiver to the representation space of the dashed $A_3$-quiver is dominant. Hence $\gamma$
  is injective. The surjectivity of $\gamma$ follows e.g.\ from~\cite[Theorem~1]{lp}.

  (c) is dual to (b).
\end{proof}

\subsection{The monoidal property of $\Phi^{2,0}$}
\label{20 monoidal}
Recall the notation of~Section~\ref{conv coh}.
The monoidal structure $\srel{A}*$ on $D^{G_{\bar0}}_\perf(\fG_{2,0}^\bullet)$ is defined via the
kernel $\BC[\CQ^A]_{2,0}^\bullet$: a $G_\CQ$-equivariant dg-$\BC[\CH]_{2,0}^\bullet$-module.
The monoidal structure $\srel!\oast$ on
$D_{!\bG_\bO}(\Gr\times\bV)$ transfered to $D^{G_{\bar0}}_\perf(\fG_{2,0}^\bullet)$ via the equivalence
$\Phi^{2,0}$ is also defined via a kernel $\CK^\bullet$ (a $G_\CQ$-equivariant
dg-$\BC[\CH]_{2,0}^\bullet$-module). We have to construct an isomorphism of
$G_\CQ$-equivariant dg-$\BC[\CH]_{2,0}^\bullet$-modules $\BC[\CQ^A]_{2,0}^\bullet\iso\CK^\bullet$.

We denote by $\BC[\CH_\loc]_{2,0}^\bullet$ the localization of $\BC[\CH]_{2,0}^\bullet$ defined as
$\BC[\CH]_{2,0}^\bullet[\det^{-1}\!\!A,\det^{-1}\!\!A',\det^{-1}\!\!A'']$. We define
$\CK_\loc^\bullet:=\BC[\CH_\loc]_{2,0}^\bullet\otimes_{\BC[\CH]_{2,0}^\bullet}\CK^\bullet$ and
$\BC[\CQ^A_\loc]_{2,0}^\bullet:=
\BC[\CH_\loc]_{2,0}^\bullet\otimes_{\BC[\CH]_{2,0}^\bullet}\BC[\CQ^A]_{2,0}^\bullet$.

We have $\Phi^{2,0}\fG_{2,0}^\bullet\simeq\varrho_\righ^{-1}E_0$.
Also, for $\CF\in D_{!\bG_\bO}(\Gr\times\bV)$ we have
$\varrho_\righ^{-1}E_0\srel!\oast\CF=\bar\jmath_{0*}\bar\jmath_0^!\CF$.
Thus~Proposition~\ref{corestriction}(a) yields an isomorphism of functors
$\Phi^{2,0}(\fG_{2,0}^\bullet\srel{A}*-)\iso(\Phi^{2,0}\fG_{2,0}^\bullet)\srel!\oast\Phi^{2,0}-$.
This isomorphism yields in turn an isomorphism of kernels
$\BC[\CQ_\forg^A]_{2,0}^\bullet\iso\CK_\forg^\bullet$, where the subscript $_\forg$ denotes
the restriction of the dg-module structure from $\BC[\CH]_{2,0}^\bullet$ to
$\BC[\CH_\forg]_{2,0}^\bullet:=\Sym(\Hom(V_1,V_2)[-2])\otimes\Sym(\Hom(V_2,V_1))
\otimes\Sym(\Hom(V'_1,V_2)[-2])\otimes\Sym(\Hom(V_2,V'_1))$.

According to~Proposition~\ref{corestriction}(b),
the following diagram of functors $D^{G_{\bar0}}_\perf(\fG_{2,0}^\bullet)\to D_{!\bG_\bO}(\Gr\times\bV)$
commutes:
\[\xymatrix{
  \bar\jmath_{-1*}\bar\jmath{}_{-1}^!\circ\Phi^{2,0} \ar[r]^-\sim \ar[d] &
  \Phi^{2,0}\circ\big(\det V_1\otimes
(p_*q^*)\circ(\det\!{}^{-1}V'_1\otimes{-})\big) \ar[d]^{\cdot\det A''} \\
\bar\jmath_{0*}\bar\jmath{}_0^!\circ\Phi^{2,0} \ar[r]^\sim & \Phi^{2,0}\circ p_*q^*.
}\]
Hence the diagram
\[\begin{CD}
\BC[\CQ_\forg^A]_{2,0}^\bullet @>{\sim}>> \CK_\forg^\bullet\\
@VV{\cdot\det A''}V @VV{\cdot\det A''}V\\
\BC[\CQ_\forg^A]_{2,0}^\bullet @>{\sim}>> \CK_\forg^\bullet
\end{CD}\]
commutes as well, and in particular the multiplication by $\det A''$ is injective on
$\CK^\bullet$, and hence $\CK^\bullet\hookrightarrow\CK_\loc^\bullet$.

Now since $\Phi^{2,0}_\loc\colon D^{G_{\bar0}}_\perf(\fA^\bullet)\iso
D_{\bG_\bO}(\Gr)\cong D_{!\bG_\bO}^\loc(\Gr\times\bV)$ (see~Lemma~\ref{restriction to zero}
and~Corollary~\ref{localized ext}) coincides with the equivariant Satake equivalence,
and the latter one is monoidal, we obtain
an isomorphism of localized kernels $\BC[\CQ_\loc^A]_{2,0}^\bullet\iso\CK_\loc^\bullet$
as $G_\CQ$-equivariant $\BC[\CH_\loc]_{2,0}^\bullet$-modules.
By the argument in the second proof of~Proposition~\ref{corestriction}(b) (using the additional
grading and purity of $\CK^\bullet$), it remains to verify that
this isomorphism restricts to the desired isomorphism from
$\BC[\CQ^A]_{2,0}^\bullet\subset\BC[\CQ_\loc^A]_{2,0}^\bullet$ to $\CK^\bullet\subset\CK_\loc^\bullet$.
For this verification it suffices to restrict the scalars to $\BC[\CH_\forg]_{2,0}^\bullet$.
But we have already seen that over $\BC[\CH_\forg]_{2,0}^\bullet$ we obtain an isomorphism
$\BC[\CQ_\forg^A]_{2,0}^\bullet\iso\CK_\forg^\bullet$.

This completes the proof of the monoidal property of $\Phi^{2,0}$.

\subsection{Fourier Transform}
\label{fourier}
We have the Fourier transform functors (along $\bV$) \[\FT\colon D_{!\bG_\bO}(\Gr\times\bV)\to
D_{*\bG_\bO}(\Gr\times\bV),\ \FT\colon D_{*\bG_\bO}(\Gr\times\bV)\to D_{!\bG_\bO}(\Gr\times\bV),\]
\[\FT\colon D_{!*\bG_\bO}(\Gr\times\bV)\to D_{!*\bG_\bO}(\Gr\times\bV).\]
Strictly speaking, the Fourier transform goes not to $D_{?\bG_\bO}(\Gr\times\bV)$, but to
$D_{?\bG_\bO}(\Gr\times(V^*\otimes\bF))$. However, we identify $V^*$ with $V$ using our choice
of (selfdual) basis $e_1,\ldots,e_N$, and accordingly change the action of $\bG_\bO=\GL(N,\bO)$
by composing it with the automorphism $g\mapsto{}^t\!g^{-1}$. Note that the resulting
Fourier transform to $D_{?\bG_\bO}(\Gr\times\bV)$ is independent of the choice of basis in $V$.

To describe the effect of $\FT$ on the coherent side, we identify $V_1\cong V_1^*$ and
$V_2\cong V_2^*$ using our bases. Furthermore, we identify
\[\Hom(V_1,V_2)\cong\Hom(V_2^*,V_1^*)\cong\Hom(V_2,V_1),\ A\mapsto B:=-{}^t\!\!A,\]
\[\Hom(V_2,V_1)\cong\Hom(V_1^*,V_2^*)\cong\Hom(V_1,V_2),\ B\mapsto A:={}^t\!B.\]
Thus we obtain an $\iota$-equivariant transposition isomorphism
\begin{multline*}\tau\colon\fG_{2,0}^\bullet=\Sym(\Hom(V_1,V_2))\otimes
  \Sym(\Hom(V_2,V_1)[-2])\\ \to\Sym(\Hom(V_1,V_2)[-2])\otimes\Sym(\Hom(V_2,V_1))=\fG_{0,2}^\bullet,
\end{multline*}
where $\iota\colon G_{\bar0}\to G_{\bar0}$ is an automorphism
$(g_1,g_2)\mapsto({}^t\!g_1^{-1},{}^t\!g_2^{-1})$.
We denote the extension of scalars via $\tau$ by
$\phitau\colon D^{G_{\bar0}}_\perf(\fG_{2,0}^\bullet)\to D^{G_{\bar0}}_\perf(\fG_{0,2}^\bullet).$
Clearly, the functor $\phitau\colon\big(D^{G_{\bar0}}_\perf(\fG_{2,0}^\bullet),\ \srel{A}*\big)\to
\big(D^{G_{\bar0}}_\perf(\fG_{0,2}^\bullet),\ \srel{B}*\big)$ is monoidal. Also, by the standard
properties of the Fourier transform, the functor
\[\FT\colon\big(D_{!\bG_\bO}(\Gr\times\bV),\ \srel!\oast\big)\to
\big(D_{*\bG_\bO}(\Gr\times\bV),\ \srel**\big)\] is monoidal.
Thus the monoidal property of $\Phi^{0,2}$ is a corollary of the following

\begin{prop}
  \label{FT}
  The functors 
  $\FT\circ\Phi^{2,0}\colon D^{G_{\bar0}}_\perf(\fG_{2,0}^\bullet)\to D_{*\bG_\bO}(\Gr\times\bV)$ and $\Phi^{0,2}\circ\phitau$
  are isomorphic.
\end{prop}

\begin{proof}
  Going over the construction of equivalences of $\Phi^{0,2},\Phi^{2,0}$, we see that it suffices
  to construct the isomorphisms $\FT(E_0)\iso E_0$ (evident), and
  \begin{equation}
    \label{isomor}
    \FT(-_\lef*\bullet*-_\righ)\iso\iota(-_\lef)*\FT(\bullet)*\iota(-_\righ)
  \end{equation}
  (left and right convolution functors
  \[D_{\bG_\bO}(\Gr)\times D_{!*\bG_\bO}(\Gr\times V)\times D_{\bG_\bO}(\Gr)\to D_{!*\bG_\bO}(\Gr\times V)).\]
  Here $\iota\colon D_{\bG_\bO}(\Gr)\to D_{\bG_\bO}(\Gr)$ is a monoidal autoequivalence induced by
  the automorphism $\iota\colon g\mapsto{}^t\!g^{-1}$ of $\bG_\bF$. Note that the Satake equivalence
  intertwines $\iota\colon D_{\bG_\bO}(\Gr)\to D_{\bG_\bO}(\Gr)$ with the same named autoequivalence
  of $\Rep(\GL_N)$. Also note that $\iota\colon D_{\bG_\bO}(\Gr)\to D_{\bG_\bO}(\Gr)$ is induced by
  the automorphism $L\mapsto L^\perp$ of $\Gr$. Here $L^\perp:=\{v\in\bV : (v,L)\in\bO\}$, and
  $(,)$ stands for the $\bF$-bilinear pairing on $\bV$ such that $(e_i,e_j)=\delta_{ij}$. Now the
  existence of the desired isomorphism~(\ref{isomor}) follows from the definitions of $\FT$ and
  convolutions~(\ref{left conv},~\ref{right conv}).
\end{proof}


\subsection{The monoidal property of $\Phi^{1,1}$}
\label{proof main}
The argument is very similar to the one of~Section~\ref{20 monoidal}.
We introduce
\begin{multline*}\BC[\CH]_{1,1}^\bullet=\Sym(\Hom(V_1,V_2)[-1])\otimes\Sym(\Hom(V_2,V_1)[-1])
  \otimes\Sym(\Hom(V'_1,V_2)[-1])\\
  \otimes\Sym(\Hom(V_2,V'_1)[-1])\otimes\Sym(\Hom(V_1,V'_1)[-1])\otimes\Sym(\Hom(V'_1,V_1)[-1]).
  \end{multline*}
Then the monoidal structure $\otimes_{\fG_{1,1}^\bullet}$ on $D^{G_{\bar0}}_\perf(\fG_{1,1}^\bullet)$
is defined via the kernel $\BC[\Delta]_{1,1}^\bullet$: the diagonal $G_{\bar0}$-equivariant
dg-$\fG_{1,1}^\bullet$-trimodule.
The fusion monoidal structure $\star$ on
$D_{!*\bG_\bO}(\Gr\times\bV)$ transferred to $D^{G_{\bar0}}_\perf(\fG_{1,1}^\bullet)$ via the equivalence
$\Phi^{1,1}$ is also defined via a kernel $\sK^\bullet$ (a $G_{\bar0}$-equivariant
dg-$\fG_{1,1}^\bullet$-trimodule).
Note that we have a different equivariant structure than for $\CK^\bullet$ (notation
of~\S\ref{20 monoidal}) because of the
different compatibility with the Hecke action. We have to construct an isomorphism of
$G_{\bar0}$-equivariant dg-$\fG_{1,1}^\bullet$-trimodules $\BC[\Delta]_{1,1}^\bullet\iso\sK^\bullet$.

We have $\Phi^{1,1}\fG_{1,1}^\bullet\simeq E_0$. Also,
we have an isomorphism of endofunctors $E_0\star-\cong\Id\colon D_{!*\bG_\bO}(\Gr\times\bV)\to
D_{!*\bG_\bO}(\Gr\times\bV)$. Thus we obtain an isomorphism of functors
$\Phi^{1,1}(\fG_{1,1}^\bullet\otimes_{\fG_{1,1}^\bullet}-)\iso(\Phi^{1,1}\fG_{1,1}^\bullet)\star\Phi^{1,1}-$.
This isomorphism yields an isomorphism of kernels $\BC[\Delta_\forg]_{1,1}^\bullet\iso\sK_\forg^\bullet$
(notation is explained in~Section~\ref{20 monoidal}).

The following diagram commutes:
\[\xymatrix{
\IC_{(-1^N,1^N)}\star\Phi^{1,1} \ar[r]^-\sim \ar[d] &
\Phi^{1,1}\circ\big(\det V_1\otimes-\otimes\det\!{}^{-1}V'_1\big) \ar[d]^{\cdot\det A''}\\
E_0\star\Phi^{1,1} \ar[r]^\sim & \Phi^{1,1}.
}\]
Hence the diagram
\[\begin{CD}
\BC[\Delta]_{1,1}^\bullet @>{\sim}>> \sK_\forg^\bullet\\
@VV{\cdot\det A''}V @VV{\cdot\det A''}V\\
\BC[\Delta]_{1,1}^\bullet @>{\sim}>> \sK_\forg^\bullet
\end{CD}\]
commutes as well, and the rest of the argument proceeds just as in~Section~\ref{20 monoidal}.

This completes the proof of the monoidal property of $\Phi^{1,1}$ along
with~Theorem~\ref{main spherical mirabolic}.

\section{A coherent realization of $D_{\GL(N-1,\bO)}(\Gr)$}
\label{coh realization N-1}

\subsection{Notation}
We consider a complex vector space $\ol{V}\!_1$ with a basis $e_2,e_3,\ldots,e_N$.
We consider the Lie superalgebra $\fg\fl(N-1|N)=\fg\fl(\ol{V}\!_1\oplus\Pi V_2)$.
We have a decomposition $\fg\fl(N-1|N)=\ol\fg_{\bar0}\oplus\ol\fg_{\bar1}$, where
$\ol\fg_{\bar1}=\Pi\Hom(\ol{V}\!_1,V_2)\oplus\Pi\Hom(V_2,\ol{V}\!_1)$, and
$\ol\fg_{\bar0}=\End(\ol{V}\!_1)\oplus\End(V_2)$. We set
$\ol{G}_{\bar0}=\GL(\ol{V}\!_1)\times\GL(V_2)$.
We consider the dg-algebra $\ol\fG{}^\bullet=\Sym(\ol\fg_{\bar1}[-1])$
with trivial differential,
and the triangulated category $D^{\ol{G}_{\bar0}}_\perf(\ol\fG{}^\bullet)$ obtained by localization
(with respect to quasi-isomorphisms) of the category of perfect $\ol{G}_{\bar0}$-equivariant
dg-$\ol\fG{}^\bullet$-modules.
Finally, we set $\ol\bG_\bO=\GL(N-1,\bO)$.

\begin{thm}
  \label{GL(N-1)}
  There exists an equivalence of triangulated categories
  $\ol\Phi\colon D^{\ol{G}_{\bar0}}_\perf(\ol\fG{}^\bullet)\iso D_{\ol\bG_\bO}(\Gr)$
  commuting with the left convolution action of the monoidal
  spherical Hecke category $\Perv_{\GL(N-1,\bO)}(\Gr_{N-1})\cong\Rep(\GL_{N-1})$ and with
the right convolution action of the monoidal
  spherical Hecke category $\Perv_{\GL(N,\bO)}(\Gr_N)\cong\Rep(\GL_N)$.
\end{thm}

The proof will be given in~Section~\ref{coda} after some preparations
in~\S\S\ref{mirres},\ref{cohmirres}.

\bigskip

Similarly to~Section~\ref{koszul}, we define
$\ol{\Lambda}:=\Lambda\big(\Hom(\ol{V}\!_1,V_2)\oplus\Hom(V_2,\ol{V}\!_1)\big)$.
We consider the derived category $D_\fid^{\ol{G}_{\bar0}}(\ol{\Lambda})$ of finite dimensional
complexes of $\ol{G}_{\bar0}\ltimes\ol\Lambda$-modules. We have the Koszul equivalence functors
\[\ol{\varkappa}\colon D_\fid^{\ol{G}_{\bar0}}(\ol{\Lambda})\iso D^{\ol{G}_{\bar0}}_\perf(\ol\fG{}^\bullet),\
SD_\fid^{\ol{G}_{\bar0}}(\ol{\Lambda})\iso SD^{\ol{G}_{\bar0}}_\perf(\ol\fG{}^\bullet).\]
We also consider the degeneration $\ul\fgl(N-1|N)$ of the Lie superalgebra $\fgl(N-1|N)$
(defined as in~\S\ref{degen}),
and the derived category of bounded complexes of integrable $\ul\fgl(N-1|N)$-modules
$SD_{\on{int}}(\ul\fgl(N-1|N))\cong SD_\fid^{\ol{G}_{\bar0}}(\ol{\Lambda})$.
The following corollary of Theorem~\ref{GL(N-1)} is proved just like~Corollary~\ref{exact spherical}.

\begin{cor}
  \label{exact GL(N-1)}
  \textup{(a)} The composed equivalence
  \[\ol{\Phi}\circ\ol{\varkappa}\colon 
  D_\fid^{\ol{G}_{\bar0}}(\ol{\Lambda})\iso D_{\ol\bG_\bO}(\Gr)\]
  is exact with respect to the tautological $t$-structure on
  $D_\fid^{\ol{G}_{\bar0}}(\ol{\Lambda})$
  and the perverse $t$-structure on $D_{\ol\bG_\bO}(\Gr)$.

  \textup{(b)} This equivalence is monoidal with respect to the tensor structure on
  $D_\fid^{\ol{G}_{\bar0}}(\ol{\Lambda})$ and the fusion $\star$ on $D_{\ol\bG_\bO}(\Gr)$.

  \textup{(c)} The equivariant derived category $D_{\ol\bG_\bO}(\Gr)$ is equivalent to the
  bounded derived category of the abelian category $\Perv_{\ol\bG_\bO}(\Gr)$.
\end{cor}

In case $N=2$, both~Theorem~\ref{GL(N-1)} and~Corollary~\ref{exact GL(N-1)}(a) were proved
in~\cite{br} by a rather different argument.

\subsection{Constructible mirabolic restriction}
\label{mirres}
Clearly, $\bV_0\setminus t\bV_0$ is a single $\bG_\bO$-orbit, and the stabilizer of the vector
  $e_1\in \bV_0\setminus t\bV_0$ is the mirabolic subgroup $\bM_\bO\subset\bG_\bO$. Hence
$D_{\bG_\bO}(\Gr\times(\bV_0\setminus t\bV_0))\cong D_{\bM_\bO}(\Gr)$. We will denote
$\Gr=\Gr_{\GL_N}$ by $\Gr_N$ to distinguish it from $\Gr_{N-1}:=\Gr_{\GL_{N-1}}$. We will also
denote by $\Gr_N^\subset\subset\Gr_N$ (resp.\ $\Gr_{N-1}^\subset\subset\Gr_{N-1}$) the 
closed subvariety classifying the sublattices in the standard one $\bV_0$ (resp.\ in
$\ol\bV\!_0:=\bO e_2\oplus\bO e_3\oplus\ldots\oplus\bO e_N$).
It is a union of the Schubert varieties numbered by the {\em negative} partitions of length
$\leq N$ (resp. $\leq N-1$). The category $\Perv_{\GL(N,\bO)}(\Gr_N^\subset)$
(resp.\ $\Perv_{\GL(N-1,\bO)}(\Gr_{N-1}^\subset)$) is monoidal and the Satake equivalence takes it
to the monoidal category of polynomial representations of $\GL_N$ (resp.\ $\GL_{N-1}$).

We have a closed embedding $\varsigma\colon\Gr_{N-1}^\subset\hookrightarrow\Gr_N^\subset,\
(\ol{L}\subset\ol\bV\!_0)\mapsto(\bO e_1\oplus\ol{L}\subset\bV_0)$.

\begin{lem}
  \label{monores}
  \textup{(a)} The functor
  $\varsigma^!\colon D_{\GL(N,\bO)}(\Gr_N^\subset)\to D_{\GL(N-1,\bO)}(\Gr_{N-1}^\subset)$ is monoidal.

  \textup{(b)} For a negative partition $\blambda$ we have
  $\varsigma^!\IC_\blambda=\IC_\blambda[|\blambda|]$ in notation of~Section~\ref{inth}
    (i.e.\ if $\lambda_1<0$, then $\varsigma^!\IC_\blambda=0$, and if $\lambda_1=0$, then
    $\varsigma^!\IC_\blambda=\IC_{(\lambda_2\geq\ldots\geq\lambda_N)}[\lambda_2+\ldots+\lambda_N]$).
\end{lem}

\begin{proof}
  (a) The Grassmannian of sublattices $\Gr_N^\subset$ (resp.\ $\Gr_{N-1}^\subset$) is the union of
  connected components $\bigsqcup_{n\in\BN}\Gr_N^{\subset,(-n)}$
  (resp.\ $\bigsqcup_{n\in\BN}\Gr_{N-1}^{\subset,(-n)}$)
  parametrizing sublattices of codimension $n$. Clearly,
  $\varsigma(\Gr_{N-1}^{\subset,(-n)})\subset\Gr_N^{\subset,(-n)}$. Moreover,
  $\varsigma(\Gr_{N-1}^{\subset,(-n)})$ is a
  connected component of the intersection of the Levi Grassmannian
  $\Gr_{\GL_1\times\GL_{N-1}}\subset\Gr_N$ with $\Gr_N^{\subset,(-n)}$.
  We have the monoidal functor of hyperbolic restriction to the Levi Grassmannian, see
  e.g.~\cite[5.3.28 at page 213]{bd}. Since
$(\bM_\bF\cdot\varsigma(\Gr_{N-1}^{\subset,(-n)}))\cap\Gr_N^{\subset,(-n)}=\varsigma(\Gr_{N-1}^{\subset,(-n)})$,
  the hyperbolic restriction to the component $\varsigma(\Gr_{N-1}^{\subset,(-n)})$ coincides with the
  corestriction $\varsigma^!$.

  (b) We have $(\BS_\blambda V_1^*)^{\GL_1}=\BS_\blambda\ol{V}{}^*\!\!\!\!_1$. Now the desired claim
  follows from~\cite[Proposition~5.3.29]{bd}.
\end{proof}

Recall that the monoidal Hecke category $D_{\GL(N,\bO)}(\Gr_N)=D_{\bG_\bO}(\Gr)$
acts by the left convolution on $D_{!*\bG_\bO}(\Gr\times\bV)$. In particular, the monoidal subcategory
$D_\bGO(\Gr_N^\subset)\subset D_\bGO(\Gr)$ acts on $D_{!*\bG_\bO}(\Gr\times\bV)$.
We define the action of $D_\bGO(\Gr_N^\subset)$ by the left convolution on
$D_{\bG_\bO}(\Gr\times(\bV_0\setminus t\bV_0))$ as follows:
$\CA*\CF:=\jmath_0^!(\CA*\jmath_{0*}\CF)$
(notation of~Section~\ref{Restriction}).
Also, the monoidal Hecke category $D_{\GL(N-1,\bO)}(\Gr_{N-1})$ acts by the left
convolution on $D_{\GL(N-1,\bO)}(\Gr_N)$. Finally, we denote by
$\Res_{\bM_\bO}^{\GL(N-1,\bO)}\colon D_{\bM_\bO}(\Gr_N)\to D_{\GL(N-1,\bO)}(\Gr_N)$ the functor
of restriction of equivariance under $\GL(N-1,\bO)\hookrightarrow\bM_\bO$.

\begin{lem}
  \label{MGL}
  We have a compatible system of isomorphisms for all $\CA\in D_{\GL(N,\bO)}(\Gr_N^\subset)$
  and $\CF\in D_{\bG_\bO}(\Gr\times(\bV_0\setminus t\bV_0))\cong D_{\bM_\bO}(\Gr)$:
  \[\Res_{\bM_\bO}^{\GL(N-1,\bO)}(\CA*\CF)\iso\varsigma^!\CA*\Res_{\bM_\bO}^{\GL(N-1,\bO)}(\CF).\]
\end{lem}

\begin{proof}
  Comparison of definitions.
\end{proof}

\subsection{Coherent mirabolic restriction}
\label{cohmirres}
Recall the subcategory
$D_\perf^{G_{\bar0},\geq}(\fG_{1,1}^\bullet)\subset D_\perf^{G_{\bar0}}(\fG_{1,1}^\bullet)$ generated by
$\{V_{1,\blambda}\otimes\fG_{1,1}^\bullet\otimes V_{2,\bmu}\mid\bla\geq0\}$ introduced in the
proof of~Proposition~\ref{corestriction}(a). The equivalence
$\Phi^{1,1}\colon D_\perf^{G_{\bar0}}(\fG_{1,1}^\bullet)\iso D_{!*\bG_\bO}(\Gr\times\bV)$
restricts to the same named equivalence
$D_\perf^{G_{\bar0},\geq}(\fG_{1,1}^\bullet)\iso D_{!*\bG_\bO}(\Gr\times\bV_0)$, and the functor
$p_*q^*\colon D_\perf^{G_{\bar0},\geq}(\fG_{1,1}^\bullet)\to D_\perf^{G_{\bar0},\geq}(\fG_{1,1}^\bullet)$
(notation of~Section~\ref{Restriction}) is isomorphic to $\Id$. Thus we have a natural
morphism $\Id\cong p_*q^*\to p_{0*}q_0^*$ (notation of~Section~\ref{Restriction}) of endofunctors
of $D_\perf^{G_{\bar0},\geq}(\fG_{1,1}^\bullet)$. Composing with another copy of $p_{0*}q_0^*$ we
obtain a morphism $p_{0*}q_0^*\to p_{0*}q_0^*\circ p_{0*}q_0^*$ of endofunctors
of $D_\perf^{G_{\bar0},\geq}(\fG_{1,1}^\bullet)$ that is easily seen to be an isomorphism. Indeed,
according to~Proposition~\ref{corestriction}(c), $\Phi^{1,1}$ takes $p_{0*}q_0^*$ to
$\jmath_{0*}\jmath_0^!$, and $\jmath_{0*}\jmath_0^!\iso\jmath_{0*}\jmath_0^!\circ\jmath_{0*}\jmath_0^!$.
Inverting the isomorphism $p_{0*}q_0^*\iso p_{0*}q_0^*\circ p_{0*}q_0^*$, we obtain an isomorphism
$p_{0*}q_0^*\circ p_{0*}q_0^*\iso p_{0*}q_0^*$ that, together with the morphism $\Id\to p_{0*}q_0^*$
equips $p_{0*}q_0^*$ with a structure of (idempotent) monad in $D_\perf^{G_{\bar0},\geq}(\fG_{1,1}^\bullet)$.

  We denote the dg-algebra with trivial
  differential $\fG_{1,1}^\bullet\otimes\BC[\Hom_{\leq N-1}(V_1,V'_1)]$ (the grading on
  $\BC[\Hom_{\leq N-1}(V_1,V'_1)]$ is trivial) by $\fF^\bullet$
  (here $\Hom_{\leq N-1}(V_1,V'_1)\subset\Hom(V_1,V'_1)$
stands for the subvariety formed by the noninvertible homomorphisms).
It is acted upon by $G_\CQ$.
  By~Theorem~\ref{main spherical mirabolic} and~Proposition~\ref{corestriction},
  we have an equivalence of categories $\Phi'$ from the Kleisli category
  (see~\cite[VI.5]{mac}) $D(p_{0*}q_0^*)$ of
  the monad $p_{0*}q_0^*$ in $D^{G_{\bar0},\geq}_\perf(\fG_{1,1}^\bullet)$ to
  $D_{\bG_\bO}(\Gr\times(\bV_0\setminus t\bV_0))$.
  For the modules $V_{1,\blambda}\otimes\fG_{1,1}^\bullet\otimes V_{2,\bmu},\
  V_{1,\blambda'}\otimes\fG_{1,1}^\bullet\otimes V_{2,\bmu'}$ over the monad $p_{0*}q_0^*$
  (here both $\blambda,\blambda'$ are partitions), we have
  \[\Hom_{D(p_{0*}q_0^*)}(V_{1,\blambda}\otimes\fG_{1,1}^\bullet\otimes V_{2,\bmu},
  V_{1,\blambda'}\otimes\fG_{1,1}^\bullet\otimes V_{2,\bmu'})=
  \big(V_{1,\blambda}^*\otimes V'_{1,\blambda'}\otimes\fF^\bullet\otimes(V^*_{2,\bmu}\otimes
  V_{2,\bmu'})\big)^{G_\CQ}.\]

  Recall the equivalence
  $D_{\bG_\bO}(\Gr\times(\bV_0\setminus t\bV_0))\cong D_{\bM_\bO}(\Gr)$. By an abuse of notation,
  we will denote the composed equivalence $D(p_{0*}q_0^*)\iso D_{\bM_\bO}(\Gr)$
  also by $\Phi'$.

  On the other hand, we consider a full subcategory
  $D^{\ol{G}_{\bar0},\geq0}_\perf(\ol\fG{}^\bullet)\subset D^{\ol{G}_{\bar0}}_\perf(\ol\fG{}^\bullet)$
  generated by the free modules $\ol{V}\!_{1,\blambda}\otimes\ol\fG{}^\bullet\otimes V_{2,\bmu}$
  for partitions $\blambda$ (of length $\leq~N-1$). We construct an equivalence
  $\varPhi\colon D(p_{0*}q_0^*)\iso D^{\ol{G}_{\bar0},\geq0}_\perf(\ol\fG{}^\bullet)$ as follows.
We consider the variety $\CP$ of sixtuples \[A\in\Hom(V_1,V_2),\ B\in\Hom(V_2,V_1),\
\ol{A}\in\Hom(\ol{V}\!_1,V_2),\] \[\ol{B}\in\Hom(V_2,\ol{V}\!_1),\ b\in\Hom(V_1,\ol{V}\!_1),\
a\in\Hom(\ol{V}\!_1,V_1),\]
such that $A=\ol{A}b,\ \ol{B}=bB,\ a=B\ol{A}.$
Clearly, \[\CP\simeq\Hom(V_1,\ol{V}\!_1)\times\Hom(\ol{V}\!_1,V_2)\times\Hom(V_2,V_1).\]
\[
\xymatrix @C=2.2em @R=1.2em{
& V_1\ar@<-0.25ex>@{-->}@/^/[dr]_A \ar@<-0.25ex>@/_/[dl]_b \\
  \ol{V}\!_1 \ar@<-0.5ex>[rr]_{\ol A} \ar@{-->}@<-0,25ex>@/^/[ur]_{a}  && V_2 \ar@<-0.25ex>@/_/[ul]_B
  \ar@{-->}[ll]_{\ol B}
}
\]
We have the natural morphisms
\[\sfp\colon\CP\to\Hom(V_1,V_2)\times\Hom(V_2,V_1)=\Pi\fg^*_{\bar1},\
\sq\colon\CP\to\Hom(\ol{V}\!_1,V_2)\times\Hom(V_2,\ol{V}\!_1)=\Pi\ol\fg{}^*_{\bar1};\]
\[\sfp(A,\ol{A},a,b,B,\ol{B})=(A,B),\ \sq(A,\ol{A},a,b,B,\ol{B})=(\ol{A},\ol{B}).\]
The variety $\CP$ is acted upon by
$G_\CP:=\GL(V_1)\times\GL(\ol{V}\!_1)\times\GL(V_2)$: \[(g_1,\bar{g}_1,g_2)(A,\ol{A},a,b,B,\ol{B})=
(g_2Ag_1^{-1},g_2\ol{A}\bar{g}{}^{-1}_1,g_1a\bar{g}_1^{-1},\bar{g}_1bg_1^{-1},g_1Bg_2^{-1},\bar{g}_1\ol{B}g_2^{-1}).\]
We have the natural morphisms \[\sfp\colon G_\CP\to G_{\bar0},\ \sfp(g_1,\bar{g}_1,g_2)=(g_1,g_2);\
\sq\colon G_\CP\to\ol{G}_{\bar0},\ \sq(g_1,\bar{g}_1,g_2)=(\bar{g}_1,g_2).\] Clearly, the morphisms
$\sfp\colon\CP\to\Pi\fg^*_{\bar1},\ \sq\colon\CP\to\Pi\ol\fg{}^*_{\bar1}$ are equivariant with
respect to $\sfp\colon G_\CP\to G_{\bar0},\ \sq\colon G_\CP\to\ol{G}_{\bar0}$.
Hence we have the convolution functor
\[\sq_*\sfp^*\colon\Coh^{G_{\bar0}}(\Pi\fg^*_{\bar1})=
\Coh(G_{\bar0}\backslash\Pi\fg^*_{\bar1})\xrightarrow{\sfp^*}
\Coh(G_\CP\backslash\CP)\xrightarrow{\sq_*}\Coh(\ol{G}_{\bar0}\backslash\Pi\ol\fg{}^*_{\bar1})=
\Coh^{\ol{G}_{\bar0}}(\Pi\ol\fg{}^*_{\bar1})\]
(in particular, $\sq_*$ involves taking $\GL(V_1)$-invariants).
We will actually need the same named functor $\sq_*\sfp^*\colon
D_\perf^{G_{\bar0}}(\fG_{1,1}^\bullet)\to D_\perf^{\ol{G}_{\bar0}}(\ol\fG{}^\bullet)$
defined similarly using the dg-algebra with trivial differential
\[\Sym\big(\Hom(V_1,V_2)[-1]\oplus\Hom(\ol{V}\!_1,V_1)[-1]\big)\otimes\Sym\Hom(\ol{V}\!_1,V_1)\]
(the grading on $\Sym\Hom(\ol{V}\!_1,V_1)$ is trivial, and if we disregard the grading, then
$\Sym\big(\Hom(V_1,V_2)[-1]\oplus\Hom(\ol{V}\!_1,V_1)[-1]\big)\otimes\Sym\Hom(\ol{V}\!_1,V_1)
\simeq\BC[\CP]$).

Now recall that $D(p_{0*}q_0^*)$ is a full subcategory of $D_\perf^{G_{\bar0}}(\fG_{1,1}^\bullet)$.
The desired functor $\varPhi\colon D(p_{0*}q_0^*)\to D^{\ol{G}_{\bar0}}_\perf(\ol\fG{}^\bullet)$
is nothing but the restriction of $\sq_*\sfp^*$ to the full subcategory
$D(p_{0*}q_0^*)\subset D_\perf^{G_{\bar0}}(\fG_{1,1}^\bullet)$. The full subcategory $D(p_{0*}q_0^*)$
is generated by the objects $p_{0*}q_0^*(V_{1,\blambda}\otimes\fG_{1,1}^\bullet\otimes V_{2,\bmu})
\cong_{D(p_{0*}q_0^*)} V_{1,\blambda}\otimes\fG_{1,1}^\bullet\otimes V_{2,\bmu}$ (the LHS and the RHS
are isomorphic in the Kleisli category $D(p_{0*}q_0^*)$)
with $\blambda$ running through the set of partitions. For a partition $\blambda$ we have
\[\varPhi(V_{1,\blambda}\otimes\fG_{1,1}^\bullet\otimes V_{2,\bmu})=
\ol{V}\!_{1,\blambda}\otimes\ol\fG{}^\bullet\otimes V_{2,\bmu}\]
in notation of~Section~\ref{inth} (i.e.\ if
$\blambda=(\lambda_1\geq\lambda_2\geq\ldots\geq\lambda_N\geq0)$ and $\lambda_N>0$, then
$\varPhi(V_{1,\blambda}\otimes\fG_{1,1}^\bullet\otimes V_{2,\bmu})=0$, and if $\lambda_N=0$, then
\[\varPhi(V_{1,\blambda}\otimes\fG_{1,1}^\bullet\otimes V_{2,\bmu})=
\ol{V}\!_{1,(\lambda_1\geq\ldots\geq\lambda_{N-1})}\otimes\ol\fG{}^\bullet\otimes V_{2,\bmu}.)\]
Hence $\varPhi(V_{1,\blambda}\otimes\fG_{1,1}^\bullet\otimes V_{2,\bmu})$ actually lies in
$D^{\ol{G}_{\bar0},\geq0}_\perf(\ol\fG{}^\bullet)\subset D^{\ol{G}_{\bar0}}_\perf(\ol\fG{}^\bullet)$.
It remains to check that $\varPhi$ is fully faithful.

For partitions $\blambda,\blambda'$, we have to check that the following
morphism is an isomorphism:
  \begin{multline*}
    \big(V_{1,\blambda}^*\otimes p_{0*}q_0^*(V'_{1,\blambda'}\otimes\BC[\Hom(V'_1,V_2)\times\Hom(V_2,V'_1)]
    \otimes V_{2,\bmu'})\otimes V_{2,\bmu}^*\big)^{\GL(V_1)\times\GL(V_2)}\\
    :=\big(V_{1,\blambda}^*\otimes V'_{1,\blambda'}\otimes\BC[\CQ_0]\otimes(V_{2,\bmu}^*\otimes
    V_{2,\bmu'})\big)^{G_\CQ}
  =\big(\BS_\blambda V_1^*\otimes\BS_{\blambda'}V_1'\otimes\BC[\CQ_0]
  \otimes(V_{2,\bmu}^*\otimes V_{2,\bmu'})\big)^{G_\CQ}\\
  \iso\big(\BS_\blambda\ol{V}{}_{\!\!1}^*\otimes\BS_{\blambda'}\ol{V}{}_{\!\!1}\otimes
  \BC[\Hom(\ol{V}\!_1,V_2)\times\Hom(V_2,\ol{V}\!_1)]
  \otimes(V_{2,\bmu}^*\otimes V_{2,\bmu'})\big)^{\GL(\ol{V}\!_1)\times\GL(V_2)}.
  \end{multline*}
\[
\xymatrix @C=2.2em @R=1.2em{
& V_1\ar@{-->}@/^/[dr]_A \ar@<-0.25ex>[dd]_(.3){A''\!\!} \ar@<-0.5ex>@/_/[dl]_b\\
  \ol{V}\!_1 \ar@<-0.5ex>@/_/[dr]_{a'} \ar@<-0.25ex>@{.>}[rr]_(.25){\ol A} \ar@{-->}@/^/[ur]_(.7){\!\!\!a}
  && V_2\ar@<-0.5ex>@/_/[ul]_B \ar@{-->}@/^/[dl]_{B'\!}
  \ar@{.>}@<-0.25ex>[ll]_(.75){\ol B}\\
& V'_1\ar@<-0.5ex>@/_/[ur]_{A'} \ar@<-0.25ex>@{-->}[uu]_(.7){\!\!B''} \ar@{-->}@/^/[ul]_(.3){\!\!b'}
}
\]
The desired isomorphism is equal to the composition of the following three. The first one
is induced by \[\BC[\Hom_{\leq N-1}(V_1,V'_1)]\iso\big(\BC[\Hom(V_1,\ol{V}\!_1)]\otimes
\BC[\Hom(\ol{V}\!_1,V'_1)]\big)^{\GL(\ol{V}\!_1)},\] see Lemma~\ref{IT}(a).
The second one is induced by the inverse of
\[\BS_{\blambda'}\ol{V}\!_1\otimes\BC[\Hom(\ol{V}\!_1,V_2)]\iso
\big(\BC[\Hom(\ol{V}\!_1,V'_1)\otimes\BS_{\blambda'}V'_1\otimes\BC[\Hom(V'_1,V_2)]\big)^{\GL(V'_1)},\]
see Lemma~\ref{IT}(b).
The third one is induced by the inverse of
\[\BC[\Hom(V_2,\ol{V}\!_1)]\otimes\BS_{\blambda^*}\ol{V}\!_1\iso
\big(\BC[\Hom(V_2,V_1)\otimes\BS_{\blambda^*}V_1\otimes
    \BC[\Hom(V_1,\ol{V}\!_1)]\big)^{\GL(V_1)},\] see Lemma~\ref{IT}(c).

\subsection{Proof of Theorem~\ref{GL(N-1)}}
\label{coda}
We consider an $\bF$-linear automorphism $\xi$ of $\bV\colon \xi(e_1)=e_1,\ \xi(e_i)=te_i$
for $i=2,\ldots,N$, and the same named induced automorphism of $\Gr$.
It is given by the action of an element $\on{diag}(1,t,\ldots,t)$ of the diagonal Cartan torus
$\bT_\bF$. Note that $\on{diag}(1,t,\ldots,t)\bM_\bO\on{diag}(1,t^{-1},\ldots,t^{-1})\supset\bM_\bO$.
Hence $\xi_*$ acts on the equivariant category $D_{\bM_\bO}(\Gr)$ (by restricting equivariance
from $\on{diag}(1,t,\ldots,t)\bM_\bO\on{diag}(1,t^{-1},\ldots,t^{-1})$ to $\bM_\bO$).
Clearly, $\xi_*$ acts on the equivariant category $D_{\GL(N-1,\bO)}(\Gr)$ as well. Recall that
$\Res_{\bM_\bO}^{\GL(N-1,\bO)}\colon D_{\bM_\bO}(\Gr)\to D_{\GL(N-1,\bO)}(\Gr)$ denotes the functor
of restriction of equivariance under $\GL(N-1,\bO)\hookrightarrow\bM_\bO$.
We denote by $\Av^{\bM_\bO}_{\GL(N-1,\bO)}\colon D_{\GL(N-1,\bO)}(\Gr)\to D_{\bM_\bO}(\Gr)$ the
corresponding right adjoint $*$-averaging functor.

\begin{lem}
  \label{loca}
  \textup{(a)} Given $\CF\in D_{\GL(N-1,\bO)}(\Gr)$, for $n\gg0$ the canonical morphism
  $\Res_{\bM_\bO}^{\GL(N-1,\bO)}\Av^{\bM_\bO}_{\GL(N-1,\bO)}\xi^n_*\CF\to\xi^n_*\CF$ is an isomorphism.

  \textup{(b)} $\xi_*$ is an auto-equivalence of $D_{\GL(N-1,\bO)}(\Gr)$, and hence the
  natural functor from $D_{\GL(N-1,\bO)}(\Gr)$ to
  $\on{colim}_{\xi_*}D_{\GL(N-1,\bO)}(\Gr)$ is an equivalence.

  \textup{(c)} The restriction of equivariance functor
  $\Res_{\bM_\bO}^{\GL(N-1,\bO)}\colon D_{\bM_\bO}(\Gr)\to D_{\GL(N-1,\bO)}(\Gr)$ induces the same named
  equivalence of the colimits
  $\Res_{\bM_\bO}^{\GL(N-1,\bO)}\colon \on{colim}_{\xi_*}D_{\bM_\bO}(\Gr)\to
  \on{colim}_{\xi_*}D_{\GL(N-1,\bO)}(\Gr)\cong D_{\GL(N-1,\bO)}(\Gr)$.
\end{lem}

\begin{proof}
  It suffices to prove (a) for an irreducible perverse $\CF\in D_{\GL(N-1,\bO)}(\Gr)$. This in
  turn follows from the fact that for any $\GL(N-1,\bO)$-orbit $\BO\subset\Gr$ and $n\gg0$,
  the shift $\xi^{-n}\BO$ becomes $\bM_\bO$-invariant. Indeed, $\bM_\bO$ is generated by
  its radical $\bU_\bO$ and $\GL(N-1,\bO)$. As $n$ grows,
  $\bU_\bO^{(n)}:=\on{diag}(1,t^{-n},\ldots,t^{-n})\bU_\bO\on{diag}(1,t^n,\ldots,t^n)$ forms a system of
  shrinking subgroups of $\bU_\bO$, and we take $n$ big enough so that the action of
  $\bU_\bO^{(n)}$ on $\BO$ is trivial (recall that the action of $\bU_\bO$ on any Schubert subvariety
  of $\Gr$ factors through a finite dimensional quotient group).

  (b) is evident, and (c) follows from (a) and (b).
\end{proof}

On the coherent side we consider an endofunctor $\eta$ of
$D_\perf^{\ol{G}_{\bar0},\geq0}(\ol\fG{}^\bullet)$ obtained by tensoring with the
polynomial representation $\det(\ol{V}\!_1)$ of $\GL(\ol{V}\!_1)$.
The full embedding $D_\perf^{\ol{G}_{\bar0},\geq0}(\ol\fG{}^\bullet)\hookrightarrow
D_\perf^{\ol{G}_{\bar0}}(\ol\fG{}^\bullet)$ gives rise to an equivalence of colimits
$\on{colim}_\eta D_\perf^{\ol{G}_{\bar0},\geq0}(\ol\fG{}^\bullet)\iso
\on{colim}_\eta D_\perf^{\ol{G}_{\bar0}}(\ol\fG{}^\bullet)\cong D_\perf^{\ol{G}_{\bar0}}(\ol\fG{}^\bullet)$.

Composing the inverse of $\varPhi$ with $\Phi'$ (notation of~Section~\ref{cohmirres}) we obtain
an equivalence $\ol{\Phi}{}'\colon D_\perf^{\ol{G}_{\bar0},\geq0}(\ol\fG{}^\bullet)\iso D_{\bM_\bO}(\Gr)$.
According to~Lemma~\ref{MGL}, the equivalence
$\ol{\Phi}{}'\colon D_\perf^{\ol{G}_{\bar0},\geq0}(\ol\fG{}^\bullet)\iso D_{\bM_\bO}(\Gr)$ intertwines
the endofunctors $\eta$ and $\xi_*$, and hence induces the desired equivalence
$\ol{\Phi}\colon D_\perf^{\ol{G}_{\bar0}}(\ol\fG{}^\bullet)\iso D_{\GL(N-1,\bO)}(\Gr)$ between
the colimits. Theorem~\ref{GL(N-1)} is proved.
$\hfill \Box$

\section{Loop rotation and quantization}
\label{loop rotation}

\subsection{Graded differential operators and convolutions}
\label{loop}
We have $H^\bullet_{\BG_m}(\pt)=\BC[\hbar]$. We consider the algebra $\fD$
of ``graded differential operators'' on $\Hom(V_2,V_1)$: a $\BC[\hbar]$-algebra
generated by $\Hom(V_2,V_1)$ and $\Hom(V_1,V_2)$ with relations $[h,h']=[f,f']=0,\
[h,f]=\langle h,f\rangle\hbar$ for $h,h'\in\Hom(V_2,V_1),\ f,f'\in\Hom(V_1,V_2)$.
It is equipped with the grading $\deg f=\deg h=1,\ \deg\hbar=2$. We denote by
$\fD_{1,1}^\bullet$ this graded algebra viewed as a dg-algebra with trivial differential.
We will also need two more versions of $\fD_{1,1}^\bullet$ differing by the gradings of
generators: in $\fD_{0,2}^\bullet$ we set $\deg f=0,\ \deg h=2$, while in $\fD_{2,0}^\bullet$
we set $\deg f=2,\ \deg h=0$.

Recall the setup and notation of~Section~\ref{conv coh}. So we have another copy $V'_1$
of $V_1$, we rebaptize the algebra $\fD$ of ``graded differential operators'' on
$\Hom(V_2,V_1)$ as $\fD_{21}$, and along with it we consider $\fD_{21'}$ (``graded differential
operators'' on $\Hom(V_2,V'_1)$) and $\fD_{1'1}$ (``graded differential
operators'' on $\Hom(V'_1,V_1)$). Note that $\fD_{21}\cong\fD_{12}$ (``graded differential
operators'' on $\Hom(V_1,V_2)$), and similarly $\fD_{21'}\cong\fD_{1'2}$, and $\fD_{1'1}\cong\fD_{11'}$.

We have morphisms \[\sfm^A\colon\Hom(V'_1,V_2)\times\Hom(V_1,V'_1)\to\Hom(V_1,V_2),\
(A',A'')\mapsto A'A'',\]
\[\sfm^B\colon\Hom(V'_1,V_1)\times\Hom(V_2,V'_1)\to\Hom(V_2,V_1),\ (B'',B')\mapsto B''B'.\]
They give rise to the functors
\begin{multline*}
  D^{\GL(V'_1)\times\GL(V_2)}_\perf(\fD_{1'2})\times
D^{\GL(V_1)\times\GL(V'_1)}_\perf(\fD_{11'})\to D^{\GL(V_1)\times\GL(V_2)}_\perf(\fD_{12}),\\
(\CM_{1'2},\CM_{11'})\mapsto\big(\sfm^A_*(\CM_{1'2}\boxtimes\CM_{11'})\big)^{\GL(V'_1)},
\end{multline*}
\begin{multline*}
  D^{\GL(V'_1)\times\GL(V_1)}_\perf(\fD_{1'1})\times
D^{\GL(V_2)\times\GL(V'_1)}_\perf(\fD_{21'})\to D^{\GL(V_2)\times\GL(V_1)}_\perf(\fD_{21}),\\
(\CM_{1'1},\CM_{21'})\mapsto\big(\sfm^B_*(\CM_{1'1}\boxtimes\CM_{21'})\big)^{\GL(V'_1)}
\end{multline*}
(here $\sfm^A_*,\sfm^B_*$ stand for the direct images in the category of $\fD$-modules).

We will actually need the corresponding functors on the categories of dg-modules
\[\srel{A}*\colon D^{G_{\bar0}}_\perf(\fD_{2,0}^\bullet)\times D^{G_{\bar0}}_\perf(\fD_{2,0}^\bullet)\to
D^{G_{\bar0}}_\perf(\fD_{2,0}^\bullet),\
\srel{B}*\colon D^{G_{\bar0}}_\perf(\fD_{0,2}^\bullet)\times D^{G_{\bar0}}_\perf(\fD_{0,2}^\bullet)\to
D^{G_{\bar0}}_\perf(\fD_{0,2}^\bullet).\]

The multiplicative group $\BG_m$ acts on $\Gr\times\oV$ as loop rotations.
The goal of this section is the following

\begin{thm}
  \label{main quantum}
  There exist monoidal equivalences of triangulated categories
  \begin{equation}
    \begin{CD}
      \big(D^{G_{\bar0}}_\perf(\fD_{2,0}^\bullet),\ \srel{A}*\big) @>{\sim}>{\Phi^{2,0}_\hbar}>
      \big(D_{!\bG_\bO\rtimes\BC^\times}(\Gr\times\oV),\ \srel!\oast\big)\\
      @V{\wr}V{\varrho_\righ}V @V{\wr}V{\varrho_\righ}V\\
      D^{G_{\bar0}}_\perf(\fD_{1,1}^\bullet) @>{\sim}>{\Phi^{1,1}_\hbar}>
      D_{!*\bG_\bO\rtimes\BC^\times}(\Gr\times\oV)\\
      @V{\wr}V{\varrho_\righ}V @V{\wr}V{\varrho_\righ}V\\
      \big(D^{G_{\bar0}}_\perf(\fD_{0,2}^\bullet),\ \srel{B}*\big) @>{\sim}>{\Phi^{0,2}_\hbar}>
      \big(D_{*\bG_\bO\rtimes\BC^\times}(\Gr\times\oV),\ \srel**\big).
    \end{CD}
  \end{equation}
  (the vertical and the middle row equivalences are {\em not} monoidal).
  The horizontal equivalences commute with the actions of the monoidal spherical Hecke category
  $\Perv_{\bG_\bO\rtimes\BC^\times}(\Gr)\cong\Rep(\GL_N)$ by the left and right convolutions.
\end{thm}

The proof will be given in~Section~\ref{q-Restriction}.

\subsection{Construction of equivalences}
  We set $\fE_\hbar^\bullet:=\Ext^\bullet_{D^{\deeq}_{!*\bG_\bO\rtimes\BG_m}(\Gr\times\ooV)}(E_0,E_0)$
  (a dg-algebra with trivial differential).
Since it is an Ext-algebra in the deequivariantized category, it is automatically
equipped with an action of $\GL_N\times\GL_N=\GL(V_1)\times\GL(V_2)=G_{\bar0}$, and we can
consider the corresponding triangulated category $D^{G_{\bar0}}_\perf(\fE_\hbar^\bullet)$.
Similarly to~Lemma~\ref{purity}, there is a canonical equivalence
$D^{G_{\bar0}}_\perf(\fE_\hbar^\bullet)\iso D_{!*\bG_\bO\rtimes\BG_m}(\Gr\times\oV)$.
It remains to construct an isomorphism $\phi_\hbar^\bullet\colon \fD_{1,1}^\bullet\iso\fE_\hbar^\bullet$.

Note that $\fE_\hbar^\bullet$ is a $\BC[\hbar]$-algebra, and $\fE_\hbar^\bullet/(\hbar=0)=
\fE^\bullet=\fG_{1,1}^\bullet=\Sym(\fg_{\bar1}[-1])$,
so that $\fE^\bullet$ acquires a Poisson bracket from this deformation. We claim that this
Poisson bracket arises from the canonical symplectic form on $\fg_{\bar1}^*$. Indeed,
the isomorphism $\phi^*$ (see the proof of~Lemma~\ref{phi isom}) over
$\Isom(V_2,V_1)\times\Hom(V_1,V_2)$ and $\Hom(V_2,V_1)\times\Isom(V_1,V_2)$ is Poisson.
Here the Poisson structure on these open subsets arises from the deformations
$\fA_\hbar^\bullet,\fB_\hbar^\bullet$ which in turn arise from the loop-rotation-equivariant
Satake category $D_{\bG_\bO\rtimes\BG_m}(\Gr)$. The corresponding Poisson brackets are the
standard ones on the cotangent bundles $T^*\Isom(V_2,V_1),T^*\Isom(V_1,V_2)$ as follows
from~\cite[Theorem~5]{bf}.

Now $\fD_{1,1}^\bullet$ is a unique graded $\BC[\hbar]$-algebra with $\fD_{1,1}^\bullet/(\hbar=0)=
\Sym(\fg_{\bar1}[-1])$ such that the corresponding Poisson bracket on
$\Sym(\fg_{\bar1}[-1])$ is the standard one. Thus $\phi_\hbar^\bullet$ and $\Phi^{1,1}_\hbar$
are constructed.

\subsection{Restriction to $\Gr\times\bV_0$ with loop rotation}
\label{q-Restriction}
Similarly to~Section~\ref{Restriction}, our goal in this section is a description in terms of
equivalence of~Theorem~\ref{main quantum} of the endofunctors
$\bar\jmath_{0*}\bar\jmath{}_0^!,\ \jmath_{0*}\jmath_0^!\colon
D_{!*\bG_\bO\rtimes\BC^\times}(\Gr\times\bV)\to D_{!*\bG_\bO\rtimes\BC^\times}(\Gr\times\bV)$.
To this end we will need the algebra $\fH$ of
``graded differential operators'' on $\Hom(V_1,V_2)\oplus\Hom(V_1,V'_1)\oplus\Hom(V'_1,V_2)$
defined similarly to $\fD$ of~Section~\ref{loop}. It is equipped with the grading where
the generators from $\Hom(V_2,V_1),\Hom(V_1,V_2),\Hom(V'_1,V_2),\Hom(V_2,V'_1)$ have degree 1,
while the generators from $\Hom(V_1,V'_1),\Hom(V'_1,V_1)$ have degrees 2,0 respectively,
and $\deg\hbar=2$.
We denote by $\fH^\bullet$ this graded algebra viewed as a dg-algebra with trivial differential.
If we denote the similar dg-algebras of ``graded differential operators'' on
$\Hom(V_1,V_2),\ \Hom(V_1,V'_1),\ \Hom(V'_1,V_2)$ by
$\fD^\bullet_{12},\ \fD^\bullet_{11'},\ \fD^\bullet_{1'2}$, then
$\fH^\bullet=\fD^\bullet_{12}\otimes\fD^\bullet_{11'}\otimes\fD^\bullet_{1'2}$.
Since the canonical line bundle on an affine space carries a canonical flat connection,
the algebras $\fD^\blt_{ij}$ admit $(\GL(V_i)\x\GL(V_j))$-equivariant anti-involutions.
Thus we can identify $\fH^\bullet\iso\fD^\bullet_{12}\otimes(\fD^\bullet_{11'}\otimes\fD^\bullet_{1'2})^\opp$.

The algebra $\fH^\bullet$ has a cyclic holonomic left dg-module $\fQ^\bullet$ of ``delta-functions
along the subvariety cut out by the equation $A=A'A''$'', see~(\ref{right half}).
We also consider the $\GL(V_1)\x\GL(V_1')$-equivariant cyclic $\fD^\blt_{11'}$-module $\fD^\blt_{11'0}$ corresponding to the $\cD_{11'}:=\cD_{\Hom(V_1,V_1')}$-module  given by \[\cD_{11'}/(\cD_{11'}\otimes\det V_1\otimes\det{}\!^{-1} V_1' )(\det A'')=\Ind_{\CO_{\Hom(V_1,V_1')}}^{\cD_{11'}} ( \CO_{\Hom_{\le N-1}(V_1,V_1')} ).\]


We define the following endofunctors of $D^{G_{\bar0}}_\perf(\fD_{1,1}^\bullet)$:
\begin{multline*}
  D^{G_{\bar0}}_\perf(\fD_{1,1}^\bullet)=D^{\GL(V'_1)\times\GL(V_2)}_\perf(\fD^\bullet_{1'2})\ni M
  \mapsto p_*q^*M\\ :=
\big(\fQ^\bullet\otimes_{\fD^\bullet_{11'}\otimes\fD^\bullet_{1'2}}(\fD^\bullet_{11'}\otimes M)\big)^{\GL(V'_1)}\in
D^{\GL(V_1)\times\GL(V_2)}_\perf(\fD^\bullet_{12})=D^{G_{\bar0}}_\perf(\fD_{1,1}^\bullet)
\end{multline*}
\begin{multline*}
  D^{G_{\bar0}}_\perf(\fD_{1,1}^\bullet)=D^{\GL(V'_1)\times\GL(V_2)}_\perf(\fD^\bullet_{1'2})\ni M
  \mapsto p_{0*}q_0^*M\\ :=
\big(\fQ^\bullet\otimes_{\fD^\bullet_{11'}\otimes\fD^\bullet_{1'2}}(\fD^\bullet_{11'0}\otimes M)\big)^{\GL(V'_1)}\in
D^{\GL(V_1)\times\GL(V_2)}_\perf(\fD^\bullet_{12})=D^{G_{\bar0}}_\perf(\fD_{1,1}^\bullet)
\end{multline*}
Then similarly to~Proposition~\ref{corestriction} one proves
\begin{prop}
  \label{qrestriction}
  \textup{(a)}  There is an isomorphism of functors
  \[\bar\jmath_{0*}\bar\jmath{}_0^!\circ\Phi^{1,1}_\hbar\simeq
  \Phi^{1,1}_\hbar\circ p_*q^*\colon D^{G_{\bar0}}_\perf(\fD_{1,1}^\bullet)\to
  D_{!*\bG_\bO\rtimes\BC^\times}(\Gr\times\oV).\]

  \textup{(b)} There is an isomorphism of functors \[\jmath_{0*}\jmath_0^!\circ\Phi^{1,1}_\hbar\simeq
  \Phi^{1,1}_\hbar\circ p_{0*}q_0^*\colon D^{G_{\bar0}}_\perf(\fD_{1,1}^\bullet)\to
  D_{!*\bG_\bO\rtimes\BC^\times}(\Gr\times\oV).\]
\end{prop}

Now using~Proposition~\ref{qrestriction} in place of~Proposition~\ref{corestriction},
one checks the monoidal properties of $\Phi^{2,0}_\hbar$ and $\Phi^{0,2}_\hbar$ similarly
to the monoidal properties of $\Phi^{2,0}$ and $\Phi^{0,2}$ checked in~Sections~\ref{20 monoidal}
and~\ref{fourier} respectively. This completes the proof of~Theorem~\ref{main quantum}.
$\hfill \Box$

\subsection{$\fD$-modules and $D_{\GL(N-1,\bO)\rtimes\BC^\times}(\Gr)$}
\label{loop N-1}
Similarly to~Section~\ref{loop} we consider the dg-algebra (with trivial differential)
$\ol\fD{}^\bullet$ of ``graded differential operators'' on $\Hom(V_2,\ol{V}\!_1$). Then,
similarly to~Theorem~\ref{GL(N-1)}, one proves

\begin{thm}
  \label{loop GL(N-1)}
  There exists an equivalence of triangulated categories
  $\ol\Phi{}_\hbar\colon D^{\ol{G}_{\bar0}}_\perf(\ol\fD{}^\bullet)\iso
  D_{\ol\bG_\bO\rtimes\BC^\times}(\Gr)$ commuting with the left convolution action of the monoidal
  spherical Hecke category $\Perv_{\GL(N-1,\bO)\rtimes\BC^\times}(\Gr_{N-1})\cong\Rep(\GL_{N-1})$ and with
the right convolution action of the monoidal
  spherical Hecke category $\Perv_{\GL(N,\bO)\rtimes\BC^\times}(\Gr_N)\cong\Rep(\GL_N)$.
\end{thm}

\end{document}